\documentclass[11pt,reqno]{amsart}

\usepackage[utf8]{inputenc}
\usepackage{lmodern}
\usepackage{amsmath,amssymb,amsthm,latexsym,cite,cancel}
\usepackage[small]{caption}
\usepackage{graphicx,wasysym,overpic,tikz,color}
\usepackage{subfigure}
\usepackage{cite}
\usepackage[colorlinks=true,urlcolor=blue,
citecolor=red,linkcolor=blue,linktocpage,pdfpagelabels,
bookmarksnumbered,bookmarksopen]{hyperref}
\usepackage[english]{babel}
\usepackage{units}
\usepackage[utf8]{inputenc}
\usepackage{enumitem}
\usepackage[left=2.1cm,right=2.1cm,top=2.71cm,bottom=2.71cm]{geometry}
\usepackage[hyperpageref]{backref}
\usepackage{float}
\usepackage[T1]{fontenc}

\usepackage[colorinlistoftodos]{todonotes}
\usepackage[mathlines]{lineno} 

\usepackage{aliascnt}

\usepackage[capitalise,noabbrev,nameinlink]{cleveref}

\theoremstyle{plain}
\newtheorem{theorem}{Theorem}[section]
\crefname{theorem}{Theorem}{Theorems}

\newaliascnt{lemma}{theorem}
\newtheorem{lemma}[lemma]{Lemma}
\aliascntresetthe{lemma}
\crefname{lemma}{Lemma}{Lemmas}

\newaliascnt{proposition}{theorem}
\newtheorem{proposition}[proposition]{Proposition}
\aliascntresetthe{proposition}
\crefname{proposition}{Proposition}{Propositions}

\newaliascnt{corollary}{theorem}

\aliascntresetthe{corollary}
\crefname{corollary}{Corollary}{Corollaries}

\theoremstyle{definition}
\newaliascnt{definition}{theorem}
\newtheorem{definition}[definition]{Definition}
\aliascntresetthe{definition}
\crefname{definition}{Definition}{Definitions}

\theoremstyle{remark}
\newaliascnt{remark}{theorem}
\newtheorem{remark}[remark]{Remark}
\aliascntresetthe{remark}
\crefname{remark}{Remark}{Remarks}

\newaliascnt{example}{theorem}

\aliascntresetthe{example}
\crefname{example}{Example}{Examples}

\crefname{subsection}{Subsection}{Subsections}
\usepackage{physics2}
\usephysicsmodule{ab}
\usephysicsmodule{diagmat}
\usephysicsmodule{xmat}

\def\abs#1{\vab{#1}}
\def\norm#1{\Vab{#1}}

\newcommand{\R}{\mathbb{R}}
\newcommand{\N}{\mathbb{N}}

\newcommand{\e}{\mathrm{e}}

\newcommand{\Z}{\mathbb{Z}}

\usepackage{mathtools}
\DeclareMathOperator{\diver}{div}
\DeclareMathOperator{\supp}{supp}

\usepackage{soul}

\usepackage{braket}

\makeatletter
\def\namedlabel#1#2{\begingroup
    #2%
    \def\@currentlabel{#2}%
    \phantomsection\label{#1}\endgroup
}
\makeatother

\tikzstyle{nodo}=[circle,draw,fill,inner sep=0pt,minimum size=%
1.5mm]
\tikzstyle{bnodo}=[circle,draw,fill,inner sep=0pt,minimum size=%
2mm]

\numberwithin{equation}{section}

\title[Normalized solutions of  NLS equations with periodic potentials]
{Normalized solutions of $L^2$-supercritical NLS equations with periodic potentials}
\author[Z. He]{Zhentao He}

\address[Z. He]{\newline\indent
	School of Mathematics
	\newline\indent
	East China University of Science and Technology
	\newline\indent
	Shanghai 200237, PR China }
\email{\href{mailto:hezhentao2001@outlook.com}{hezhentao2001@outlook.com}}

\author[N. Ikoma]{Norihisa Ikoma}

\address[N. Ikoma]{\newline\indent
	Department of Mathematics
	\newline\indent
    Faculty of Science and Technology
    \newline\indent
	Keio University
	\newline\indent
	Yagami Campus, 3-14-1 Hiyoshi, Kohoku-ku, Yokohama, Kanagawa 223-8522, Japan}
\email{\href{mailto:ikoma@math.keio.ac.jp}{ikoma@math.keio.ac.jp}}

\author[C. Ji]{Chao Ji}

\address[C. Ji]{\newline\indent
	School of Mathematics
	\newline\indent
	East China University of Science and Technology
	\newline\indent
	Shanghai 200237, PR China }
\email{\href{mailto:jichao@ecust.edu.cn}{jichao@ecust.edu.cn}}



\begin{document}
\begin{abstract}
In this paper, we study normalized solutions  to the following $L^2$-supercritical nonlinear Schr\"{o}dinger equation with a periodic potential
\begin{equation*}
	\begin{dcases}
		-\Delta u +V (x)u + \lambda u=\chi_{{\Omega}}(x)f(u)\quad \text{in }\R^N,\\
		u>0 \quad \text{in} \ \R^N,\\
		\int_{\R^N}\abs{u}^2\, dx =\mu.
	\end{dcases}
\end{equation*}
Here $N \geq 1$, $\mu>0$ is prescribed, $\lambda \in \mathbb{R}$ is a Lagrange multiplier, $V\in C(\R^N)$ is $1$-periodic in $x_1,...,x_N$,  
$f \in C^1(\R)$ is a nonlinearity having $L^2$-supercritical growth at infinity, 
$\Omega\subset  \R^N$ is either a (nonempty) bounded open set with smooth boundary $\partial \Omega$ or the whole space \(\R^N\),
and $\chi_{{\Omega}}$ is the characteristic function of ${\Omega}$. 
In both cases, we prove the existence of normalized solutions of mountain pass type when $\mu>0$ is small
and $f$ behaves like a power function $|t|^{p-2}t$ with $2+4/N<p<2^*$ at infinity. 
Moreover, if $\Omega$ is bounded, $1 \leq N \leq 4$ and $f$ has $L^2$-supercritical growth near the origin, 
then the existence of normalized solutions is obtained for all $\mu>0$. 
On the other hand, when $\Omega=\R^N$, we prove the existence of solutions corresponding to local minimizers when $\mu>0$ is small. 
When $\Omega$ is bounded, the results are obtained via a monotonicity trick for mountain pass values and blow-up analysis with the Morse index estimates. 
In the case where $\Omega = \R^N$, we develop the concentration-compactness argument based on the Morse index for the existence of mountain pass type solutions. 
\end{abstract}
\maketitle
\tableofcontents

\section{Introduction}
In this paper, we study the existence of normalized solutions to the following $L^2$-supercritical nonlinear Schr\"{o}dinger (NLS for short) equation with a periodic potential
\begin{equation}\tag{$P_{\mu}$}\label{eqV}
    \begin{dcases}
    -\Delta u +V (x)u + \lambda u=\chi_{{\Omega}}(x)f(u)\quad \text{in }\R^N,\\
    u>0 \quad \text{in} \ \R^N,\\
    \int_{\R^N}\abs{u}^2\, dx =\mu,
\end{dcases}
\end{equation}
where $N \geq 1$, $\mu>0$ is prescribed, $\lambda \in \mathbb{R}$ is a Lagrange multiplier, $V\in C(\R^N)$ is $1$-periodic in $x_1,...,x_N$,  $f$ is a nonlinearity satisfying the assumptions specified below,  $\Omega\subset  \R^N$ is a (nonempty) bounded open set with smooth boundary $\partial \Omega$ or the whole space \(\R^N\), and $\chi_{{\Omega}}$ is the characteristic function of ${\Omega}$.

Formally, functions satisfying the differential equation in \eqref{eqV} correspond to standing waves of the following time-dependent NLS equation
\begin{equation}\label{time-dep}
i \dfrac{\partial \Psi}{\partial t}(t,x) = - \Delta_x \Psi(t,x) + V(x) \Psi(t,x) - \chi_{\Omega}(x) g(\abs{\Psi(t,x)})\Psi(t,x),  \quad (t,x) \in (0,T) \times \R^N, 
\end{equation}
via the ansatz $\Psi(t,x) = e^{i\lambda t} u(x)$, where $f(u)=g(\abs{u})u$. 
Such equations arise naturally in models of Bose--Einstein condensates  subject to periodic external potentials 
and of nonlinear photonic crystals with spatially periodic material parameters; see, e.g., Pankov \cite{Pa05}.  
The associated evolution conserves the $L^2$-mass, namely,
\[
\int_{\mathbb{R}^{N}}|\Psi(t,x)|^{2}\,dx
=
\int_{\mathbb{R}^{N}}|\Psi(0,x)|^{2}\,dx
\quad \text{for every $t \in [0,T)$}. 
\]
Accordingly, the constraint $\Vab{u}_{2}^{2}=\mu$ in \eqref{eqV} prescribes the mass of the standing waves. 
Depending on the physical setting, this conserved quantity may be interpreted as the optical power in nonlinear optics or as the total number of particles in a Bose--Einstein condensate.
We also point out that under some additional assumptions on $f$ (or $g$), the energy corresponding to \eqref{time-dep} is conserved:
\[
E_V(\Psi(t,x)) = E_V(\Psi(0,x)) \quad \text{for each $t \in [0,T)$},
\]
where 
\[
E_V(u) \coloneq \frac{1}{2} \int_{\R^N} \vab{\nabla u}^2 + V(x) \vab{u}^2 \, dx - \int_{\R^N} \chi_\Omega(x) F(\vab{u}) \, dx, \quad 
F(s) \coloneq \int_0^s f(t) \, dt.
\]
See, for instance, \cite{Ca03,CaHa98}. 
In view of the Lagrange multiplier method, with suitable assumptions on $f$, 
solving \eqref{eqV} is equivalent to finding a critical point $u$ of $E_V|_{S_\mu}$ with $u > 0$ in $\R^N$, where 
\[
S_\mu \coloneq \Set{ u \in H^1(\R^N):\int_{\R^N}\abs{u}^2\,dx=\mu }. 
\]
Thus, \eqref{eqV} is related to the two natural conserved quantities for \eqref{time-dep}, and 
is an interesting problem from the mathematical and physical point of view. 

\subsection{Literature review}
When \(V\equiv0\) and \(\Omega=\mathbb R^N\), there is a vast literature on the existence, multiplicity, and stability of normalized solutions to problem \eqref{eqV}. 
To describe the different variational regimes, recall that, for the power nonlinearity $f(u)=|u|^{p-2}u$,
the \(L^2\)-critical exponent is $2+\frac{4}{N}$ and the problem is referred to as \(L^2\)-subcritical, \(L^2\)-critical, or \(L^2\)-supercritical according as
\[
2<p< 2+\frac{4}{N},\quad
p=2+\frac{4}{N},\quad\text{or}\quad
2+\frac{4}{N} <p<2^*\coloneq \frac{2N}{(N-2)_+},
\]
respectively.

The \(L^2\)-subcritical regime was considered in \cite{St80,St82,CL} and it was revealed that 
$E_0$ is bounded below on $S_\mu$ and a solution to \eqref{eqV} is obtained by minimizing $E_0$ on $S_\mu$. 
In particular, Cazenave and Lions \cite{CL} proved the orbital stability of standing waves of \eqref{time-dep} 
by developing the argument of the concentration compactness lemma via the global minimizing problem.  
On the other hand,  the variational structure changes substantially in the \(L^2\)-supercritical regime: 
$E_0$ is no longer bounded from below on \(S_\mu\), and hence the global minimization problem does not work. 
The first existence result in this regime is due to Jeanjean \cite{jeanjean1} via the mountain pass theorem for $E_0|_{S_{{\rm rad},\mu}}$, 
where $S_{{\rm rad} ,\mu} \coloneq S_\mu \cap H^1_{\rm rad} (\R^N)$ and 
$H^1_{\rm rad} (\R^N)$ is the space of all radial functions in $H^1(\R^N)$. 
In this case, one of the difficulties is to show the boundedness of Palais--Smale sequences of $E_0|_{S_{{\rm rad} ,\mu}}$. 
To overcome this point, Jeanjean introduced the augmented functional
\[
\widetilde E_0(u,s)
\coloneq E_0 \ab( e^{Ns/2} u \ab( e^s \cdot ) )
= 
\frac{e^{2s}}{2}
\int_{\mathbb R^N}|\nabla u|^2\,dx
-
e^{-Ns}
\int_{\mathbb R^N}
F\left(e^{\frac{Ns}{2}}u(x)\right)\,dx 
\colon S_{{\rm rad} ,\mu}\times \R \to \R,
\]
and $\widetilde{E}_0$ is exploited to generate a special Palais--Smale sequence $\{u_n\} \subset S_{{\rm rad} ,\mu}$ of $E_0$ satisfying 
an additional property $P(u_n) \to 0$, where 
\[
P(u) \coloneq \partial_s \widetilde{E}_0(u,0) = \int_{\R^N} \vab{\nabla u}^2 \, dx - \frac{N}{2}\int_{\R^N} f(u) u - 2 F (u) \, dx .
\]
This additional property is useful to obtain the boundedness of Palais--Smale sequences 
as well as to extract a strongly convergent subsequence by combining the compact embedding $H^1_{\rm rad} (\R^N) \subset L^q(\R^N)$ for $2<q<2^*$.   
A different approach using the scaling $e^{Ns/2} u(e^s \cdot)$ was developed 
in Bartsch and Soave \cite{BaSo17}. 
On the other hand, in Bieganowski and Mederski \cite{BiMe21}, the existence of solutions was shown through a minimizing problem
\[
\inf \Set{ E_0(u) :  u \in H^1(\R^N), \Vab{u}_2^2 \leq \mu, P(u) = 0 }
\]
and the profile decomposition when $\mu$ is small. For further results concerning problem \eqref{eqV} with \(V\equiv0\) and \(\Omega=\mathbb R^N\), we refer to
\cite{IkTa19, JeLe22, JeLu20, So20a, So20b} and the references therein.

Next, we move to the problem with a nonconstant potential $V$, $f$ having an $L^2$-supercritical growth at infinity and $\Omega = \R^N$. 
In this case, the boundedness and convergence of Palais--Smale sequences are issues. 
As in the case $V\equiv 0$, under suitable assumptions on $V$, 
the functional $\widetilde{E}_V(u,s) = E_V( e^{Ns/2} u(e^s \cdot) )$ is of class $C^1$ on $H^1(\R^N) \times \R$ and 
\[
\partial_s \widetilde{E}_V(u,0) 
= \Vab{\nabla u}_2^2 - \frac{1}{2} \int_{\R^N} \ab( x \cdot \nabla V(x) ) u^2 \, dx 
- \frac{N}{2} \int_{\R^N} f(u) u - 2F(u) \, dx.
\]
If we can control the term $\int_{\R^N} (x \cdot \nabla V(x)) u^2 \, dx$, then 
it is still possible to prove the boundedness of Palais--Smale sequences with $\partial_s \widetilde{E}_V(u_n,0) \to 0$. 
Indeed, in \cite{BMRV21,LiZh24,MRV22,VeYu26}, this idea is used to obtain the existence of solutions to \eqref{eqV} 
when $\Omega = \R^N$ and $V(x) \to 0$ as $\vab{x} \to +\infty$ or $V(x) \to +\infty$ as $\vab{x} \to + \infty$.

Another approach to obtain a bounded Palais--Smale sequence is a monotonicity trick due to \cite{Str88,Je99} 
for unconstrained functionals. For functionals constrained on $S_\mu$, 
this technique together with information of the Morse index in the spirit of Fang and Ghoussoub \cite{FaGh92} 
was developed in Borthwick et al. \cite{BCJS}. 
Moreover, in \cite{BCJS23}, the result of \cite{BCJS} was applied to get an $L^2$-normalized solution to 
the following problem on a  noncompact metric graph $\mathcal{G}$:
\begin{equation}\label{eqG}
	-u''+\lambda u = \chi_{\mathcal{K}}(x) \vab{u}^{p-2} u \quad \text{on every edge $e$ of $\mathcal{G}$}, \quad 
	\int_{\mathcal{G}} u^2 \, dx = \mu, \quad 
	\sum_{e \ni v} u_e'(v) = 0,
\end{equation}
where $p>6$, $\mathcal{K}$ is a compact core of $\mathcal{G}$ and $u_e'(0)$ represents either $u'(0)$ or $-u'(\ell_\e)$ for $e = [0,\ell_\e]$. 
Indeed, it was proved in \cite{BCJS23} that for any $\mu>0$, \eqref{eqG} admits a positive solution $u$ with $\lambda > 0$ 
via the monotonicity trick with the Morse index information and a blow-up analysis. 
See Carrillo et al. \cite{CGJT} for further results on \eqref{eqG}, Carrillo et al. \cite{CDCGJT} for blow-up analysis of equations on metric graphs and Dovetta, Jeanjean and Serra \cite{DJS} for the existence results when $\mathcal{K} = \mathcal{G}$ (here the compactness of $\mathcal{K}$ is dropped). 
There are other results using the result of \cite{BCJS} and the blow-up analysis. 
Bartsch, Qi and Zou \cite{BQZ24} studied \eqref{eqV} with $\Omega = \R^N$, $f(u) = \vab{u}^{p-2} u + \vab{u}^{q-2} u$, 
$2<q<2+4/N<p<2^*$ and a general potential $V$ by imposing an integrability condition and estimates on $V$ and $x \cdot \nabla V(x)$. 
In particular, they found two solutions $u$ and $v$ with $E_V(u) < 0 < E_V(v)$ for a small $\mu>0$. 
On the other hand, for a bounded smooth domain $D$ and for any small $\mu>0$, 
Chang, R\u adulescu and Zhang \cite{CRZ} obtained the existence of solutions to 
\[
-\Delta u + \lambda u = f(u) \quad \text{in} \ D, \quad \frac{\partial u}{\partial \nu} = 0 \quad \text{on} \ \partial D, \quad \Vab{u}_2^2 = \mu. 
\]
Our assumptions below on $f$ are similar to \cite{CRZ}. 
We also refer to Carrillo and Jeanjean \cite{CaJe} for the problem with a radial potential.

Compared with the above-mentioned cases, 
the literature on normalized solutions for NLS equations with nonconstant periodic potentials remains relatively limited. 
Alves and Ji \cite{AlJi22} proved the existence of $L^2$-normalized solutions with periodic potentials in the $L^2$-subcritical regime 
via the minimization method. 
On the other hand, in \cite{AW}, 
Ackermann and Weth studied the existence and stability properties of multibump solutions with prescribed $L^2$-norm. In particular, in \cite[Corollary 1.6]{AW}, by assuming that $V \in C^2(\R^N)$ is $1$-periodic in all coordinates, positive, and has a nondegenerate critical point at some $x_0 \in \R^N$, the authors investigated the following problem
\begin{equation}\label{eqpav}
	\tag{$P_{\mu,\varepsilon}$}
	-\varepsilon^2 \Delta u + V(x)u + \lambda u =  |u|^{p-2}u, \quad u \in H^1(\mathbb{R}^N), \quad \Vab{u}_2^2 = \mu, 
\end{equation}
where $p\in(2,2^*)\setminus\{2+\frac{4}{N}\}$. 
They proved that for every $\mu>0$ there exist $n_\mu \in \mathbb{N}^+$ and a sequence $\varepsilon_n \to 0$ 
such that for every $n \geq n_\mu$ problem (\hyperref[eqpav]{$P_{\mu,\varepsilon_n}$}) has infinitely many geometrically distinct positive solutions 
with some specified properties. However, this problem is different from \eqref{eqV}. Indeed, 
by introducing $v(x) \coloneq \varepsilon^{ - \frac{2}{p-2} } u(x)$, \eqref{eqpav} is equivalent to 
\[
-\Delta v + \varepsilon^{-2} \ab( V + \lambda ) v = \vab{v}^{p-2} v, \quad v \in H^1(\R^N), \quad 
\Vab{v}_2^2 = \varepsilon^{ - \frac{4}{p-2} } \mu,
\]
which is a different setting from \eqref{eqV}.

This paper is motivated by \cite{AlJi22,BCJS23,CGJT,DJS,BQZ24,CRZ}, 
and our aim is to establish a counterpart to the above results in the $L^2$-supercritical regime in the presence of periodic potentials.

\subsection{Main results}

We first deal with the case where $\Omega$ is bounded. 
Since we are interested in positive solutions to \eqref{eqV}, 
we impose conditions on $f$ in $[0,\infty)$: 
\begin{enumerate}[label=(A\arabic*),ref=A\arabic*]
\item \label{A1} $f \in C([0,\infty))$ and $\displaystyle \lim_{t \to 0^+} \frac{f(t)}{t} = 0$; 

\item \label{A2} 
$f \in C^1([0,\infty))$ and there exist constants  $C',q_1,q_2 > 0$ and $\alpha \in (0, 1]$ such that 
$2+\alpha \leq q_1\leq q_2<+\infty$ for $N=1,2$ and $2+\alpha \leq q_1\leq q_2\leq 2^*$ for $N\geq 3$, and
\[
|f'(t) - f'(s)| \le C' |t - s|^\alpha \ab( |t|^{q_1-2-\alpha} + |s|^{q_1-2-\alpha} + |t|^{q_2-2-\alpha} + |s|^{q_2-2-\alpha} ) 
\quad \text{for all $t,s \in [0,\infty)$};
\]
in what follows, we use the convention $2^*=\infty$ when $N=1,2$; 

\item \label{A3} there exist constants $p \in \left(2+\frac{4}{N},2^*\right)$ and $a_0>0$ such that 
\[
\lim_{t\to +\infty}\frac{f(t)}{t^{p-1}}=a_0;
\]

\item \label{A4}
for the constants $p$ and $a_0>0$ in \eqref{A3}, 
\[
\liminf_{t\to+\infty}\frac{f'(t)}{t^{p-2}} \geq a_0(p-1);
\]
\end{enumerate}
Throughout this paper, by an odd extension, we may assume that 
\begin{equation}\label{eqodd}
    f(t)=-f(-t) \quad \text{for each $t \in \R$}.
\end{equation}
A simple example of $f$ satisfying \eqref{A1}--\eqref{A4} is 
\begin{equation}\label{exam-f}
	f(t) = \sum_{i=1}^k a_i t^{p_i-1} + t^{p_{k+1}-1}, \quad 2 < p_1 < p_2 < \dots < p_k < p_{k+1} < 2^*,\quad p_{k+1}>2+\frac{4}{N}, \quad a_1,\dots, a_k \in \R. 
\end{equation}
Notice that \eqref{A1}--\eqref{A4} are almost the same as the conditions in \cite{CRZ}.

Our first result is the existence of solutions to \eqref{eqV} when $\Omega$ is bounded and $\mu>0$ is small:

\begin{theorem}\label{th1}
     Let $N \geq 1$, $V\in C(\R^N)$ be $1$-periodic in $x_1,...,x_N$, 
     $\Omega\subset  \R^N$ be a (nonempty) bounded open set with smooth boundary $\partial \Omega$ and \eqref{A1}--\eqref{A4} hold. 
     Then, there exists $\mu_0>0$ such that, for every $\mu\in (0,\mu_0)$, problem \eqref{eqV} has a solution $(u,\lambda)\in H^1(\R^N) \times (-\sigma_0,+\infty)$, 
     where 
     \begin{equation}\label{bot-spec}
     \sigma_0 \coloneq 
     \inf \sigma (-\Delta+V(x))=\inf_{u \in H^1(\R^N)\backslash\{0\}}\frac{\int_{\R^N}\left(\abs{\nabla u}^2+ V(x)\abs{u}^2\right)\, dx}{\int_{\R^N}\abs{u}^2\, dx}.
     \end{equation}
\end{theorem}

Next, we deal with the existence of solutions for any $\mu>0$. For this purpose, we assume an additional condition on $f$: 
\begin{enumerate}[resume,label=(A\arabic*),ref=A\arabic*,resume]
	\item \label{A5} 
	there exist constants $q_3,C''>0$ such that $2+\frac{4}{N}<q_3<2^*$ and
	\[
	\abs{F(t)}\leq C'' t^{q_3} \quad \text{for any $t \in [0,1]$, where $F(t) \coloneq \int_0^t f(s) \,d s$}.
	\]
\end{enumerate}

\begin{remark}
	If $q_1 > 2+4/N$ holds in \eqref{A2}, then \eqref{A5} is derived from \eqref{A1}--\eqref{A3} by setting $q_3=q_1$. 
	Therefore, the function $f$ in \eqref{exam-f} enjoys \eqref{A1}--\eqref{A5} provided $p_1$ in \eqref{exam-f} satisfies $p_1 > 2+4/N$. 
\end{remark}

Our second result for bounded $\Omega$ reads as follows: 

\begin{theorem}\label{th2}
	Let $1 \leq N \leq 4$, $V\in C(\R^N)$ be $1$-periodic in $x_1,...,x_N$, 
	$\Omega\subset  \R^N$ be a (nonempty) bounded open set with smooth boundary $\partial \Omega$ and \eqref{A1}--\eqref{A5} hold. 
	Then, for every $\mu>0$, problem \eqref{eqV} has a solution $(u,\lambda)\in H^1(\R^N) \times (-\sigma_0,+\infty)$ where $\sigma_0$ is defined in \eqref{bot-spec}.
\end{theorem}

On the restriction on $N$ in \cref{th2}, we refer to the discussion in the next subsection.

Next, we turn to the case $\Omega =\R^N$ and aim to find a local minimizer and mountain pass type solution. 
This case is not a direct consequence of the localized problem above, since the compactness produced by the bounded support of the nonlinearity is no longer available. Moreover, the local minimizer and the mountain pass solution are obtained under two different sets of assumptions and by different variational arguments.
For this purpose, we prepare the following conditions on $f$: 
\begin{enumerate}[label=(A\arabic*), ref=A\arabic*, resume]
	\item \label{nA6} 
	there exists $q_4 \in (2,2^*)$ such that $f(t) / t^{q_4-1} \to 0$ as $t \to +\infty$; 
	
	\item \label{nA7}
	for each $t \in (0,+\infty)$ and $\theta \in (1,+\infty)$, $F(\theta t) - \theta^2 F(t) > 0$; 

	\item \label{nA8}
	there exist $2 < q_5 < 2 + 4/N < q_6  < 2^*$ and $C_1,C_2>0$ such that 
	\[
	C_1 t^{q_5} - C_2 t^{q_6} \leq F(t)
	\quad \text{for any $t \in [0,+\infty)$};
	\]
	
	\item \label{nA9}
	there exists $C_3>0$ such that $f(t) t \leq C_3 t^{2+4/N}$ for any $t  \in [0,1]$;
	
	\item  \label{A6} 
	the function $(0,+\infty) \ni t \mapsto \frac{f(t)}{t}$ is strictly increasing.
\end{enumerate}

It is easily seen that $f$ in \eqref{exam-f} satisfies \eqref{A1}--\eqref{A4}, \eqref{nA6}--\eqref{nA8} and \eqref{A6} provided $a_1,\dots, a_k > 0$ and $2<p_1< 2+4/N$. 
On the other hand, \eqref{nA9} is a one-sided bound and $f$ in \eqref{exam-f} satisfies \eqref{nA9} 
if $a_i \leq 0$ holds for any $p_i < 2+4/N$. 
Condition \eqref{nA7} is taken from Yang, Qi and Zou \cite[($f5$)]{YQZ22} in which the $L^2$-subcritical nonlinearity was considered. 
We emphasize that \eqref{nA6}--\eqref{nA8} are used for the local minimization problem in \cref{T:whole} \ref{T:whole-i}, whereas the mountain pass solutions in \cref{T:whole} \ref{T:whole-ii} and \ref{T:whole-iii} are obtained under \eqref{A1}--\eqref{A4} together with \eqref{nA9} (respectively \eqref{A6}).

\begin{theorem}\label{T:whole}
	Let $N \geq 1$, $V \in C(\R^N)$ be $1$-periodic in $x_1,\dots,x_N$ and $\Omega = \R^N$. 
	\begin{enumerate}[label={\rm (\roman*)}]
		\item \label{T:whole-i}
		Suppose \eqref{A1} and \eqref{nA6}--\eqref{nA8}. Then there exists $\bar{\mu}_1 > 0$  such that 
		for each $\mu \in (0,\bar{\mu}_1)$, the minimization problem
        \begin{equation}\label{e:loc-min}
			m_{V,\mu} \coloneq \inf_{\mathcal{A}_\mu} E_V(u), \quad 
			\mathcal{A}_\mu \coloneq \Set{ u \in S_\mu \colon \int_{\R^N} \vab{\nabla u}^2 + (V(x)-\sigma_0) u^2 dx < 8 }
		\end{equation}
		admits a minimizer $u$ with $u>0$ in $\R^N$ and $E_V(u) = m_{V,\mu} < \sigma_0\mu/2$ and $m_{V,\mu} \to 0$ as $\mu \to 0^+$ where $\sigma_0$ is defined in \eqref{bot-spec}; 
		hence $(u,\lambda)$ is a solution to problem \eqref{eqV} with some $\lambda \in \R$. 
		
		\item \label{T:whole-ii}
		Assume \eqref{A1}--\eqref{A4} and \eqref{nA9}. 
		Then there exists $\bar{\mu}_2 >0$ such that for each $\mu \in (0,\bar{\mu}_2)$, 
		problem \eqref{eqV} has a solution $(u_\mu,\lambda) \in H^1(\R^N) \times (-\sigma_0,+\infty)$ with $E_V(u_\mu) > \sigma_0 \mu/2$ 
		where $\sigma_0$ is defined in \eqref{bot-spec}.
		Moreover, $E_V(u_\mu) \to +\infty$ as $\mu \to 0^+$.

		\item \label{T:whole-iii}
		Assume \eqref{A1}--\eqref{A4} and \eqref{A6}. Then there exists $\bar{\mu}_3 > 0$ such that 
		for any $\mu \in (0,\bar{\mu}_3)$, 
		problem \eqref{eqV} admits a solution $(u_\mu,\lambda) \in H^1(\R^N) \times (-\sigma_0,+\infty)$ with $E_V(u_\mu) > \sigma_0 \mu/2$.
		In addition, $E_V(u_\mu) \to +\infty$ as $\mu \to 0^+$.

	\end{enumerate}
\end{theorem}

To the best of the authors' knowledge, \cref{th1,th2,T:whole} provide the first existence results of $L^2$-normalized solutions 
in the $L^2$-supercritical regime with general periodic potentials. 
It is worth pointing out that our results are new even for the simplest model $f(t) = |t|^{p-2} t$ with $2+4/N < p < 2^*$ 
for both cases: when $\Omega$ is bounded and when $\Omega = \R^N$. The existence results in \cref{th1,th2,T:whole} can be regarded as a higher-dimensional analogue of \cite{BCJS23,DJS}. 
On the other hand, by \cref{T:whole} \ref{T:whole-i} and \ref{T:whole-iii}, 
if \(f\) satisfies \eqref{A1}--\eqref{A4}, \eqref{nA6}--\eqref{nA8}, and \eqref{A6}, 
then problem \eqref{eqV} admits at least two solutions, one with energy below $\sigma_0\mu/2$ and the other with energy above $\sigma_0\mu/2$. 
Therefore, \cref{T:whole} \ref{T:whole-i} and \ref{T:whole-iii} are a counterpart of the multiplicity result in \cite[Theorems 1.6 and 1.8]{BQZ24} 
for periodic potentials. 
Finally, we mention that our proof of \cref{th1} can be extended to handle  more general bounded potentials, 
and this result is related to \cite{CaJe}. For the details, we refer to \cref{sec:generalV}.

\subsection{Difficulties and ideas of the proof}
To prove \cref{th1,th2,T:whole}, we first notice that without loss of generality, we may suppose 
\begin{equation}\label{inf-spect}
	\sigma_0 = \inf \sigma (-\Delta+V(x))=\inf_{u \in H^1(\R^N)\backslash\{0\}}\frac{\int_{\R^N}\left(\abs{\nabla u}^2+ V(x)\abs{u}^2\right)\, dx}{\int_{\R^N}\abs{u}^2\, dx}=0.
\end{equation}
Indeed, if $\sigma_0\neq 0$, then we can replace $-\Delta+V(x)$ by $-\Delta+V(x)-\sigma_0$. 
In the sequel, $\sigma_0 = 0$ is always assumed.

Among \cref{th1,th2,T:whole}, \cref{T:whole} \ref{T:whole-i} deals with local minimizers of $E_V|_{S_\mu}$, 
and others treat critical points of mountain pass type for $E_{V}|_{S_\mu}$. 
A basic idea of the proof of \cref{T:whole} \ref{T:whole-i} is similar to the subcritical case \cite{YQZ22,AlJi22}, 
and the strict subadditivity with $m_{V,\mu} < 0$ plays a role in ruling out dichotomy and vanishing. 
To obtain $m_{V,\mu} < 0$ for small $\mu>0$, test functions in Heinz, K\"upper and Stuart \cite{HKS92} are used, 
and allow us to remove the smallness condition on $V$ in \cite[Theorem 1.1]{AlJi22}.

On the other hand, to prove the existence of mountain pass type critical points, 
proving the boundedness of Palais--Smale sequences is difficult in general. 
Since the term $\int_{\R^N} (x \cdot \nabla V) u^2 \, dx$ is not small for periodic potentials, 
the use of $\widetilde{E}_V$ may not be helpful. 
Therefore, we adapt the scheme developed in \cite{BCJS23,CJS24,CRZ}, that is, 
the monotonicity trick with the estimate of Morse index in \cite{BCJS} and the blow-up analysis in Esposito and Petralla \cite{EP} and Pierotti and Verzini \cite{PV} (cf. \cite{CDCGJT}). 
For the case where $\Omega$ is bounded (\cref{th1,th2}), the implementation of this scheme requires three additional ingredients: 
\begin{enumerate}[label=(\Alph*), ref=\Alph*]
	\item \label{A}
	we rewrite $-\Delta+V$ in divergence form 
	and derive a Gagliardo--Nirenberg type inequality (GN type inequality for short) adapted to the periodic potential in \cref{subsec:-rewr,subsec:GNineq};
	
	\item \label{B}
	we establish a Liouville type result for \eqref{eqV} with $\lambda =0$, and it is fundamental to recover the compactness in $L^2(\R^N)$ 
	for bounded Palais--Smale sequences obtained via the monotonicity trick \cite{BCJS}; 
	this is discussed in \cref{subec:Liouville} and \cref{T:Liouv-musmall};
	
	\item \label{C}
	 the localized nonlinearity leads, after blow-up, to limit equations with nonlinearities supported on half-spaces, for which suitable nonexistence results are proved; 
	See \cref{subsec:half}. 
\end{enumerate}

With regard to \eqref{A}, the GN type inequality involving $V$ is a useful tool. 
In particular, in the setting of \cref{th2}, a mountain pass structure of $E_V|_{S_\mu}$ is found 
for every $\mu>0$ without a smallness condition on $V$, which is often imposed for problems with potentials. 
For this purpose, the GN type inequality plays a role. 
In addition, the GN type inequality is also useful in the proof of \cref{lemb1}, which is a key to prove \cref{T:whole} \ref{T:whole-ii}.

Concerning \eqref{B}, for \eqref{eqV} with $\lambda =0$, the effect of $V$ should be taken into account. 
Compared to \cite{BCJS23} in which ODE techniques are available and no $V$ appears, 
our argument is more involved. 
Here the rewriting of the operator $-\Delta + V$ in \eqref{A} is helpful for $1 \leq N \leq 4$ (see \cref{lemliou}), 
and the restriction on $N$ in \cref{th2} comes from this result. 
On the other hand, when $\mu>0$ is small, we may prove the nonexistence result in \cref{T:Liouv-musmall} under suitable conditions for $N \geq 1$.

As for \eqref{C}, due to the effect of $\chi_\Omega$ for the nonlinearity, 
the following limit equation appears after blow-up: 
\[
-\Delta u + \lambda u = \chi_H u^{p-1} \quad \text{in} \ \R^N,
\]
where $u$ is stable outside a compact set (see \cref{defstable}) and $H$ is a half-space. 
When $\lambda = 0$, it turns out that the proof of Chen and Li \cite{ChLi91} and Quittner and Souplet \cite[\S8.4]{QuSo19} 
for the result by Gidas and Spruck \cite{GS} works. 
On the other hand, when $\lambda > 0$, we combine ideas from Farina \cite{Fa} and Esteban and Lions \cite{EsLi82}. 
Since we deal with the case $N \geq 2$, the situation is more complicated. 
Here we mention that due to the result for $\lambda >0$, 
we may simplify the behavior of blowing up sequences in \cref{lem62}. 
(cf. \cite[Proposition 4.2]{BCJS23}, \cite[Theorem 4.2]{CJS24} and \cite[Lemma 4.3]{CRZ}.)

Now we turn to the case $\Omega = \R^N$ (\cref{T:whole} \ref{T:whole-ii} and \ref{T:whole-iii}). 
In this case, to recover the compactness of bounded Palais--Smale sequences, 
we need a concentration--compactness type argument. 
In order to avoid dichotomy, 
we exploit estimates on an approximate Morse index $\tilde{m}_{\tau,0} (u)$ which is considered in \cite{BCJS} (see \cref{d:Mor-apMor}), the Morse index $m_{\lambda,\tau} (u)$, the Lagrange multiplier and the energy $E_{V,\tau} (u)$ (see \eqref{eqdefevtau}), 
and this is different from \cite{BQZ24,DJS,MRV22}. 
In the setting of \cref{T:whole} \ref{T:whole-ii}, a key is \cref{lemb1}, in which we estimate the Lagrange multiplier $\lambda$ 
and the approximate Morse index $\tilde{m}_{\tau,0} (u)$ for small $\mu>0$ via the blow-up argument. 
On the other hand, for \cref{T:whole} \ref{T:whole-iii}, 
we derive estimates on the energy $E_{V,\tau} (u)$ and the Lagrange multiplier depending on the approximate Morse index 
for small $\mu>0$. 
These estimates together with the fact that the Morse index of solutions to \eqref{eqV} is at least one 
lead to the compactness.

This paper is organized as follows. In \cref{sec:Pre}, 
we rewrite $-\Delta + V$ into the divergence form and prove a Gagliardo--Nirenberg type inequality with $V$ 
and some Liouville type results. 
\cref{sec:mp} is devoted to finding the mountain pass structure of $E_{V,\tau}$ 
and proving the nonexistence result with $\lambda = 0$ and the small $L^2$-norm. 
In \cref{sec:comp-bPS}, the behavior of bounded Palais--Smale sequences is analyzed and in \cref{sec: Ex-MP-a.e.}, we prove the existence of critical points corresponding to the mountain pass value for almost all $\tau \in [1/2,1]$. 
\cref{sec:blow-up} treats the behavior of blow-up sequences and 
\cref{sec:pf-main} provides the proof of \cref{th1,th2}. 
In \cref{sec:the-whole-sp}, we prove \cref{T:whole}: part \ref{T:whole-i} by the local minimization and concentration--compactness argument, and parts \ref{T:whole-ii} and \ref{T:whole-iii} by combining the mountain pass construction with concentration--compactness and the estimates on the approximate Morse index, the Morse index, the Lagrange multiplier and 
the energy. 
In \cref{app:lemnelambda+0}, \cref{lemnelambda+0} is proved.  In \cref{sec:generalV}, we extend \cref{th1} to general potentials $V \in C(\R^N) \cap L^\infty (\R^N)$.

In the following, we denote by $C$ and $\tilde{C}$ positive constants that may change from line to line.

\section{Preliminaries}
\label{sec:Pre}

In the following sections, we always assume that $V \in C(\R^N)$ is $1$-periodic in $x_1,\dots,x_N$ and satisfies \eqref{inf-spect}. 

\subsection{Another expression of $-\Delta + V$ without lower order terms}
\label{subsec:-rewr}

In this subsection, we rewrite $-\Delta u + V(x) u$. This will be useful to derive a Gagliardo--Nirenberg type inequality 
involving $\int_{\R^N} \abs{\nabla u}^2 + V(x) u^2 \, dx$ as well as a Liouville type result (see \cref{lemgn,lemliou} below). 
To this end, we first recall \eqref{inf-spect}: $0 = \sigma_0 = \inf \sigma (-\Delta + V)$. 
Let \(Q=(0,1)^N\), and consider the periodic eigenvalue problem
for \(-\Delta+V\) on \(Q\). Denote its lowest eigenvalue by
\(\lambda_0\). By \cite[Theorem 3.1.2(i)]{Eastham} when \(N=1\), and by
\cite[Theorem 6.9.1]{Eastham} when \(N\geq 2\),
the eigenvalue \(\lambda_0\) admits a positive
\(1\)-periodic eigenfunction \(\psi_0\).
Moreover, by \cite[Theorem 2.4.1, Eq.~(3.2.1), and Theorem 5.3.2]{Eastham} 
when \(N=1\), and by \cite[Theorems 6.5.1 and 6.9.2]{Eastham} when 
\(N\geq 2\), the lowest eigenvalue \(\lambda_0\) coincides with $\sigma_0$. 
Consequently, 
\[
\lambda_0=\sigma_0=0,
\]
and the periodic extension of \(\psi_0\) is a positive \(1\)-periodic solution to
\[
-\Delta\psi_0+V(x)\psi_0=0
\qquad\text{in }\mathbb{R}^N.
\]
Since $V \in C(\R^N)$, elliptic regularity implies $\psi_0 \in W^{2,r}_{\rm loc} (\R^N)\ \cap C^1(\R^N)$ for any $r<\infty$. 
Using $\psi_0$, let us consider the following uniformly elliptic operator:
\[
L_0 v \coloneq - \diver \ab( \psi_0^2 \nabla v ).
\]
For any $u \in W^{2,1}_{\text{loc}}(\R^N)\cap W^{1,\infty}_{\text{loc}}(\R^N)$, 
set $v \coloneq u/\psi_0$. 
It is easily seen from  \cite[Proposition 9.4]{Br} that $v \in W^{2,1}_{\text{loc}}(\R^N)\cap W^{1,\infty}_{\text{loc}}(\R^N)$ and 
\begin{equation}\label{eqschpsiv}
    \begin{aligned}
(-\Delta + V) u &= -\Delta(\psi_0 v) + V \psi_0 v \\
&= -(\psi_0 \Delta v + 2 \nabla \psi_0 \cdot \nabla v + v \Delta \psi_0) + V \psi_0 v \\
&=\psi_0 (-\Delta v) - 2 \nabla \psi_0 \cdot \nabla v\\
&=-\frac{1}{\psi_0} \diver (\psi^2_0 \nabla v) = \frac{1}{\psi_0} L_0v. 
\end{aligned}
\end{equation}
Therefore, for any $u \in C^\infty_0(\R^N)$, we derive 
\begin{equation}\label{equpsi}
  \int_{\R^N}\ab(\abs{\nabla u}^2+ V(x)\abs{u}^2) dx =\int_{\R^N}\psi^2_0 \abs{\nabla (u/\psi_0)}^2 dx.
\end{equation}
By density of $C_0^\infty(\mathbb{R}^N)$ in $H^1(\mathbb{R}^N)$, \eqref{equpsi} holds for any $u \in H^1(\R^N)$. 

\subsection{A Gagliardo--Nirenberg type inequality}
\label{subsec:GNineq}

In this subsection, we derive a Gagliardo--Nirenberg type inequality involving 
\[
\abs{u}_V:=\ab[\int_{\R^N}\ab(\abs{\nabla u}^2+ V(x)\abs{u}^2) dx]^\frac{1}{2},
\]
which plays a role in finding a (uniform) mountain pass structure of $E_{V,\tau}|_{S_\mu}$ (see \eqref{eqdefevtau} for the definition of $E_{V,\tau}$). 
For this purpose, we recall the standard Gagliardo--Nirenberg inequality: for $q \in [2,2^*)$, 
\begin{equation}\label{stdGN}
	\norm{v}_q^q \leq C_{q,N} \norm{\nabla v}_2^{\beta_q q} \norm{v}_2^{(1-\beta_q) q} \quad \text{for every $ v \in H^1(\R^N)$},
\end{equation}
where $\beta_q \coloneq N \ab( \frac{1}{2} - \frac{1}{q} )$. 
By \eqref{equpsi} and the Gagliardo-Nirenberg inequality for $v=u/\psi_0$ with $u \in C^\infty_0(\R^N)$, we obtain
\begin{equation*}\label{eqgnc0i}
   	\norm{ u }_q^q \leq C \norm{ u/\psi_0 }_q^q \leq C  \norm{ u/\psi_0 }_2^{(1 - \beta_q)q} \norm{ \nabla (u/\psi_0) }_2^{\beta_q q}
	\leq C \norm{ u }_2^{{(1 - \beta_q)q}} \abs{u}_{V}^{\beta_q q}. 
\end{equation*}
From the density of $C^\infty_0(\R^N)$ in $H^1(\R^N)$, 
the following Gagliardo-Nirenberg type inequality is obtained: 

\begin{lemma}\label{lemgn}
Let $N \geq 1 $.  For any  $q \in [2, 2^*)$, there exists a constant $C_{q, N, V}>0$, depending only on  $q,N$ and $V$, such that
\begin{equation*}\label{eqgn}
    \|u\|_{q}^q \leq C_{q,N,V} \|u\|_{2}^{(1 - \beta_q)q} \abs{u}_{V}^{\beta_q q} \quad \text{for any $u \in H^1(\R^N)$},
\end{equation*}
where $\beta_q = N \left( \frac{1}{2} - \frac{1}{q} \right)$. 
Moreover, for $N \geq 3$, there exists a constant $C_{N, V}>0$, depending only on  $N$ and $V$, such that
\begin{equation*}\label{eqgn-critical}
    \|u\|_{2^*}^{2^*} \leq C_{N,V}   \abs{u}_{V}^{2^*} \quad \text{for any $u \in H^1(\R^N)$}.
\end{equation*}
\end{lemma}

\subsection{Liouville type results for NLS equations with periodic potentials}
\label{subec:Liouville}

This subsection is devoted to proving a Liouville type result for \eqref{eqV} with $\lambda = 0$. 
This will be used to rule out the case $\lambda \leq 0$ for each nonnegative solution to \eqref{eqV} with a finite Morse index $m_{\lambda,\tau}(u) < \infty$. 
See \cref{d:Mor-apMor} for the definition of $m_{\lambda,\tau} (u)$. 

\begin{lemma}\label{lemliou}
  Suppose $0<s \leq N/(N-2)$ when $N \geq 3$ and $s=+\infty$ when $N=1,2$. 
 Assume that $v \in L^s(\R^N \setminus B_{R_0} ) \cap C(\R^N \setminus B_{R_0} ) \cap H^2_{\rm loc} (\R^N \setminus B_{R_0} ) $ is a nonnegative function satisfying 
 $v(x) \to 0$ as $|x| \to +\infty$ and 
 $L_0v = -\diver( \psi^2_0 \nabla v ) \geq 0$ in $\R^N \setminus B_{R_0}$ for some $R_0>0$. 
 Then $v \equiv 0$ in $\R^N \setminus B_{R_0}$. 
In particular, if $1 \leq N \leq 4$, $\Omega$ is bounded, and $(u,\lambda)\in H^1(\R^N) \times \R$ is a solution to \eqref{eqV}
 under \eqref{A1} and \eqref{A3}, then $\lambda > 0$ holds.
\end{lemma}

\begin{proof}
We first prove the latter assertion and let $(u,\lambda) \in H^1(\R^N) \times \R$ be a solution to \eqref{eqV} 
and set $v \coloneq u / \psi_0$. By \eqref{A3} and the fact that $\psi_0\in C^1(\R^N)$ is positive and periodic, 
elliptic regularity leads to $u,v \in W^{2,\nu}_{\rm loc} (\R^N) \cap L^2(\R^N) \cap C^1(\R^N)$ for each $\nu< \infty$ and  $v(x) \to 0$ as $|x| \to +\infty$. 
By writing $a_\pm \coloneq \max \set{\pm a, 0} \geq 0$, since $u \geq 0$ and 
\[
-\Delta u + \ab[ V + \lambda - \chi_\Omega \frac{f(u)}{u}  ]_+ u \geq 0 \quad \text{in} \ \R^N,
\]
the strong maximum principle (e.g., \cite[Theorem 8.19]{GT}) implies $u>0$ in $\R^N$ and $v>0$ in $\R^N$. 
From \eqref{eqschpsiv}, the boundedness of $\Omega$ and $u>0$, 
there exists $R_0>0$ such that 
\[
L_0v = \psi_0 \ab( -\Delta u + V u ) = - \lambda \psi_0 u + \psi_0 \chi_{\Omega} f(u) = -\lambda \psi_0 u \quad \text{in} \ \R^N \setminus B_{R_0}. 
\]
Since $v \in L^2(\R^N)$, $v>0$ in $\R^N$ and $N/(N-2) \geq 2$ for $N=3,4$, the former assertion excludes the case $\lambda \leq 0$, 
and $\lambda > 0$ holds.

To prove the first assertion, we argue by contradiction and suppose that $v\not \equiv 0$ in $\R^N\setminus B_{R_0}$. 
The strong maximum principle, applied separately to each unbounded connected component when $N=1$, yields $v>0$ on at least one such component; 
for $N\geq2$, it yields $v>0$ in $\R^N\setminus B_{R_0}$.
In the sequel, we separate the cases $N=1$, $N=2$ and $N \geq 3$.

{\bf Case 1:} 
When $N=1$, we may assume that $v>0$ on $(R_0,+\infty)$. 
Hence,  $v\in C^1((R_0,+\infty))$ and $v(x)\to0$ as $x\to+\infty$ imply that there exists $x_0>R_0$ such that $v'(x_0)<0$.
Since $\psi_0^2 v' \in H^1_{\rm loc} (\R \setminus (-R_0,R_0))$, 
$- (\psi_0^2 v')' \geq 0$ a.e. $(R_0,+\infty)$ and $\psi_0^2(x_0) v'(x_0) < 0$, 
we have $\psi_0^2(x) v'(x) \leq \psi_0^2(x_0) v'(x_0) < 0$ for every $x \in (x_0,+\infty)$. 
However, since $\psi_0 > 0$ in $\R$ and $\psi_0$ is periodic, this contradicts $v(x) \to 0$ as $x \to +\infty$. 

{\bf Case 2:} 
We deal with the case $N=2$. By \cite[Theorem A.2]{KN}, there exists a unique fundamental solution 
$F\in H^1_{\text{loc}}(\R^2\setminus\{0\})$ for $L_0 $, with pole at the origin, and constants $K_1,K_2>0$ 
and $R_1>1$ such that
\begin{equation}\label{eqlem31n21}
    \int_{\R^2} \psi^2_0 \nabla F \cdot \nabla \varphi\, dx= \varphi(0) \quad \text{for each $\varphi\in C_0^\infty(\R^2)$}
\end{equation}    
and
\begin{equation}\label{eqlem31n22}
   -K_1\log\abs{x}\leq F(x)\leq -K_2\log\abs{x} \quad \text{for any $x \in \R^2$ with $\abs{x}> R_1$}. 
\end{equation}
Let $R_2:=\max\{R_0,R_1+1\}$. 
Then, by \eqref{eqlem31n21}, we have 
\[
L_0 F=0 \quad \text{in } \R^2\backslash \overline{B_{R_2}},
\]
which implies that $F \in C^2(\R^2\setminus \overline{B_{R_2}})\cap C(\overline{\R^2\setminus \overline{B_{R_2}}})$. 

We now repeat the argument in \cite[Lemma A.2]{I} and \cite[Theorem 29 in Chapter 2]{PW}. 
For every $r \geq R_2$, set $\displaystyle m(r):= \min_{\abs{x}=r}v(x)$ and $\displaystyle M(r):= \max_{\abs{x}=r}F(x)$. Then, by \eqref{eqlem31n22}, we know that, for $r>R_2$ large enough, $M(R_2)>M(r)$.
For $r>R_2$ large enough and $x \in \bar{B_r} \setminus B_{R_2}$, set
\[
\varphi_r(x) \coloneq \frac{M(R_2)m(r)-M(r)m(R_2)}{M(R_2)-M(r)}+\frac{m(R_2)-m(r)}{M(R_2)-M(r)}F(x). 
\]
Since $v>0$ in $\R^2 \setminus B_{R_0}$ and $v \in C(\R^2 \setminus B_{R_0})$, 
we find that $m(r)>0$ holds for all $r > R_0$. Moreover, the weak maximum principle gives 
\[
\min_{ R_2 \leq |x| \leq r } v = \min \Set{ m(R_2) , m(r) } .
\]
By $v(x) \to 0$ as $|x| \to +\infty$, 
we may take a sufficiently large $r_0>R_2$ so that $m(r) = \min_{|x| = r} v < m(R_2)$ for all $r \geq r_0$.
Therefore, for each $r>r_0$, we obtain
\begin{equation}\label{eqphi1r1}
   \varphi_r(x) \leq m(R_2) \quad \text{for } \abs{x}=R_2, \quad \text{and}\quad \varphi_r(x) \leq m(r) \quad \text{for } \abs{x}=r,
\end{equation}
and
\begin{equation}\label{eqlug}
    L_0 \ab( v - \varphi_r ) \geq 0 \quad \text{in } B_{r}\backslash \bar{B}_{R_2}.
\end{equation}
By \eqref{eqphi1r1}, for any $r>r_0$, we have
\[
v -\varphi_r \geq 0\quad \text{for } \abs{x}=R_2 \text{ or } \abs{x}=r
\]
and from \eqref{eqlug} and the maximum principle, it follows that
\[
v(x) \geq \varphi_r(x) \quad \text{for every  $x \in B_{r}\backslash \bar{B}_{R_2}$}.
\]
Since $v(x) \to 0$ as $|x| \to +\infty$, we know that $m(r) \to 0$ as $r \to +\infty$.
Thus, by \eqref{eqlem31n22}, we see that  $M(r) \to -\infty$, hence,
\[
\begin{aligned}
  v(x) &\geq \lim_{r\to+\infty}\left(\frac{M(R_2)m(r)-M(r)m(R_2)}{M(R_2)-M(r)}+\frac{m(R_2)-m(r)}{M(R_2)-M(r)}F(x)\right)\\
  &= \lim_{r\to+\infty}\frac{-M(r)m(R_2)}{M(R_2)-M(r)}\\
  &= m(R_2) > 0 \quad \text{for any } x\in \R^2\backslash \bar{B}_{R_2}, 
\end{aligned}
\]
which leads to a contradiction since $v(x) \to 0$ as $|x| \to \infty$.

  {\bf Case 3:} We treat the case $N \geq 3$.  
By \cite{GW} (see also \cite{LSW} and \cite[Theorem 3.8]{D}), 
for any $r>0$, there exist a unique Green function $G_r(x,y):B_r\times B_r\backslash\{x=y\}\to\R$ and two constants $K_1',K_2' >0$ depending only on $N$ and $\psi_0$ such that 
\begin{equation}\label{eqgreen}
  \int_{B_r} \psi^2_0 \nabla G_r (\cdot,y)\cdot\nabla \varphi \, dx=\varphi(y) \quad \text{for all $\varphi\in C_0^\infty(B_r)$},
\end{equation}
and
\begin{equation}\label{eqgreeng}
 K_1' \vab{x-y}^{2-N} \leq G_r(x,y) \leq K_2' \abs{x-y}^{2-N} \quad \text{for all $x,y \in B_r$ with $\abs{x-y} \leq\frac{1}{2}\operatorname{dist}(y,\partial B_r)$}.
\end{equation}
Set $G_r(x):=G_r(x,0)$. Then, by \eqref{eqgreen}, for any $r>1$, we have
\begin{equation*}
    L_0G_r  =0 \quad \text{in } B_r\backslash \bar{B}_1.
\end{equation*}
Notice that $G_r \in H^1( B_r \setminus B_{R_0} )\cap W^{1,1}_0(B_r)$ (see \cite[Eq. (1.3)]{GW}). 
	From the equation, $G_r \in C^2( B_{r} \setminus B_{R_0} ) \cap C( \overline{ B_r \setminus B_{R_0} } )$ 
	with $G_r = 0$ on $\partial B_r$. 
	By \eqref{eqgreeng} and $v>0$ in $\R^N \setminus B_{R_0}$, 
	there exists $c_0>0$ such that for all sufficiently large $r$, 
	\[
	c_0 \max_{ \partial B_{R_0} } G_r < \min_{ \partial B_{R_0} } v.
	\]
	On the other hand, notice that 
	\[
	\left\{\begin{aligned}
		&L_0 \ab( c_0 G_r ) = 0 \leq L_0 v \quad 
		\text{in $B_r \setminus B_{R_0}$}, 
		\\
		&c_0 G_r < v \quad \text{on} \ \partial B_{R_0}, \quad 
		c_0 G_r = 0 < v \quad \text{on} \ \partial B_r.
	\end{aligned}\right.
	\]
	The comparison theorem leads to 
	$c_0 G_r \leq v$ in $B_r \setminus B_{R_0}$ for all sufficiently large $r$. 
	By letting $r \to + \infty$, we see that 
	$c_1 |x|^{2-N} \leq v$ in $\R^N \setminus B_{R_0}$, however, this contradicts $v \in L^s(\R^N \setminus B_{R_0})$. 
\end{proof}

\subsection{A classification result for NLS equations with nonlinearities localized on half-spaces}
\label{subsec:half}

In this subsection, in order to apply blow-up analysis, we study the bounded solutions to the following equation
 \begin{equation}\label{eqneH}
    -\Delta u + \lambda u=\chi_{{H}}\abs{u}^{p-2}u \quad \text{in }\R^N,
 \end{equation}
 where  $\lambda\geq 0$, $p>2$ and $H$ is a half-space.

 We first consider the case where $\lambda=0$. The following lemma follows directly from  Liouville's theorem for $N=1,2$ 
 and the proof of \cite{ChLi91} and \cite[\S 8.4]{QuSo19} 
 for the result of \cite{GS}:
 \begin{lemma}\label{lemnelambda+0}
    Let $N \geq 1$,  $p\in (2,2^*)$ and $u \in C(\R^N) \cap H^1_{\rm loc} (\R^N)$  be a bounded nonnegative solution of 
    \begin{equation}\label{eqneH0}
    -\Delta u =\chi_{{H}}\abs{u}^{p-2}u \quad \text{in }\R^N. 
 		\end{equation} 
 Then $u \equiv 0$.
 \end{lemma}
For completeness, the proof of \cref{lemnelambda+0} is provided
in \cref{app:lemnelambda+0}.

To deal with the case where $\lambda > 0$, 
we need the notation of stability from \cite[Definition 2.1]{EP} and \cite[Definition]{Fa}:  
\begin{definition}\label{defstable}
	Let $\lambda\geq 0$, $p>2$ and $H$ be a half-space.  We say that a solution $u$ of \eqref{eqneH}
	\begin{itemize}
		\item 
		is stable if
		\[
		Q_{\lambda,u}(\varphi) \coloneq \int_{\R^N} |\nabla \varphi|^2 + \lambda\varphi^2 - (p-1)\chi_{{H}}|u|^{p-2}\varphi^2 \,dx\ge 0 
		\quad \text{for all $\varphi \in C^1_0(\R^N)$};
		\]
		\item 
		is stable outside a compact set $K$ if $Q_{\lambda, u}(\varphi ) \ge 0$ for any $\varphi \in C_0^1(\R^N \setminus K)$;
		\item 
		has Morse index $m_{\lambda,H}(u)$ equal to $k$ if $k$ is 
		the maximal dimension of a subspace $W \subset C_0^1(\R^N)$ so that $Q_{\lambda,u}(\varphi)< 0$ for any $\varphi \in W \setminus \{0\}$.
	\end{itemize}
\end{definition}

By adapting the argument in \cite{Fa,EsLi82}, we obtain the following lemma.

 \begin{lemma}\label{lem25}
   Let $N \geq 1$, $\lambda>0$,  $p>2$ and $u \in C^1(\R^N)$  be a bounded nonnegative solution of \eqref{eqneH}. 
   Assume that $u$ is stable outside a compact set $K$. 
   Then $u \equiv 0$ holds. 
%
 \end{lemma}

\begin{proof} 
We first prove that $u(x) \to 0$ as $\abs{x} \to +\infty$ and $u \in H^1(\R^N)$. 
We argue by contradiction and suppose that there exist a constant $\delta>0$ and a sequence $\{x_n\}\subset \R^N$ such that $\abs{x_n}\to+\infty$ as $n\to+\infty$ and 
\begin{equation}\label{equxndelta}
	u(x_n)\geq\delta \quad \text{for any $n\in \mathbb{N}$}.
\end{equation}
Let $u_n(x) := u(x + x_n)$ and $\chi_{{H}_n}(x):=\chi_{{H}}(x+x_n)$. Then, by passing to a subsequence, we obtain $\chi_{{H}_n}(x) \to \chi_{{H}_\infty}(x)$ for almost every $ x \in \R^N$, where $H_\infty$ is $\emptyset$, $\R^N$ or a half-space. Moreover, since $u$ is bounded, we know that $\{u_n\}$ is bounded in $L^\infty(\R^N)$. Thus, by elliptic regularity and up to a subsequence, there exists $u_\infty \in C^1(\R^N)$ such that $u_n \to u_\infty$ in $C^1_\text{loc}(\R^N)$, where $u_\infty$ satisfies 
\begin{equation}\label{eqneHinfty}
	-\Delta u_\infty +\lambda u_\infty = \chi_{{H}_\infty} u_\infty^{p-1} \quad \text{in} \ \R^N. 
\end{equation}
By \eqref{equxndelta}, we obtain $u_\infty(0)\geq \delta>0$, which implies that $u_\infty> 0$ in $\R^N$ due to the strong maximum principle. 

Next, we exclude the case $H_\infty = \emptyset$. Indeed, if $H_\infty = \emptyset$, then 
$u_\infty$ is a bounded positive solution to $-\Delta u_\infty + \lambda u_\infty = 0$ in $\R^N$, 
hence $(\vab{\xi}^2 + \lambda) \widehat{u_\infty} (\xi) = 0$ in $\mathcal{S}^*$, 
where $\widehat{u_\infty}$ and $\mathcal{S}^*$ denote the Fourier transform of $u_\infty$ and the space of tempered distributions. 
Since $\lambda > 0$, we infer $\widehat{u_\infty} \equiv 0$ and $u_\infty \equiv 0$, which is a contradiction. 
Thus, $H_\infty = \emptyset$ does not occur.

To proceed, when $H_\infty$ is a half-space, we remark that 
by a rotation and a translation, we may assume 
\[
H_\infty = \set{x \in \R^N : x_N>0 }.
\]
In this case, for the later use, we claim that 
\begin{equation}\label{compH}
	u_\infty (x',x_N) \geq u_\infty (x',-x_N) > 0 \quad \text{for each $(x',x_N) \in H_\infty$}. 
\end{equation}
Since $u_\infty$ is a bounded positive solution to \eqref{eqneHinfty}, $u_\infty(x)$ can be expressed as 
\[
u_\infty(x) = \int_{\R^N} G_\lambda (x-y) \chi_{H_\infty} (y) u_\infty(y)^{p-1} \, dy
= \int_{H_\infty} G_\lambda \ab( \abs{x-y} ) u_\infty(y)^{p-1} \, dy \quad \text{for every $x \in \R^N$},
\]
where $G_\lambda$ is the Green function of $-\Delta + \lambda$. 
Since $G_\lambda(\cdot)$ is decreasing in $(0,\infty)$, 
for any $(x',x_N), (y',y_N) \in H_\infty$, the fact $\vab{(x'-y',x_N-y_N)} \leq \vab{(x'-y',-x_N-y_N)}$ yields 
\[
\begin{aligned}
	u_\infty(x',x_N) - u_\infty(x',-x_N) 
	&= 
	\int_{H_\infty} \bab{ G_\lambda \ab( \vab{ (x'-y',x_N-y_N) } ) - G_\lambda \ab( \vab{( x'-y',-x_N-y_N )} ) } u_\infty(y)^{p-1} \, dy
	\\
	& \geq 0.
\end{aligned}
\]
Therefore, \eqref{compH} holds.

Let $\eta \in C_0^\infty(\mathbb{R}^N)$ be a radial function such that
\begin{equation}\label{eqeta}
	\begin{cases}
		0 \le \eta \le 1, \\
		\eta \equiv 1 \quad &\text{in } B_{R}, \\
		\eta \equiv 0 \quad &\text{in } \R^N \setminus B_{2R}, \\
		R|\nabla \eta|,R^2\abs{\Delta(\eta^2)} \le C \quad &\text{in } B_{2R} \setminus B_R.
	\end{cases}
\end{equation}
Since $\abs{x_n} \to +\infty$ and $u$ is stable outside a compact set, $u_\infty$ is stable, that is 
\[
Q_{\lambda, u_\infty,H_\infty}(\varphi) \coloneq 
\int_{\R^N} |\nabla \varphi|^2 + \lambda\varphi^2 - (p-1)\chi_{{H}_\infty}|{u_\infty}|^{p-2}\varphi^2 \,dx\ge 0 \quad \text{for all } \varphi \in C_0^1(\R^N).
\]
Let $\varphi \coloneq u_\infty^{\alpha}\eta$, where $\alpha>1$ will be specified later. 
Then, integrating by parts leads to 
\[
\begin{aligned}
	\int_{\R^N}\abs{\nabla \varphi}^2\,dx&= \int_{\R^N} \alpha^2 u_\infty^{2\alpha-2} \eta^2 |\nabla {u_\infty}|^2+ u_\infty^{2\alpha} |\nabla \eta|^2 
	+ 2\alpha u_\infty^{2\alpha-1} \eta (\nabla {u_\infty} \cdot \nabla \eta)\,dx\\
	&=\int_{\R^N}  \alpha^2 u_\infty^{2\alpha-2} \eta^2 |\nabla {u_\infty}|^2+ u_\infty^{2\alpha} |\nabla \eta|^2 
	+  \frac{1}{2} \nabla( u_\infty^{2\alpha}) \cdot \nabla(\eta^2) \,dx\\
	&=\int_{\R^N}  \alpha^2 u_\infty^{2\alpha-2} \eta^2 |\nabla {u_\infty}|^2+ u_\infty^{2\alpha} |\nabla \eta|^2 - \frac{1}{2} u_\infty^{2\alpha} \Delta(\eta^2)\,dx;
\end{aligned}
\]
hence, by $Q_{\lambda,u_\infty,H_\infty}(\varphi )\geq 0$, we obtain
\begin{equation}\label{eqneHm}
	\begin{aligned}
		\int_{\R^N} \alpha^2 u_\infty^{2\alpha-2} \eta^2 |\nabla {u_\infty}|^2 + u_\infty^{2\alpha}|\nabla \eta|^2 - \frac{1}{2} u_\infty^{2\alpha}\Delta(\eta^2) 
		+ \lambda u_\infty^{2\alpha}\eta^2 \,dx
		&=\int_{\R^N} |\nabla \varphi|^2 + \lambda\varphi^2  \,dx 
		\\
		&\ge (p-1) \int_{{H}_\infty} u_\infty^{p-2+2\alpha}\eta^2\,dx.
	\end{aligned}
\end{equation}
Use $u_\infty^{2\alpha-1}\eta^2$ as a test function in \eqref{eqneHinfty} to get 
\begin{equation}\label{eqneHn}
	(2\alpha-1) \int_{\R^N} u_\infty^{2\alpha-2} \eta^2 |\nabla {u_\infty}|^2 - \frac{1}{2\alpha} u_\infty^{2\alpha}\Delta(\eta^2) + \lambda u_\infty^{2\alpha}\eta^2 \,dx
	= \int_{{H}_\infty} u_\infty^{p-2+2\alpha}\eta^2\,dx.
\end{equation}
It follows from \eqref{eqneHm} and \eqref{eqneHn} that 
\begin{equation}\label{eqneHineq}
	\begin{aligned}
		&\int_{\R^N} u_\infty^{2\alpha}|\nabla \eta|^2 - \frac{\alpha-1}{2(2\alpha-1)} u_\infty^{2\alpha}\Delta(\eta^2)\,dx
		\geq \frac{A(\alpha)}{2\alpha-1} \int_{{H}_\infty}  u_\infty^{p-2+2\alpha}\eta^2 + \frac{(\alpha-1)^2}{2\alpha-1}\lambda  u_\infty^{2\alpha}\eta^2 \,dx,
	\end{aligned}
\end{equation}
where $A(\alpha)\coloneq-\alpha^2+2(p-1)\alpha-(p-1)$. By a direct computation, we know that $A(\alpha)\geq 0$ on $[p-1-\sqrt{(p-1)(p-2)},p-1+\sqrt{(p-1)(p-2)}]$, and thus, for $p>2$, we can always choose a suitable $\alpha > 1$ such that $A(\alpha) \ge 0$.  Fix $\alpha>1$ such that $A(\alpha) \ge 0$. Then, by \eqref{eqeta} and \eqref{eqneHineq}, we have
\begin{equation}\label{eqneHind}
	\frac{C_\alpha}{R^2}\int_{B_{2R}\setminus B_{R}} u_\infty^{2\alpha} \,dx\geq \int_{B_R \cap H_\infty } u_\infty^{2\alpha} \,dx,
\end{equation}
where $C_\alpha>0$ depends only on $\alpha>1$, $\lambda>0$ and $N \geq 1$. 
When $H_\infty = \set{x \in \R^N | x_N>0}$, by changing $C_\alpha$ to $2C_\alpha$, 
\eqref{compH} and \eqref{eqneHind} give 
\begin{equation}\label{half}
	\int_{B_R} u_\infty^{2\alpha} \, dx 
	\leq 2 \int_{B_R \cap H_\infty} u_\infty^{2\alpha} \, dx \leq \frac{C_\alpha}{R^2}\int_{B_{2R}\setminus B_{R}} u_\infty^{2\alpha} \,dx.
\end{equation}

In both cases $H_\infty = \R^N$ and $H_\infty = \set{x_N>0}$, define 
\[
m(r) \coloneq \int_{B_r} u_\infty^{2\alpha}\,dx.
\]
Then,  \eqref{eqneHind} if $H_\infty = \R^N$ and \eqref{half} if $H_\infty$ is a half-space lead to
\[
m(2R) \ge \ab( 1 + \frac{R^2}{C_\alpha} ) m(R).
\]
Since $u_\infty$ is positive, fix any $R_0>\sqrt{C_\alpha}$ such that $m(R_0)>0$.
Then, we obtain 
\[
m(2^k R_0) \ge m(R_0) \prod_{j=0}^{k-1} \ab( 1 + \frac{4^j R_0^2}{C_\alpha} )>m(R_0)4^\frac{k(k-1)}{2} 
= m(R_0) 2^{k(k-1)}.
\]
However, since $u_\infty$ is bounded, we have 
\[
m(2^k R_0)=\int_{B_{2^k R_0}} u_\infty^{2\alpha}\,dx\leq \norm{u_\infty}_\infty^{2\alpha}\int_{B_{2^k R_0}}1\,dx\leq C{2^{Nk}}R_0^N\norm{u_\infty}_\infty^{2\alpha},
\]
which is a contradiction since $\frac{2^{k(k-1)}  }{2^{Nk}} \to+\infty$ as $k \to +\infty$. 
Thus $u(x) \to 0$ as $\abs{x} \to +\infty$.

From $\lambda > 0$ and $u(x) \to 0$ as $\abs{x} \to +\infty$, 
the standard argument yields 
\begin{equation}\label{exp-decay}
	\abs{u(x)}+ \abs{\nabla u(x)} \leq Ce^{-\sqrt{\lambda}\abs{x}},
\end{equation}
which implies that $u \in H^1(\R^N)$.

Finally, we prove $u \equiv 0$. 
By a rotation and a translation, we may suppose $H = \set{x \in \R^N | x_N > 0}$. 
Since $\chi_H \vab{u}^{p-2} u \in L^2(\R^N)$ in view of \eqref{exp-decay}, 
we have $u \in H^2(\R^N)$. 
As in \cite[Proposition I.2]{EsLi82}, 
testing \eqref{eqneH} with $\partial u / \partial x_N \in H^1(\R^N)$ gives 
\[
\int_{\R^N} \nabla u \cdot \nabla \ab( \frac{\partial u}{\partial x_N} ) + \lambda u \frac{\partial u}{\partial x_N} \, dx 
= \int_H \vab{u}^{p-2} u \frac{\partial u}{\partial x_N} \, dx. 
\]
By \eqref{exp-decay} and Fubini's theorem, the left-hand side is computed as 
\[
\int_{\R^N} \nabla u \cdot \nabla \ab( \frac{\partial u}{\partial x_N} ) + \lambda u \frac{\partial u}{\partial x_N} \, dx 
= \int_{\R^{N-1}} dx' \int_{-\infty}^\infty \frac{1}{2} \frac{\partial}{\partial x_N} \ab( \vab{\nabla u}^2 + \lambda u^2 ) \, dx_N = 0.
\]
Therefore, recalling $H= \set{x \in \R^N | x_N>0}$ yields 
\[
0 = \int_H \vab{u}^{p-2} u \frac{\partial u}{\partial x_N} \, dx 
= \int_{\R^{N-1}} dx' \int_0^\infty \frac{1}{p} \frac{\partial}{\partial x_N} \vab{u}^p \, dx_N 
= - \frac{1}{p} \int_{\R^{N-1}} \vab{u(x',0)}^p \, dx'. 
\]
Hence, $u(x',0) \equiv 0$ in $\R^{N-1} \times \set{0}$. 

Since we may suppose $u \in C^2( \R^{N-1} \times [0,\infty) )$, \cite[Theorem 1.1]{EsLi82} leads to $u \equiv 0 $ in $\R^{N-1} \times [0,\infty)$. 
On the other hand, exploiting $u$ as a test function to \eqref{eqneH} on $\R^{N-1} \times (-\infty,0)$ implies $u \equiv 0$ in $\R^{N-1} \times (-\infty, 0)$, 
which completes the proof. 
%
%
\end{proof}

\begin{remark}\label{R:entire}
	In \cite[Theorem 2.3]{EP}, it is proved that when $N \geq 2$ and $2<p<2^*$, 
	$u(x) \to 0$ as $\vab{x} \to +\infty$ and $u \in H^1(\R^N)$ hold for any solution $u$ to 
	\[
	-\Delta u + u = \vab{u}^{p-2} u \quad \text{in} \ \R^N,
	\]
	which is stable outside a compact set. 
	We remark that the proof of \cite[Theorem 2.3]{EP} also works for $N=1$ and the same conclusion holds when $N=1$. 
\end{remark}

 \section{Mountain pass geometry}
 \label{sec:mp}
 
 This section is devoted to finding a uniform mountain pass geometry for a family $\{E_{V,\tau}\}_{\tau \in [1/2,1]}$ of functionals 
 in order to apply the monotonicity trick in \cite{BCJS}. 
 In addition, we state the nonexistence of critical points $u$ of $E_{V,\tau}$ for small $L^2$-norm with $u>0$ in $\R^N$.

To define $E_{V,\tau}$, we first decompose $f$ into two parts as in \cite{CRZ}. 
 By \eqref{A1} and \eqref{A3} with \eqref{eqodd}, 
 there exists $R_0 > 0$ such that $f(t)t > 0$ for $|t| \ge R_0$. Utilizing this observation, we define
 \begin{equation}\label{eqdeff1f2}
 	f_1(t) \coloneq \eta_0(t)f(t) \quad \text{and} \quad f_2(t) \coloneq (1 - \eta_0(t))f(t),
 \end{equation}
 where $\eta_0 \in C^\infty(\R)$ is an even function such that $|\eta_0'(t)| \le 2$ for $R_0 < |t| < R_0 + 1$ and 
 \[
 \eta_0(t) = \begin{dcases}
 	1 & \text{if }  |t| \ge R_0 + 1, \\
 	0 & \text{if } |t| < R_0.
 \end{dcases}
 \]
It is clear that $f_1(t)t \geq 0$ for all $t \in \R$. 
 For $\tau \in [\frac{1}{2},1]$, we define $E_{V,\tau} : H^1(\R^N) \to \mathbb{R}$ by
 \begin{equation}\label{eqdefevtau}
 	E_{V,\tau}(u) \coloneq \frac{1}{2}\abs{u}_V^2 - \int_\Omega F_2(u) \,dx - \tau \int_\Omega F_1(u) \,dx, \quad 
 	F_i(t) \coloneq \int_0^t f_i(s) \, ds \quad (i=1,2). 
 \end{equation}
 Note that any positive critical point of $E_{V,1}|_{S_\mu}$ gives a solution to \eqref{eqV} and 
 \begin{equation}\label{mono-Evtau}
 	E_{V,\tau_1} (u) \geq E_{V,\tau_2} (u) \quad \text{for each $u \in H^1(\R^N)$ and $\frac{1}{2} \leq \tau_1 < \tau_2 \leq 1$}. 
 \end{equation}
In \cref{subsec:mp-small}, the uniform mountain pass structure for $E_{V,\tau} |_{S_\mu}$ is proved under \eqref{A1} and \eqref{A3} for a small $\mu>0$. 
On the other hand, in \cref{subsec:mp-all}, the same structure is established under \eqref{A1}, \eqref{A3} and \eqref{A5} for all $\mu > 0$. 
Finally, \cref{subsec33} is devoted to the nonexistence of positive critical points of $E_{V,\tau}$ when the $L^2$ norm is small. 
 
 \subsection{For $\mu>0$ small enough}
 \label{subsec:mp-small}

\begin{lemma}\label{lemmps}
Let $N \geq 1$ and suppose \eqref{A1} and \eqref{A3}. 
Then, for every $\alpha_0>0$, there exists $\mu^*>0$ such that, for any $\mu \in (0, \mu^*)$, 
there exist $w_1, w_2 \in S_\mu$ such that $w_1,w_2\geq 0$ and 
\[
    c_\tau = c_\tau(\mu) \coloneq 
    \inf_{\gamma \in \Gamma} \max_{t \in [0,1]} E_{V,\tau}(\gamma(t)) \geq \frac{3}{8} \alpha_0  > \max\{E_{V,\tau}(w_1), E_{V,\tau}(w_2), 0\} 
    \quad \text{for all $\tau \in \ab[ \frac{1}{2}, 1 ]$}, 
\]
where
\begin{equation*}
  \Gamma \coloneq \Set{\gamma \in C([0,1], S_\mu) : \gamma(0) = w_1, \, \gamma(1) = w_2 }.
\end{equation*}
Moreover, $[1/2, 1] \ni \tau \mapsto c_\tau \in (0,\infty)$ is nonincreasing. 
\end{lemma}

\begin{proof}
The monotonicity of $c_\tau$ follows from \eqref{mono-Evtau}. 
To find $w_1,w_2$ and to prove the inequality for $c_\tau$, 
set $B_\alpha \coloneq \set{u \in S_\mu: \abs{u}_V^2=\alpha}$. 
By \eqref{A1}, \eqref{A3} and definitions of $f_1$ and $f_2$ given in \eqref{eqdeff1f2}, 
for any $\varepsilon_0>0$, there exists a constant $C_{\varepsilon_0}>0$ such that 
\begin{equation}\label{eqsubsec121}
    \abs{f_1(t)}\leq C\abs{t}^{p-1}\text{ and }\abs{f_2(t)}\leq \varepsilon_0\abs{t}+C_{\varepsilon_0}\abs{t}^{p-1} \quad \text{for all $t \in \R$}.
\end{equation}
Fix $\alpha_0>0$ and $\varepsilon_0>0$. Then, by \cref{lemgn}, 
for any $u \in B_{\alpha_0}$ and $\tau \in [\frac{1}{2},1]$, we obtain 
\[
E_{V,\tau}(u) 
\geq 
\frac{1}{2}\alpha_0-\frac{1}{2}\varepsilon_0\norm{u}_2^2-C\norm{u}_p^p 
\geq 
\frac{1}{2}\alpha_0-\frac{1}{2}\varepsilon_0\mu-C C_{p,N,V}\mu^{(1 - \beta_p)p/2}\alpha_0^{\beta_p p/2}. 
\]
Thus, there exists  $\mu^*\in(0,  \min \{ \alpha_0/(4\varepsilon_0) , \alpha_0 / 8 \} )$ small enough such that
\begin{equation}\label{lweonBa}
	E_{V,\tau}(u) \geq \frac{3\alpha_0}{8} \quad \text{for every $\mu\in (0,\mu^*)$, $u \in B_{\alpha_0}$ and $\tau \in \ab[\frac{1}{2},1]$}. 
\end{equation}

Fix any $\mu\in (0,\mu^*)$. 
By recalling \eqref{inf-spect} 
\[
0 = \inf \sigma (-\Delta+V)=\displaystyle\inf_{u\in H^1(\R^N)\setminus \{0\}}\frac{\abs{u}_V^2}{\norm{u}_2^2}, 
\] 
\eqref{eqsubsec121}, $0<\mu < \mu^* < \min \{ \alpha_0 / (4\varepsilon_0) , \alpha_0/8\} $ and \cref{lemgn} imply that 
there exists $w_1 \in S_\mu$ satisfying $w_1\geq 0$ and $\abs{w_1}_V^2\in (0,\alpha_0/8)$ small enough such that 
for each $\tau \in [\frac{1}{2},1]$, 
\[
\begin{aligned}
    E_{V,\tau}(w_1) &\leq \frac{1}{2}\abs{w_1}_V^2+ \frac{\varepsilon_0}{2}\|w_1\|_{2}^2 +CC_{p,N,V} \|w_1\|_{2}^{(1 - \beta_p)p} \abs{ w_1}_{V}^{\beta_p p}\\
    &\leq\frac{3}{16}\alpha_0 +CC_{p,N,V} \mu^{(1 - \beta_p)p/2} \abs{ w_1}_{V}^{\beta_p p} \leq \frac{\alpha_0}{4}.
\end{aligned}
\]
By \eqref{A1} and \eqref{A3}, we find $R_1>R_0>0$ such that 
\[
\abs{F_2(t)} \leq C\abs{t}^2 \quad \text{for all $t \in \R$}
\]
and
\[
F_1(t) \geq \frac{a_0}{4p}\abs{t}^p-C\abs{t}^2 \quad \text{for every $t \in \R$ with $\abs{t} \geq R_1$}. 
\]
Let $v \in C_0^\infty(\R^N)\cap S_{\mu}$ satisfy $v \geq 0$. 
For any $s>0$, set $v_s(x) \coloneq s^{N/2}v(s(x-x_0))$ for some  $x_0 \in \Omega$. 
Then, by $p>2+\frac{4}{N}$ and $\tau \in [\frac{1}{2},1]$, 
for $w_2 \coloneq v_{s_0}$ with $s_0>0$ sufficiently large, 
we have $\supp w_2 \subset \Omega$, $\abs{w_2}^2_V> \alpha_0$ and
\[
\begin{aligned}
    E_{V,\tau}(w_2)&\leq \frac{1}{2}\norm{\nabla w_2}_2^2+\frac{1}{2}\norm{V}_\infty\norm{w_2}_2^2+C\norm{w_2}_2^2-\frac{a_0 }{8p}\norm{w_2}_p^p\\
    &=\frac{s_0^2}{2}\norm{\nabla v}_2^2+\frac{1}{2}\norm{V}_\infty\mu+C\mu-\frac{a_0s_0^{Np/2-N}}{8p}\norm{v}_p^p<0.
\end{aligned}
\]
Thus, by $\abs{w_1}_V^2\in (0,\alpha_0/8)$ and $\abs{w_2}^2_V> \alpha_0$, for all $\tau \in [\frac{1}{2},1]$, 
\eqref{lweonBa} gives
\[
c_\tau=\inf_{\gamma\in\Gamma}\max_{t\in[0,1]}E_{V,\tau}(\gamma(t)) \geq \frac{3\alpha_0}{8} > \frac{\alpha_0}{4} \geq\max\{E_{V,\tau}(w_{1}),E_{V,\tau}(w_{2})\}.
\]
This completes the proof. 
\end{proof}

\subsection{For every $\mu>0$}
\label{subsec:mp-all}

\begin{lemma}\label{lemmpe}
 Let $N \geq 1$ and assume that \eqref{A1}, \eqref{A3} and \eqref{A5} hold. 
 Then, for every $\mu >0$, there exist $\alpha_0>0$, $w_1, w_2 \in S_\mu$ such that $w_1,w_2\geq 0$ and 
\begin{equation*}
    c_\tau = c_\tau(\mu) \coloneq \inf_{\gamma \in \Gamma} \max_{t \in [0,1]} E_{V,\tau}(\gamma(t)) 
    \geq \frac{1}{4} \alpha_0 > 
    \max\set{E_{V,\tau}(w_1), E_{V,\tau}(w_2), 0} \quad \text{for all $\tau \in \ab[\frac{1}{2}, 1 ]$},
\end{equation*}
where
\begin{equation*}
    \Gamma \coloneq \Set{\gamma \in C([0,1], S_\mu) : \gamma(0) = w_1, \, \gamma(1) = w_2}.
\end{equation*} 
Furthermore, $[\frac{1}{2}, 1] \ni \tau \mapsto c_\tau \in (0,\infty)$ is nonincreasing. 
\end{lemma}

\begin{proof}
The monotonicity is an easy consequence of \eqref{mono-Evtau}. 
Fix $\mu > 0$ and set $B_\alpha \coloneq \set{u \in S_\mu: \abs{u}_V^2=\alpha}$. By \eqref{A1}, \eqref{A3}, \eqref{A5} and \eqref{eqdeff1f2}, 
there exists $C>0$ such that 
\begin{equation}\label{eqa5f2}
  \abs{F_1(t)} \leq C \abs{t}^{p}, \quad \abs{F_2(t)}\leq C\abs{t}^{q_3}  \quad \text{for each $t \in \R$}. 
\end{equation}
Then, by \cref{lemgn}, for any $u \in B_\alpha$,
\[
\begin{aligned}
    E_{V,\tau}(u) &\geq 
    \frac{1}{2}\abs{u}_V^2-CC_{q_3,N,V} \|u\|_{2}^{(1 - \beta_{q_3})q_3} \abs{ u}_{V}^{\beta_{q_3} q_3}-CC_{p,N,V} \|u\|_{2}^{(1 - \beta_p)p} \abs{ u}_{V}^{\beta_p p}\\
    &= \frac{1}{2}\alpha-CC_{q_3,N,V} \mu^{(1 - \beta_{q_3})q_3/2} \alpha^{\beta_{q_3} q_3/2}-CC_{p,N,V} \mu^{(1 - \beta_p)p/2} \alpha^{\beta_p p/2}.
\end{aligned}
\]
Therefore, since $2+\frac{4}{N}<q_3<p<2^*$, there exists $\alpha_0>0$ such that
\[
E_{V,\tau}(u) \geq \frac{\alpha_0}{4} \quad \text{for each $u \in B_{\alpha_0}$ and $\tau \in \ab[ \frac{1}{2}, 1 ]$}.
\]
By \eqref{inf-spect}, \eqref{eqa5f2} and \cref{lemgn}, there exists $w_1 \in S_\mu$ satisfying $w_1 \geq 0$ and $\abs{w_1}_V^2 \in (0,\alpha_0/8)$ small enough 
such that for all $\tau \in [\frac{1}{2},1]$, 
\[
\begin{aligned}
    E_{V,\tau}(w_1) &\leq \frac{1}{2}\abs{w_1}_V^2+CC_{q_3,N,V} \|w_1\|_{2}^{(1 - \beta_{q_3})q_3} \abs{ w_1}_{V}^{\beta_{q_3} q_3} 
    +CC_{p,N,V} \|w_1\|_{2}^{(1 - \beta_p)p} \abs{ w_1}_{V}^{\beta_p p}\\
    &\leq
    \frac{\alpha_0}{16} 
    + CC_{q_3,N,V} \mu^{(1 - \beta_{q_3})q_3/2} \abs{ w_1}_{V}^{\beta_{q_3} q_3} 
    +  CC_{p,N,V} \mu^{(1 - \beta_p)p/2} \abs{ w_1}_{V}^{\beta_p p} \leq \frac{\alpha_0}{8}.
\end{aligned}
\]
Now, by defining $w_2$ as in \cref{lemmps}, we complete the proof of \cref{lemmpe}.
\end{proof}

\subsection{Nonexistence of positive critical points of $E_{V,\tau}$ with small $L^2$-norm}\label{subsec33}

By \cref{lemmps}, under \eqref{A1} and \eqref{A3}, we notice that 
\begin{equation}\label{lwb-mp}
\liminf_{\mu \to 0^+} c_\tau(\mu) \geq \liminf_{\mu \to 0^+} c_{1} (\mu) = +\infty \quad \text{for all $\tau \in \ab[\frac{1}{2},1]$}. 
\end{equation}
Now we state the nonexistence result of positive critical points of $E_{V,\tau}$ with small $L^2$ norm, 
which plays a crucial role in the proof of  \cref{th1} for $N \geq 5$.
	\begin{theorem}\label{T:Liouv-musmall}
		Suppose $N \geq 1$, \eqref{A1} and \eqref{A3}. 
		For any $\varepsilon > 0$, there exists $\mu_\varepsilon > 0$ such that for each $\tau \in [1/2,1]$, the following problem 
		\[
		-\Delta u + Vu = \chi_{\Omega} \ab( \tau f_1(u) + f_2(u) ),\quad  u>0 \quad \text{in} \ \R^N, \quad \Vab{u}_2^2 \leq \mu_\varepsilon, 
		\quad E_{V,\tau} (u) \geq \varepsilon
		\]
		admits no nontrivial solution. 
		In particular, there exists $\mu_0 \in (0,\mu^*)$ such that 
		for each $\tau \in [1/2,1]$, there is no critical point $u$ of $E_{V,\tau}$ with $u>0$ in $\R^N$, 
		$\Vab{u}_2^2 \leq \mu_0$ and $E_{V,\tau} (u) \geq \inf_{0<\mu \leq \mu_0}c_\tau( \mu )$. 
	\end{theorem}

Since we use a blow-up argument to prove \cref{T:Liouv-musmall},
the detailed proof is provided in \cref{sec:blow-up}.

 \section{Compactness of bounded Palais-Smale Sequences}
 \label{sec:comp-bPS}

 The aim of this section is to study the behavior of bounded Palais--Smale sequences of $E_{V,\tau}|_{S_\mu}$. 
 Throughout this section, we always assume   that $\Omega\subset  \R^N$ is a (nonempty) bounded open set with smooth boundary $\partial \Omega$, and that \eqref{A1} and \eqref{A3}  hold. Before proceeding, we prepare some notation. 
 For $u \in S_\mu$ and $E'_{V,\tau} (u) \in H^{-1} (\R^N)$, set 
 \[
 \norm{E'_{V,\tau}(u)}_{ T_u^*S_\mu } \coloneq 
\sup_{ v \in T_uS_\mu, \norm{v}_{H^1} \leq 1 } \abs{E_{V,\tau} '(u) v} 
 = 
 \inf_{\lambda \in \R} \norm{ E_{V,\tau}'(u) + \lambda (u , \cdot)_2 }_{H^{-1}},
 \]
 where $T_uS_\mu$ denotes the tangent space of $S_\mu$ at $u \in S_\mu$ and is given by 
 \[
 T_uS_\mu = \Set{ v \in H^1(\R^N) : (u, v)_2 = 0 }. 
 \] 
 Let $\mu>0$, $\tau \in [\frac{1}{2},1]$ and $\{u_n\} \subset S_\mu$ be a bounded Palais-Smale sequence for $E_{V,\tau}$ constrained on $S_\mu$, at some level $c\in\R$, that is, 
 \begin{align}
 	\{u_n\} \text{ is bounded in } H^1(\R^N), \label{PS0}
 	\\
 	E_{V,\tau}(u_n) \to c, \label{PS1}
 	\\
 	\norm{ E_{V,\tau}'(u_n) }_{ T^*_{u_n} S_\mu } \to 0. 
 	\nonumber
 \end{align}
Note that the last condition is equivalent to 
\begin{equation}\label{PS2}
E^{ \prime}_{V,\tau}(u_n) + \lambda_n (u_n, \cdot)_2 \to 0 \quad \text{in }H^{-1}(\R^N)
\end{equation}
for some $\{\lambda_n\} \subset \R$.

Since \(\{u_n\} \subset H^1(\R^N)\) is bounded and $\Omega \subset \R^N$ is bounded,  we may assume that, up to a subsequence, there exists \( u_\tau \in H^1(\R^N) \) such that 
\begin{equation}\label{eqweakconverge}
\begin{aligned}
u_n &\rightharpoonup  u_\tau \quad \text{in } H^1(\R^N),\\
u_n &\to u_\tau \quad \text{in } L^p(\Omega) \text{ for all } p\in [2,2^*),\\
u_n &\to u_\tau \quad \text{a.e. in } \R^N. 
\end{aligned}
\end{equation}
Observe also that, from \eqref{PS2} and the boundedness of $\{u_n\} \subset H^1(\R^N)$, we obtain
\[
   o(1)= E_{V,\tau}'(u_n )u_n+\lambda_n(u_n,u_n)_2 =E_{V,\tau}'(u_n )u_n+\lambda_n \mu,
\]
which implies that $\{\lambda_n\}$ is bounded. Thus, up to a subsequence, there exists \(\lambda_\tau \in \mathbb{R}\) such that 
\begin{equation}\label{eqlambdatau}
    \lambda_n \rightarrow \lambda_\tau \quad \text{as }n \to +\infty.
\end{equation}
It is clear that $u_{\tau}\in H^{1}(\R^N)$ solves 
\begin{equation}\label{eqsolutionutau}
    -\Delta u_\tau+V(x)u_\tau+\lambda_\tau u_\tau=\tau \chi_\Omega f_1(u_\tau)+\chi_\Omega f_2(u_\tau).
\end{equation}

Next, we focus on proving that $u_{n} \to u_\tau$ in $H^1(\R^N)$ as $n \to +\infty$ provided $\lambda_{\tau} >0$, 
which ensures that the limit \(u_{\tau}\) satisfies the mass constraint \(\Vab{u_\tau}_2^2 = \mu\) (cf. \cite[Proposition 5.1]{CGJT}). 

\begin{lemma}\label{lemurhoconverge}
Let $N \geq 1$, $\mu>0$ and $\tau \in [\frac{1}{2},1]$. 
Suppose \eqref{A1}, \eqref{A3}, \eqref{PS0}, \eqref{PS1}, \eqref{PS2}, \eqref{eqweakconverge} and \eqref{eqlambdatau}. 
If $\lambda_\tau \geq 0$, then
\begin{equation}\label{strong conv}
\int_{\R^N}|\nabla (u_{n}-u_{\tau})|^{2}+ \left(V(x)+\lambda_{\tau}\right)|u_{n}-u_{\tau}|^{2}\,dx\to 0 \quad \text{as }n \to \infty.
\end{equation}
In particular, under \(\lambda_{\tau}>0\), the sequence \(\{u_{n}\}\) converges strongly to $u_{\tau}$ in \(H^{1}(\R^N)\).
\end{lemma}

\begin{proof}
Using \eqref{PS2} and \eqref{eqlambdatau}, for every \(\eta \in H^1(\R^N)\), we have, as $n \to+\infty$,
\begin{equation*}\label{eqsolution u2}
\begin{aligned}
o(1)\|\eta\|_{H^1}&=\ab(E^{ \prime}_{V,\tau}(u_{n})+\lambda_{n}(u_{n},\cdot)_2)\eta\\
&=\int_{\R^N}\nabla u_{n}\nabla\eta\,dx+\int_{\R^N}\left(V(x)+\lambda_{\tau}\right)u_{n}\eta\,dx-\int_{\Omega}\left(\tau f_1(u_n)+f_2(u_n) \right)\eta\,dx .
\end{aligned}
\end{equation*}
Since \(u_\tau\) is a solution to \eqref{eqsolutionutau}, choosing \(\eta=\eta_{n}:=u_{n}-u_{\tau}\), by \eqref{eqweakconverge}, we obtain, as $n \to+\infty$,
\begin{align*}
o(1) =&\int_{\R^N}\nabla (u_{n}-u_\tau)\nabla
\eta_n\,dx+\int_{\R^N}\left(V(x)+\lambda_{\tau}\right)(u_{n}-u_\tau)\eta_n\,dx\\
&-\int_{\Omega}\left(\tau f_1(u_n)+f_2(u_n)-\tau f_1(u_\tau)-f_2(u_\tau)\right)\eta_n\,dx\\
=&\int_{\R^N}|\nabla (u_{n}-u_{\tau})|^{2}+ \left(V(x)+\lambda_{\tau}\right)|u_{n}-u_{\tau}|^{2}\,dx+o(1),
\end{align*}
which yields \eqref{strong conv}. 
When $\lambda_\tau>0$, since $\norm{u}_\tau \coloneq \ab(\int_{\R^N}\abs{
\nabla u}^2\,dx+\int_{\R^N}\left(V(x)+\lambda_{\tau}\right)\abs{u}^2\,dx )^{1/2}$ defines an equivalent norm on $H^1(\R^N)$, 
see, e.g., \cite[\S3.3]{SW}, we conclude that $u_n\to u_\tau$ in $H^1(\R^N)$. 
\end{proof}
We end this section by the following lemma, which is analogous to \cite[Lemma 5.2]{CGJT} and 
useful to prove $\lambda_\tau > 0$ via \cref{lemliou}. 

\begin{lemma}\label{lem42}
Let $N \geq 1$, $\mu>0$ and $\tau \in [\frac{1}{2},1]$. Suppose that \eqref{A1}--\eqref{A3} hold and 
$\sup_{n \in \N} \norm{u_n}_{H^1} \leq M$. 
Then for every \(d \in \N\) and \( \delta > 0 \), 
there exist a subspace \(Y_{d,\delta}\subset H^{1}(\R^N)\) and \(C_{\delta} > 0\) such that \(\dim Y_{d,\delta} =d\) and
\begin{equation}\label{eqlem42}
 E''_{V,\tau}(u_n)[w, w] + \lambda \|w\|_2^2 \leq - C_\delta  \|w\|_{H^1}^{2} \quad \text{for all \(\lambda \leq -\delta\) and \(w \in Y_{d,\delta}\)}. 
\end{equation}

\end{lemma}
\begin{proof}
Let $d \in \N$ and $\delta > 0$ be given. 
Since $0 = \inf \sigma (-\Delta + V)$, we may find $\varphi_\delta \in C^\infty_c(\R^N) $ such that 
\[
\| \varphi_\delta \|_2 = 1, \quad | \varphi_\delta |_V^2 < \frac{\delta}{3}. 
\]
From the compactness of $\supp \varphi_\delta$, there exist $\alpha_1,\dots,\alpha_{d} \in \Z^N$ such that 
\[
\begin{aligned}
   &\supp \varphi_\delta( \cdot - \alpha_{i}) \cap \Omega = \emptyset \quad \text{for every $1 \leq i \leq d$},\\
	&\supp \varphi_\delta ( \cdot - \alpha_{i}) \cap \supp \varphi_\delta (\cdot - \alpha_{j}) = \emptyset \quad \text{for every $1 \leq i < j \leq d$}.
\end{aligned}
\]
For any $\varphi_{\delta,i}(x) \coloneq \varphi_\delta(x - \alpha_{i})$, the periodicity of $V$ leads to 
\[
| \varphi_{\delta,i} |_{V}^2 = | \varphi_\delta |_V^2 < \frac{\delta}{3}.
\]
Choose $\tilde{C}_\delta > 0$ so that 
\[
\tilde{C}_\delta \Vab{\varphi_\delta}_{H^1}^2 \leq \Vab{\varphi_\delta}_2^2. 
\]
By \(\supp \varphi_{\delta,i} \cap \Omega = \emptyset\) and \( \norm{\varphi_\delta}_2 = 1 \), for any $\lambda \leq - \delta$, 
\[
\begin{aligned}
	\int_{\R^N} | \nabla \varphi_{\delta, i} |^2 + V(x) \varphi_{\delta,i}^2 + \lambda \varphi_{\delta,i}^2 \,dx- 
	\int_{\Omega} \ab[ f_2'(u_n) + \tau f_1'(u_n) ] \varphi_{\delta,i}^2 \, dx
	&= 
	|\varphi_{\delta,i}|_V^2 + \lambda \| \varphi_{\delta,i} \|_2^2 
	\\
	&<  - \frac{2\delta}{3} \norm{\varphi_\delta}_2^2 \leq - \frac{2\tilde{C}_\delta}{3} \delta \Vab{\varphi_\delta}_{H^1}^2  . 
\end{aligned}
\]
Set \(Y_{d,\delta}\coloneq \operatorname{span}  \set{ \varphi_{\delta,1},\dots, \varphi_{\delta,d} } \). 
It is easily seen that $\dim Y_{d,\delta} = d$ and \eqref{eqlem42} holds. 
\end{proof}


\section{Existence of MP solutions for a dense set}
\label{sec: Ex-MP-a.e.}

 This section is devoted to proving the existence results of the following $L^2$-normalized equation for almost every $\tau \in [1/2,1]$
\begin{equation}\tag{$P_{\mu,\tau}$}\label{eqVt}
    \begin{dcases}
    -\Delta u +V (x)u + \lambda u=\tau\chi_{{\Omega}}f_1(u)+\chi_{{\Omega}}f_2(u)\quad \text{in }\R^N,\\
    u>0 \quad \text{in \(\R^N\)}, \\
    \int_{\R^N}\abs{u}^2\, dx =\mu.
\end{dcases}
\end{equation}
We first recall the definition of the Morse index of solutions to \eqref{eqVt} and 
an approximate Morse index (see \cite[Definition 1.4]{BCJS}).

\begin{definition}\label{d:Mor-apMor}
	Let $u \in S_\mu$, $\lambda \in \R$ and $\theta \geq 0$. 
	\begin{enumerate}
		\item 
		The Morse index $m_{\lambda,\tau} (u)$ is defined by 
		\[
		m_{\lambda,\tau} (u) 
		\coloneq 
		\sup \Set{ \dim L :
			\begin{aligned}
				&\text{$L \subset H^1(\R^N)$ is a subspace such that for all $\varphi \in L \setminus \set{0}$}
				\\
				& Q_{\lambda,\tau,u} (\varphi) \coloneq \int_{\R^N} \abs{\nabla \varphi}^2 + (V + \lambda) \varphi^2 
				- \chi_\Omega \ab( \tau f_1'(u) + f_2'(u) ) \varphi^2 \, dx  < 0
			\end{aligned}
		}.
		\]
		
		\item 
    An approximate Morse index $\tilde{m}_{\tau,\theta} (u)$ is defined by 
		\[
		\tilde{m}_{\tau,\theta}(u) 
		\coloneq 
		\sup \Set{ \dim L :
			\begin{aligned}
				&\text{$L \subset T_u S_\mu$ is a subspace such that for all $\varphi \in L \setminus \set{0}$}
				\\
				& E_{V,\tau}''(u) [\varphi,\varphi] - \frac{E_{V,\tau}'(u) u}{\norm{u}_2^2} \ab( \varphi, \varphi )_2 < - \theta \norm{\varphi}_{H^1}^2
			\end{aligned}
		}.
		\]
	\end{enumerate}
\end{definition}

Now, we state our main results in this section. 
 \begin{theorem}\label{th51}
     Let $N \geq 1$ and $\Omega\subset  \R^N$ be a (nonempty) bounded open set with smooth boundary $\partial \Omega$, and assume \eqref{A1}--\eqref{A3}. 
     Recall the numbers $\mu_0$ and $c_\tau$ in \cref{T:Liouv-musmall} and \cref{lemmps}. 
     Then, for every $\mu\in (0,\mu_0)$, problem \eqref{eqVt} has a solution $(u_\tau,\lambda_\tau)\in H^1(\R^N) \times (0,+\infty)$ for almost every $\tau \in [1/2,1]$ 
     such that $E_{V,\tau}(u_\tau)=c_\tau$ and $\tilde{m}_{\tau,0}(u_\tau)\leq 1$. 
     Moreover, $m_{\lambda_\tau,\tau} (u_\tau) \leq 2$ holds.  
\end{theorem}

\begin{theorem}\label{th52}
	Let $ 1 \leq N \leq 4$  and $\Omega\subset  \R^N$ be a (nonempty) bounded open set with smooth boundary $\partial \Omega$, assume \eqref{A1}--\eqref{A3} and \eqref{A5}, and recall $c_\tau$ from \cref{lemmpe}. 
	Then, for every $\mu>0$, problem \eqref{eqVt} has a solution $(u_\tau,\lambda_\tau)\in H^1(\R^N) \times (0,+\infty)$ for almost every $\tau \in [1/2,1]$ 
	such that $E_{V,\tau}(u_\tau)=c_\tau$ and $\tilde{m}_{\tau,0}(u_\tau)\leq 1$. 
	Furthermore, $m_{\lambda_{\tau},\tau} (u_\tau) \leq 2$ holds. 
\end{theorem}

\begin{proof}[Proofs of \cref{th51,th52}]
We prove \cref{th51,th52} simultaneously and for the sake of brevity, 
if $N=1,2,3,4$ and \eqref{A5} holds, then we write $\mu_0=+\infty$.

By \eqref{A2} and the definitions of $f_1$ and $f_2$ given in \eqref{eqdeff1f2}, it is clear that $f_1$ and $f_2$ satisfy \eqref{A2} and $f_1(t)t \geq 0$ for all $t\in \R$. Set 
$\mathcal{F} \coloneq \Gamma$ and $B_{0} \coloneq \{w_1,w_2\}$ given by \cref{lemmps} or \cref{lemmpe}, and
\[
A(u) \coloneq \frac{1}{2}\abs{u}_V^2-\int_{\Omega}F_2(u)\,dx\text{ and }B(u) \coloneq\int_{\Omega}F_1(u)\,dx.
\]
From \eqref{eqdeff1f2} and \eqref{A3}, it follows that \[
\abs{F_2(t)} \leq C\abs{t}^2 \quad \text{for all $t \in \R$}.
\]
Then, for any $ u \in S_\mu$, we have $\int_\Omega\abs{F_2(u)}\,dx \leq C\mu$, hence, \( \set{ u \in S_\mu : A(u) \leq K } \) is bounded in \(H^1(\R^N)\) for every $K \geq 0$.
Therefore, the set \(G \coloneq \set{ u \in S_\mu : A(u) \leq K_1, B(u) \leq K_2 } \) is bounded in \(H^1(\R^N)\) for each $K_1,K_2 \geq 0$. 
Since $f_1$ and $f_2$ are odd due to \eqref{eqodd}, for all \(\mu \in (0,\mu_0)\), by \cref{lemmps,lemmpe}, 
\cite[Theorem 2.1 and Remark 1.8]{BCJS} are applicable 
for almost every $\tau \in [1/2,1]$ (see the proof of \cite[Theorem 1.10]{BCJS}, 
and note also that {\rm(i)} in the proof of \cite[Theorem 1.10]{BCJS} now follows from  the boundedness of $G$ and  \cite[Eq.(3.5)]{BCJS} 
as in the proof of \cite[Theorem 3.1]{BCJS}). 
Hence, there exist sequences $\{u_n\}\subset S_\mu$ and $\zeta_n\to0^+$ such that, as $n \to +\infty$,
\begin{enumerate}[label=\rm(\roman*)]
    \item $E_{V,\tau}(u_n) \to c_\tau>0$;
    \item \label{bps2}$\norm{E_{V,\tau}'(u_n) }_{ T_{u_n}^* S_{\mu} }\to 0$;
    \item \label{bps3} $u_n \geq 0$ and $\{u_n\}$ is bounded in $H^1(\R^N)$; 
    \item \label{bps4} $\tilde{m}_{\tau,\zeta_n}(u_n)\leq 1$.
\end{enumerate}
By the boundedness of $\{u_n\}$ in $H^1(\R^N)$ and \eqref{A2},  $\lambda_n \coloneq -\frac{1}{\mu}E'_{V,\tau}(u_n)u_n$ is bounded. 
Up to a subsequence, we may assume $u_n \rightharpoonup u_\tau$ weakly in $H^1(\R^N)$ and $\lambda_n \to \lambda_\tau \in \R$ without loss of generality. 
By \cite[Remarks 1.6 and 1.7]{BCJS}, (\hyperref[bps2]{ii}) and (\hyperref[bps4]{iv}) imply that 
\begin{equation*}
E^{ \prime}_{V,\tau}(u_n) + \lambda_n (u_n, \cdot)_2 \to 0 \quad \text{in }H^{-1}(\R^N),
\end{equation*}
and, if there exists a subspace $W_n \subset T_{u_n}S_\mu$ such that 
\begin{equation}\label{eqineqmorse}
    E''_{V,\tau}(u_n)[w,w]+\lambda_n(w,w)_2<-\zeta_n \norm{w}^2_{H^1} \quad \text{for all }w \in W_n\setminus\{0\},
\end{equation}
then $\dim W_n \leq 1$ holds. 
Observe that the codimension of $T_{u_n}S_\mu$ in $H^1(\R^N)$ is $1$. 
Thus, if inequality \eqref{eqineqmorse} holds for every $w \in W_n \setminus\{0\}$ for a subspace $W_n \subset H^1(\R^N)$, 
then $\dim W_n \leq 2$ holds. 
From this and \cref{lem42} we deduce $\lambda_{\tau} \geq 0$. 
Indeed, if $\lambda_\tau < 0$ holds, then 
by exploiting \cref{lem42} with $d=3$ and $\delta \coloneq \abs{\lambda_\tau}/2$, 
there exist a subspace $Y \subset H^1(\R^N)$ and $C>0$ such that $\dim Y=3$ and 
\[
E_{V,\tau}''(u_{n}) [w,w] + \lambda_{n} \norm{w}_2^2 \leq - \frac{C}{2} \norm{w}_{H^1}^2 
\quad \text{for all $w \in Y$ and sufficiently large $n$}. 
\]
However, this contradicts the above observation and $\lambda_\tau \geq 0$ holds.

Now, since \eqref{PS0}--\eqref{PS2} are satisfied, by \eqref{eqlambdatau} and \eqref{eqsolutionutau}, 
the weak limit $u_\tau$ of $\{u_n\}$ in $H^1(\R^N)$ solves
\begin{equation}\label{eqsec5solve}
    -\Delta u_\tau+V(x)u_\tau+\lambda_\tau u_\tau=\tau \chi_\Omega f_1(u_\tau)+\chi_\Omega f_2(u_\tau),
    \quad u_\tau \geq 0 \quad \text{in} \ \R^N.
\end{equation}
By $\lambda_\tau \geq 0$ and \cref{lemurhoconverge}, 
we have $\abs{u_n-u_\tau}_V\to 0$ as $n\to +\infty$. 
In addition, Fatou's lemma leads to $\norm{u_\tau}_2^2\leq \mu < \mu_0$. 
Since $\norm{u_n - u_\tau}_r \to 0$ for $r \in (2,2^*)$ due to \cref{lemgn}, we obtain
\[
E_{V,\tau}(u_\tau)=\lim_{n\to+\infty}E_{V,\tau}(u_n)=c_\tau>0, 
\]
which implies that $u_\tau\not \equiv 0$. 
By \eqref{eqsec5solve} and the fact $u_\tau \geq 0$, we see 
\[
-\Delta u_\tau + \bab{ V + \lambda_\tau - \tau \chi_\Omega \frac{f_1(u_\tau)}{u_\tau} - \chi_\Omega \frac{f_2(u_\tau)}{u_\tau}  }_+ u_\tau \geq 0 \quad \text{in} \ \R^N;
\]
hence, the strong maximum principle gives $u_\tau>0$ in $\R^N$. 
By \cref{lemliou} when $1 \leq N \leq 4$ and \cref{T:Liouv-musmall} when $N \geq 5$ and $\mu \in (0,\mu_0)$, 
we have $\lambda_\tau > 0$. 
Thus, \cref{lemurhoconverge} leads to $\Vab{u_n - u_\tau}_{H^1} \to 0$. 
Recalling the definition of $\lambda_n$ with $\lambda_n \to \lambda_\tau > 0$, $\tilde{m}_{\tau,\zeta_n} (u_n) \leq 1$, $\zeta_n \to 0^+$ 
and $\Vab{u_n - u_\tau}_{H^1} \to 0$, 
we find that $\tilde{m}_{\tau,0}(u_\tau) \leq  1$ (cf. \cite[Remark 2.6]{CGJT}). 
Since the codimension of $T_{u_\tau} S_\mu$ is $1$,  it is easily seen that $m_{\lambda_{\tau},\tau} (u_\tau) \leq 2$, 
which completes the proof of \cref{th51,th52}.
\end{proof}

 \section{Blow-up analysis}
 \label{sec:blow-up}
 
 In order to obtain the boundedness of $\{(u_\tau,\lambda_\tau)\}_\tau$ found in \cref{th51,th52}, 
 the aim of this section is to study the blow-up phenomena of solutions to the following equation:
 \begin{equation}\label{eqVnmu}
    \begin{cases}
    -\Delta u +V (x)u + \lambda u= \tau \chi_{\Omega} f_1(u) + \chi_{\Omega} f_2(u) \quad \text{in }\R^N,\\
    u>0 \quad \text{in} \ \R^N.\\
\end{cases}
\end{equation}
Throughout this section, \eqref{A1}--\eqref{A3} are always assumed 
and let $\Omega \subset \R^N$ be a (nonempty) bounded open set with smooth boundary.

We deal with sequences 
\(\{(u_{n},\lambda_{n})\} \subset H^1(\R^N)\times \R^+\) and \(\{\tau_n\} \subset [\frac{1}{2},1] \) such that 
\begin{equation}\label{equn}
	\tau_n \to 1^-, \quad 
 \{(u_{n},\lambda_{n})\}\subset H^1(\R^N)\times \R^+\text{ satisfies } \eqref{eqVnmu}, \lambda_n \to +\infty  \text{ and } m_{\lambda_n,\tau_n}(u_n)\leq \bar{k}
\end{equation}
for some $\bar{k} \in \mathbb{N}$ not depending on $n \in \mathbb{N}$. 
In \cref{sec:pf-main}, we shall prove that if \(\{(u_{\tau_n},\lambda_{\tau_n})\}\) in \cref{th51,th52} is unbounded with $\tau_n \to 1^-$, 
then \eqref{equn} is fulfilled.

By elliptic regularity, we know that $u_n \in W^{2,r}_{\rm loc}(\R^N)$ for all $r < +\infty$ and 
$u_n\in C^{1,\alpha}(\R^N)$ for all $\alpha \in (0,1)$. 
Thus, we can pick $P_n \in \R^N$ such that $u_n(P_n)=\norm{u_n}_\infty$.

\begin{lemma}\label{lem61}
    Let \eqref{A1}--\eqref{A3} and \eqref{equn} hold. Then, $P_n \in \bar\Omega$ for $n \in \mathbb{N}$ large enough, 
    $\norm{u_n}_\infty \to +\infty$ as $n \to+\infty$ and 
    \[
		\tau_n \frac{f_1(u_n(P_n)) }{u_n(P_n)} -\lambda_n -V(P_n) \geq 0 \quad \text{for all sufficiently large $n \in \mathbb{N}$}.
		\]
\end{lemma}

\begin{proof}
Assume $P_n \in \bar\Omega^c$.
By Bony's maximum principle \cite{Bony,Lions}, 
for every $v \in W^{2,r}_{\text{loc}}(\R^N)$ with $r>N$, and for any strict local maximum point $x_0$ of $v$, 
there exists $\varepsilon \in C([0,\infty) , [0,\infty)  )$ such that 
$\varepsilon(0) = 0$ and for each $\delta > 0$, $\abs{A_\delta} > 0$ holds where 
\[
A_\delta \coloneq \Set{ y \in B_\delta(x_0) :  \abs{\nabla v(y)} \leq \varepsilon(\delta), \ D^2v(y) \leq 0  }.
\]
We apply this for $v (x) \coloneq u_n(x) - \abs{x-P_n}^4$. Since $P_n \in \bar \Omega^c$ is a strict local maximum of $u_n - \vab{x-P_n}^4$ 
and $\bar \Omega^c$ is open, there exists $\{y_{n,k}\}_k \subset \R^N$ such that 
\[
\begin{aligned}
	&\vab{y_{n,k} - P_n} < \frac{1}{k}, \quad D^2 v(y_{n,k})  \leq 0, \quad \vab{\nabla v (y_{n,k})} \leq \varepsilon \ab( k^{-1} ), \quad 
	y_{n,k} \in \bar \Omega^c,
	\\
	&
	-\Delta u_n(y_{n,k}) + V(y_{n,k}) u_n(y_{n,k}) + \lambda_n u_n(y_{n,k}) 
	= 
	\chi_\Omega (y_{n,k}) \bab{ \tau_n f_1(u_n(y_{n,k})) + f_2(u_n(y_{n,k})) } = 0. 
\end{aligned}
\]
Since $-\Delta u_n(y_{n,k}) \geq O(k^{-2})$ and $u_n(y_{n,k}) \to u_n(P_n) > 0$ as $k \to \infty$, 
we have 
\[
O(k^{-2}) \leq -\lambda_n-V(y_{n,k}),
\]
which is not possible for $n \in \mathbb{N}$ large enough, since $V$ is bounded and $\lambda_n \to +\infty$ as $n \to +\infty$.

   Next, we prove that  $\norm{u_n}_\infty \to +\infty$ as $n \to+\infty$. 
Let $P_n \in \bar \Omega$. As in the above, due to Bony's maximum principle, 
there exists $\{y_{n,k}\} \subset \R^N$ such that 
\[
\begin{aligned}
O (k^{-2})  \leq - \frac{\Delta u_n(y_{n,k})}{u_n(y_{n,k})} 
= 
\chi_{ \bar{\Omega} } (y_{n,k}) \left\{ \frac{f_2(u_n(y_{n,k}))}{u_n(y_{n,k})} + \tau_n \frac{f_1(u_n(y_{n,k}))}{u_n(y_{n,k})} \right\} - V(y_{n,k}) - \lambda_n. 
\end{aligned}
\]
By $y_{n,k} \to P_n$ and taking $\limsup_{k \to \infty}$, we obtain 
\begin{equation*}\label{eqPnpomega}
    0\leq \max\Set{\tau_n \frac{f_1(u_n(P_n)) }{u_n(P_n)}+\frac{f_2(u_n(P_n))}{u_n(P_n)}-\lambda_n -V(P_n), \ -\lambda_n -V(P_n) }.
\end{equation*}
Since $\lambda_n \to+\infty$ as $n \to +\infty$, for $n \in \mathbb{N}$ large enough, we have $-\lambda_n - V(P_n) < 0$ and  
\[
\tau_n \frac{f_1(u_n(P_n)) }{u_n(P_n)}+\frac{f_2(u_n(P_n))}{u_n(P_n)}-\lambda_n -V(P_n) \geq 0,
\]
which implies that $u_n(P_n) \to +\infty$ as $n \to +\infty$ and $f_2(u_n(P_n)) = 0$. This completes the proof. 
\end{proof}

\begin{remark}\label{r:loc-max}
	If $x_n \in \R^N$ is any local maximum point of $u_n$, then as above, for $n\in \mathbb{N}$ large enough,
	we may verify $x_n \in \bar \Omega$, $u_n(x_n) \to +\infty$ and 
	\[
	\tau_n \frac{ f_1(u_n(x_n)) }{u_n(x_n)} - \lambda_n - V(x_n) \geq 0. 
	\]
\end{remark}

We now perform a blow-up analysis for this type of sequence, in the spirit of \cite{EP,PV} 
as in \cite{BCJS23,CJS24,CDCGJT,CRZ}. 

\begin{lemma}\label{lem62}
    Let \eqref{A1}--\eqref{A4} and \eqref{equn} hold. Let $P_n \in \R^N$ be such that, for some $R_n \to \infty$,
\[
    u_n(P_n) =\max_{B_{R_n \tilde{\varepsilon}_n}(P_n)} u_n(x) \quad \text{where} \quad 
    \tilde{\varepsilon}_n \coloneq u_n(P_n)^{ - (p-2)/2 } \to 0. 
\]
Set $\varepsilon_n \coloneq \lambda_n^{-\frac{1}{2}}$. Then
\begin{equation}\label{eqlem621}
    \left( \frac{\tilde{\varepsilon}_n}{\varepsilon_n} \right)^2 \to \tilde{\lambda} \in (0, a_0],
\end{equation}
where $a_0$ is the number in \eqref{A3}. Moreover, 
\begin{equation}\label{eqlem622}
    \limsup_{n \to +\infty} \frac{\operatorname{dist}(P_n, \partial\Omega)}{\tilde{\varepsilon}_n} = +\infty
\end{equation}
holds, and passing to a subsequence if necessary, we have
\begin{enumerate}[label={\rm (\roman*)}]
    \item 
    $P_n \in \Omega$;
    \item 
    the scaled sequence
    \begin{equation*}\label{eqlem623}
        U_n(x) \coloneq a_0^{\frac{1}{p-2}} \varepsilon_n^{\frac{2}{p-2}} u_n(\varepsilon_n x + P_n) 
        = a_0^{ \frac{1}{p-2} } \lambda_n^{ - \frac{1}{p-2} } u_n ( \lambda_n^{ - \frac{1}{2} } x + P_n )
    \end{equation*}
    converges to some $U \in H^1(\mathbb{R}^N)$ in $C^1_{\rm loc}(\mathbb{R}^N)$, where $U$ is a unique positive solution to 
    \begin{equation}\label{limeq}
        \begin{dcases}
            -\Delta U + U = |U|^{p-2}U,\quad U>0 & \text{in } \mathbb{R}^N, \\
            U(0) = \max_{x \in \mathbb{R}^N} U, & \\
            U(x) \to 0 \quad \text{as } |x| \to +\infty;
        \end{dcases}
    \end{equation}
    \item
    there exists $\phi_n \in C_0^\infty(\R^N)$, with $\operatorname{supp}\phi_n \subset B_{R\varepsilon_n}(P_n)$ for some $R > 0$, 
    such that $Q_{\lambda_n, \tau_n,u_n}(\phi_n ) < 0$;
    \item
    for all $R > 0$ and $q \ge 1$,
    \[
        \lim_{n \to \infty} \lambda_n^{\frac{N}{2} - \frac{q}{p-2}} \int_{B_{R\varepsilon_n}(P_n)} u_n^q \, dx 
        = \lim_{n \to \infty} a_0^{ - \frac{q}{p-2} } \int_{B_R(0)} U_n^q \, dy = a_0^{ - \frac{q}{p-2} } \int_{B_R(0)} U^q\, dy.
    \]
\end{enumerate}
%
\end{lemma}

\begin{proof}
By \cref{r:loc-max}, we deduce that $P_n \in \bar\Omega$ for $n \in \mathbb{N}$ large enough, $u_n(P_n) \to +\infty$ as $n \to+\infty$ and 
\[
\tau_n \frac{f_1(u_n(P_n)) }{u_n(P_n)} -\lambda_n -V(P_n) \geq 0 \quad  \text{for all large $n \in \mathbb{N}$}.
\]
To prove \eqref{eqlem621}, when $n \in \N$ is large enough, 
by \eqref{A3} and $\tau_n \to 1^-$, 
\[
0 \leq \frac{\lambda_n}{u_n(P_n)^{p-2}} \leq o(1) + \tau_n \frac{ f_1(u_n(P_n)) }{ u_n(P_n)^{p-1} } = o(1) + a_0. 
\]
Thus, up to a subsequence, we have
\begin{equation}\label{eqlun}
        \tilde{\varepsilon}_n^2 \lambda_n = \frac{\lambda_n}{u_n(P_n)^{p-2}} \to \tilde{\lambda} \in [0,a_0].
\end{equation}

Next, we prove $\tilde{\lambda}>0$, which implies \eqref{eqlem621}.  For this purpose, define the rescaled function
\[
\tilde{U}_n(y) \coloneq \tilde{\varepsilon}_n^{\frac{2}{p-2}}u_n(\tilde{\varepsilon}_ny+P_n) 
= \frac{1}{u_n(P_n)} u_n ( \tilde{\varepsilon}_n y + P_n ). 
\]
It is immediate to check that 
\begin{equation*}\label{eqscale}
\begin{dcases}
\begin{aligned}
            &-\Delta \tilde{U}_{n} +\tilde{\varepsilon}_n^2V (\tilde{\varepsilon}_nx+P_n)\tilde{U}_{n} + \tilde{\varepsilon}_n^2\lambda_{n} \tilde{U}_{n}\\
            &\quad = 
            \tilde{\varepsilon}_n^{2(p-1)/(p-2)}\chi_\Omega (\tilde{\varepsilon}_nx+P_n) 
            \ab[ f_2 \ab( \tilde{\varepsilon}_n^{-2/(p-2)} \tilde{U}_n )+\tau_n  f_1 \ab( \tilde{\varepsilon}_n^{-2/(p-2)} \tilde{U}_n ) ] \quad \text{in }\R^N,
            \end{aligned}
        \\ 
        \tilde{U}_n(0)=1,\quad 0<\tilde{U}_n(y)\leq 1 \quad \text{for }\abs{y}\leq R_n.
\end{dcases}
\end{equation*}
By elliptic regularity, up to a subsequence, there exists $\tilde{U} \in W^{2,r}_{\rm loc} (\R^N)$ for any $r<+\infty$ 
such that $\tilde{U}_n \to \tilde{U}$ in $C^1_\text{loc}(\R^N)$. For every $x \in \R^N$ such that $\tilde{U}(x)>0$, we have 
\[
\lim_{n \to \infty}\frac{u_n(\tilde{\varepsilon}_nx+P_n)}{u_n(P_n)}=\lim_{n \to \infty}\tilde{U}_n(x)=\tilde{U}(x)>0;
\]
hence, by \eqref{A3}, as $n \to +\infty$, 
\begin{equation*}
\label{eqlem21f10}
\begin{aligned}
    0 \leq \tilde{\varepsilon}_n^{2(p-1)/(p-2)}f_1\ab(\tilde{\varepsilon}_n^{-2/(p-2)} \tilde{U}_n (x))
    =
    \frac{f_1\left(u_n(\tilde{\varepsilon}_nx+P_n) \right)}{u_n(\tilde{\varepsilon}_nx+P_n)^{p-1}} \cdot \frac{u_n(\tilde{\varepsilon}_nx+P_n)^{p-1}}{u_n(P_n)^{p-1}}
    \to  a_0 \tilde{U}(x)^{p-1}.
\end{aligned}
\end{equation*}
Similarly, for every $x \in \R^N$ such that $\tilde{U}(x)=0$, 
we obtain 
\[
\lim_{n \to \infty}\frac{u_n(\tilde{\varepsilon}_nx+P_n)}{u_n(P_n)}=0 = \tilde{U}(x).
\]
Hence, by \eqref{A3}, as $n \to +\infty$, 
\[
\begin{aligned}
	0 \leq \tilde{\varepsilon}_n^{2(p-1)/(p-2)} f_1\ab(\tilde{\varepsilon}_n^{-2/(p-2)} \tilde{U}_n (x) )
   = \frac{f_1\left(u_n(\tilde{\varepsilon}_nx+P_n) \right)}{u_n(\tilde{\varepsilon}_nx+P_n)^{p-1}} \cdot \frac{u_n(\tilde{\varepsilon}_nx+P_n)^{p-1}}{u_n(P_n)^{p-1}}
   \leq C\frac{u_n(\tilde{\varepsilon}_nx+P_n)^{p-1}}{u_n(P_n)^{p-1} } \to 0. 
    \end{aligned}
\]
In conclusion, for every $x \in \R^N$, we have 
\begin{equation}\label{eqlem21f1}
    \tau_n \tilde{\varepsilon}_n^{2(p-1)/(p-2)} f_1\ab(\tilde{\varepsilon}_n^{-2/(p-2)} \tilde{U}_n (x)) \to a_0 \tilde{U}(x)^{p-1} \quad \text{as }n \to +\infty.
\end{equation}
Moreover, it follows from \eqref{A1}, \eqref{A3} and \eqref{eqdeff1f2} that, for every $x \in\R^N$, 
\begin{equation}
\label{eqlem21f12}
0 \leq \tilde{\varepsilon}_n^{2(p-1)/(p-2)}f_1\ab( \tilde{\varepsilon}_n^{-2/(p-2)} \tilde{U}_n (x) )  \leq C \tilde{U}_n (x)^{p-1},
\end{equation}
and
\begin{equation}\label{eqlem21f2}
   \begin{aligned}
    \abs{\tilde{\varepsilon}_n^{2(p-1)/(p-2)}f_2\ab( \tilde{\varepsilon}_n^{-2/(p-2)} \tilde{U}_n (x) ) } \leq C\tilde{\varepsilon}_n^{ \frac{2(p-1)}{p-2} }.
\end{aligned} 
\end{equation}

Let $\tilde{\Omega}_n \coloneq\frac{\Omega-P_n}{\tilde{\varepsilon}_n}$.
Since $\tilde{U}_n \to \tilde{U}$ in $C^1_\text{loc}(\R^N)$, by \eqref{eqlem21f1}--\eqref{eqlem21f2} and Lebesgue's dominated convergence theorem, 
we conclude that, up to a subsequence, for every $\varphi \in C^\infty_c(\R^N)$,
\[
\begin{aligned}
    \tilde{\varepsilon}_n^{2(p-1)/(p-2)}\int_{\tilde{\Omega}_n} \tau_nf_1  \ab( \tilde{\varepsilon}_n^{-2/(p-2)} \tilde{U}_n )\varphi 
    + f_2 \ab( \tilde{\varepsilon}_n^{-2/(p-2)} \tilde{U}_n)\varphi\,dx 
    \to \int_{D} a_0 \tilde{U}^{p-1} \varphi\,dx\quad \text{as }n \to +\infty,
\end{aligned}
\]
where
\[
D \coloneq \begin{dcases}
    \R^N &\displaystyle\text{if }\lim_{n\to+\infty}\frac{\operatorname{dist}(P_n,\partial \Omega)}{\tilde{\varepsilon}_n}=+\infty,\\ 
    H &\displaystyle\text{if }\lim_{n\to+\infty}\frac{\operatorname{dist}(P_n,\partial \Omega)}{\tilde{\varepsilon}_n}<+\infty,
\end{dcases}
\]
and $H$ is a half-space. 
Thus, by the strong maximum principle together with elliptic regularity, $\tilde{U}$ satisfies
\begin{equation*}\label{eqtv}
\begin{dcases}
        -\Delta \tilde{U} + \tilde{\lambda} \tilde{U}=  a_0 \chi_D \tilde{U}^{p-1} ,\quad \tilde{U}>0 \quad  \text{in }\R^N,
        \\ \tilde{U}(0)=1=\norm{\tilde{U}}_\infty.
\end{dcases}
\end{equation*}
Therefore, when $D=\R^N$ (resp. $D=H$), $\tilde{\lambda} > 0$ holds thanks to the result of \cite{GS} (resp. \cref{lemnelambda+0}). 
Since $\tilde{\lambda} > 0$, \eqref{eqlun} leads to 
\[
\ab( \frac{ \tilde \varepsilon_n }{\varepsilon_n} )^2 = \frac{\lambda_n}{u_n(P_n)^{p-2}} \to \tilde \lambda \in (0,a_0]
\]
and \eqref{eqlem621} holds.

To proceed, we claim that $m_{\tilde \lambda , D , a_0}(\tilde{U})\leq \bar{k}$, where 
\[
m_{\tilde \lambda, D, a_0}(\tilde{U}) \coloneq 
\sup \Set{ \dim W \in \N : \begin{aligned}
		&\text{$W \subset C^1_0(\R^N)$ is a subspace such that for all $\varphi \in W \setminus \set{0}$}\\
		&Q_{\tilde \lambda, D, a_0} (\varphi) \coloneq 
		\int_{\R^N} \abs{\nabla \varphi}^2 + \tilde \lambda \varphi^2 - (p-1) a_0 \chi_D \tilde U^{p-2} \varphi^2 \, d x < 0
\end{aligned} 
}.
\]
Assume by contradiction that $m_{\tilde \lambda, D, a_0}(\tilde{U})\geq \bar{k}+1$. 
Then, there exist $\delta_0 > 0$ and $\varphi_1,\varphi_2,...,\varphi_{\bar{k}+1} \in C_c^\infty(\R^N)$, which are orthonormal  in $L^2(\R^N)$, such that 
\begin{equation}\label{neg-k+1}
	Q_{\tilde \lambda, D, a_0} (\varphi) \leq - \delta_0 \norm{\nabla \varphi}_{2}^2 \quad 
	\text{for each $\varphi \in W \coloneq \operatorname{span} \set{\varphi_1,\dots, \varphi_{\bar k + 1}}$},
\end{equation}
where we used the fact that every norm on $W$ is equivalent. 
For each $n \in \N$, set 
\[
\tilde{W}_n \coloneq \Set{ \tilde{\varepsilon}_n^{ - \frac{N-2}{2} } \varphi \ab( \tilde{\varepsilon}_n^{-1} (x-P_n) ) 
: \varphi \in W  }.
\]
For any $\psi_n(x) = \tilde{\varepsilon}_n^{ - \frac{N-2}{2} } \varphi(\frac{x-P_n}{\tilde{\varepsilon}_n}) \in \tilde{W}_n $ where $\varphi \in W$, we have 
\begin{equation}\label{eqlem21m1}
\begin{aligned}
   Q_{\lambda_n,\tau_n,u_n} (\psi_n) 
   &= 
   \int_{\R^N}\abs{\nabla \psi_n}^2+V(x) \psi_n^2+\lambda_n \psi_n^2-\chi_{\Omega}\ab[\tau_n f_1'(u_n)+f_2'(u_n) ]\psi_n^2\,dx\\
   &= 
   \int_{\R^N}\abs{\nabla \varphi}^2+\tilde{\varepsilon}_n^{2}V(\tilde{\varepsilon}_nx+P_n)\varphi^2+\tilde{\varepsilon}_n^{2}\lambda_n \varphi^2\\
    &\quad\qquad-\tilde{\varepsilon}_n^{2}\chi_{\tilde{\Omega}_n}\ab[\tau_nf_1'\ab(u_n(\tilde{\varepsilon}_n x+P_n))+ f_2'\ab(u_n(\tilde{\varepsilon}_n x+P_n))]\varphi^2\,dx.
\end{aligned}
\end{equation}
On one hand, by \eqref{A1} and the definition of $f_2$ given in \eqref{eqdeff1f2}, we have 
\begin{equation}\label{eqlem21m2}
\tilde{\varepsilon}_n^{2}\abs{f_2'\ab(u_n(\tilde{\varepsilon}_n x+P_n))}
=\tilde{\varepsilon}_n^{2}\abs{ f_2'\ab( \tilde{\varepsilon}_n^{-2/(p-2)}\tilde{U}_n(x) ) } \leq \tilde{\varepsilon}_n^{2}C.
\end{equation}
On the other hand, let $\delta \in (0,\delta_0)$ be so small that 
\begin{equation}\label{ch-delta}
	a_0 \delta \int_{\R^N} \tilde{U}^{p-2} \varphi^2 \, dx \leq \frac{\delta_0}{2} \norm{\nabla \varphi}_2^2 \quad \text{for every $\varphi \in W$}.
\end{equation}
By \eqref{A4} and definition of $f_1$ given in \eqref{eqdeff1f2}, there exist $M_\delta,C_\delta>0$ such that
\[
\abs{f'_1(t)}\leq C_\delta \quad \text{for all $t \in [-M_\delta,M_\delta]$}, 
\quad f'_1(t)\geq a_0(p-1-\delta) t^{p-2} \quad \text{for each $t >M_\delta$}.
\]
Then,
\begin{equation}\label{eqlem21hx1}
    \begin{aligned} 
&\tau_n \int_{ \tilde{\Omega}_n }\tilde{\varepsilon}_n^{2}f_1' \ab(\tilde{\varepsilon}_n^{-2/(p-2)}\tilde{U}_n(x) ) \varphi^2\, dx  
\\
\geq \ &
- \tilde{\varepsilon}_n^{2}C_\delta\int_{\tilde{\Omega}_n\cap \set{\tilde{\varepsilon}_n^{-2/(p-2)}\tilde{U}_n \leq M_\delta }} \varphi^2 \,dx
+ \tau_n a_0 (p-1-\delta)\int_{\tilde{\Omega}_n\cap \set{ \tilde{\varepsilon}_n^{-2/(p-2)}\tilde{U}_n> M_\delta } } \tilde{U}_n^{p-2}\varphi^2\,dx\\
\geq \ & 
-\tilde{\varepsilon}_n^{2}C_\delta\int_{\tilde{\Omega}_n}\varphi^2\,dx
+\tau_n a_0 (p-1-\delta)\int_{\tilde{\Omega}_n\cap \set{ \tilde{\varepsilon}_n^{-2/(p-2)}\tilde{U}_n > M_\delta }} \tilde{U}_n^{p-2}\varphi^2\,dx
.
\end{aligned}
\end{equation}
Letting
\[
h_n(x) \coloneq 
\chi_{\set{ \tilde{\varepsilon}_n^{-2/(p-2)}\tilde{U}_n > M_\delta } } \tilde{U}_n^{p-2}(x) \quad 
\text{and} \quad  \operatorname{supp}(W) \coloneq\bigcup_{j=1}^{j=\bar{k}+1}\operatorname{supp}(\varphi_j),
\]
since $\tilde{U}_n \to \tilde{U}$ in $C^1_\text{loc}(\R^N)$, by Lebesgue’s dominated convergence theorem, we have 
\begin{equation}
\label{eqlem21hx2}
\begin{aligned}
	 \max_{ w \in W, \norm{\nabla w}_{2} = 1 }
	\int_{ \tilde{\Omega}_n} \abs{h_n-\tilde{U}^{p-2}}w^2\, dx
\leq \max_{ w \in W, \norm{\nabla w}_{2} = 1  }\norm{w}_\infty^2\int_{\operatorname{supp}(W)}\abs{h_n-\tilde{U}^{p-2}}\, dx\to 0 \quad \text{as }n\to+\infty.
\end{aligned}
\end{equation}
Thanks to \eqref{eqlun} and \eqref{neg-k+1}--\eqref{eqlem21hx2}, for any $\varphi \in W$ with $\norm{\nabla \varphi}_2^2 = 1$, 
\[
Q_{ \lambda_n, \tau_n, u_n } (\psi_n) 
\leq Q_{ \tilde \lambda, D, a_0  } (\varphi) + \delta a_0 \int_{\R^N} \tilde{U}^{p-2} \varphi^2 \, dx + o(1)
\leq - \frac{\delta_0}{2} + o(1),
\]
where $o(1) \to 0$ as $n \to \infty$ uniformly with respect to $\varphi \in W$ with $\norm{\nabla \varphi}_2 = 1$. 
By $\dim \tilde W_n = \bar k + 1$, for sufficiently large $n$, we obtain $m_{\lambda_n,\tau_n} (u_n) \geq \bar k + 1$,  
which contradicts $m_{\lambda_n,\tau_n}(u_n)\leq \bar{k}$. Thus, we conclude $m_{\tilde \lambda, D, a_0}(\tilde{U})\leq \bar{k}$.

Notice that 
\[
U_n(x) = a_0^{ \frac{1}{p-2} } \ab( \frac{\varepsilon_n}{\tilde \varepsilon_n} )^{ \frac{2}{p-2} } 
\tilde U_n \ab( \frac{\varepsilon_n}{\tilde \varepsilon_n} x ) \to a_0^{\frac{1}{p-2}} \tilde \lambda^{ - \frac{1}{p-2} } 
\tilde U \ab( \frac{x}{\sqrt{\tilde \lambda}} ) \eqcolon U(x)
\]
in $W^{2,r}_{\rm loc} (\R^N)$ for each $r<\infty$ and $U$ satisfies 
\[
\left\{\begin{aligned}
	&-\Delta U + U = \chi_{\tilde{H} } U^{p-1} \quad \text{in} \ \R^N, \quad U > 0 \quad \text{in} \ \R^N, \\
	& U(0) = \max_{\R^N} U, \quad m_{ 1, \tilde{H}, 1 } (U) \leq \bar k,
\end{aligned}\right.
\]
where $\tilde{H} \coloneq \tilde{\lambda}^{1/2} D$. Note that $U$ is stable outside a compact set due to \cite[Remark 1]{Fa}. 
When $\tilde{H}$ is a half-space, 
\cref{lem25} gives a contradiction $U \equiv 0$. Therefore, $D=\R^N$, \eqref{eqlem622} and assertion (i) hold. 
Furthermore, 
\cref{R:entire} implies that $U \in H^1(\R^N)$ is a unique positive solution due to \cite{Kw89}. 
Thus, assertion (ii) holds.

Assertion (iv) is an easy consequence of (ii) and the change of variables. 
To prove assertion (iii), notice that $U \in H^1(\R^N)$ and 
\[
\int_{\R^N} \abs{\nabla U}^2 + U^2 - (p-1) U^p \, dx = - (p-2) \int_{\R^N} U^p \, dx < 0. 
\]
Since $U$ and $\tilde U$ are related by scaling and a scalar multiple, 
there exists $\delta_0 > 0$ and $\varphi \in C^\infty_0(\R^N)$ such that 
\[
\norm{\nabla \varphi}_2^2 = 1, \quad 
\int_{\R^N} \abs{\nabla \varphi}^2 + \tilde \lambda \varphi^2 - (p-1) a_0 \tilde{U}^{p-2} \varphi^2 \, dx \leq - \delta_0 
\]
As in the above, a function defined by 
$\phi_n (x) \coloneq \tilde \varepsilon_n^{ - \frac{N-2}{2} } \varphi ( \tilde \varepsilon_n^{-1} (x-P_n) ) \in C^\infty_0( \R^N )$ 
satisfies $\supp \phi_n \subset B_{R \tilde \varepsilon_n } (P_n) $ for some $R > 0$ and 
$Q_{\lambda_n,\tau_n,u_n} (\phi_n) < 0$. 
Hence, assertion (iii) also holds. 
\end{proof}

As in \cite[Theorem 3.2]{EP} (cf. \cite[\S2]{PV} and \cite[Lemma 4.4]{CRZ}), by repeating the argument in the proof of \cref{lem62}, 
we obtain the following: 
%
\begin{proposition}\label{prophk}
    Let \eqref{A1}--\eqref{A4} and \eqref{equn} hold. Then there exist $P^1_n,...,P^k_n$ with $k \leq \bar{k}$ such that 
    \[
    \lambda_n^{1/2}\abs{P^i_n-P^j_n} \to +\infty, \quad i,j=1,...,k, \quad i\neq j \quad \text{as }n \to +\infty
    \]
    and
    \[
    u_n(P_n^i)=\max_{B_{ R_n\lambda_n^{-1/2}}(P_n^i)} u_n, \quad i =1,...,k,
    \]
    for some $R_n \to +\infty$ as $n \to +\infty$. Moreover, 
    \[
    \lim_{R\to +\infty}h_k(R)=0,
    \]
    where $h_k(R)$ is given by 
    \[
    h_k(R) \coloneq \limsup_{n \to +\infty}\ab(\lambda_n^{-\frac{1}{p-2}}\max_{d_{n,k} (x) \geq R\lambda_n^{-1/2}} u_n(x) ) 
    \text{ and } 
    d_{n,k} (x) \coloneq \min \Set{ \abs{x-P_n^i} : i=1,...,k }.
    \]
\end{proposition}

Following the argument of \cite[Theorem 3.2]{EP} and \cite[Lemma 2.12]{PV} with \cref{prophk}, we conclude that 
\begin{proposition}\label{propbu}
     Suppose \eqref{A1}--\eqref{A4} and \eqref{equn}. Let $P^1_n,...,P^k_n \in \R^N$ be the points in \cref{prophk}, 
     and $U$ the unique solution to \eqref{limeq}. 
     Then there exist positive constants $C_0,\gamma$ such that 
     \begin{equation}\label{eqprop23e}
         \abs{u_n}\leq C_0\lambda_n^{\frac{1}{p-2}}\sum_{i=1}^ke^{-\gamma \sqrt{\lambda_n}\abs{x-P_n^i}} 
         \quad \text{for each $x \in \R^N$ and $n \in \mathbb{N}$}
     \end{equation}
     and for every $q \geq 1$,
     \[
     \lambda_n^{\frac{N}{2}-\frac{q}{p-2}}\int_{\R^N}\abs{u_n}^q\,dx\to a_0^{-\frac{q}{p-2}} k \int_{\R^N} \abs{U}^q\,dx \quad 
     \text{as }n \to +\infty.
     \]
\end{proposition}
\begin{proof}
By \eqref{A1} and \eqref{A3}, there exists $C_1> 0$ such that for every $t \in \R$, 
\begin{equation}\label{eqprop231}
\abs{f_1(t)}\leq C_1\abs{t}^{p-1}\quad \text{and}\quad \abs{f_2(t)}\leq C_1\abs{t} .
\end{equation}
Let 
\[
h_n(x) \coloneq \chi_\Omega(x)\frac{\tau_nf_1(u_n(x))+f_2(u_n(x))}{u_n(x)}. 
\]
By \cref{prophk}, if $R>0$ is large enough, then there exists $n(R)>0$ such that  for any $n \geq n(R)$,
\[
\lambda_n^{-\frac{1}{p-2}}\max_{\set{d_{n,k}(x)\geq R\lambda_n^{-1/2}}}\abs{u_n(x)}\leq \ab(\frac{1}{8C_1})^{\frac{1}{p-2}}
\quad \text{and} \quad \frac{\lambda_n}{8}\geq \norm{V}_\infty+C_1.
\]
Thus, due to \eqref{eqprop231}, in $\set{d_{n,k}\geq R\lambda_n^{-1/2}}$ for $n \geq n(R)$, we have
\begin{equation}\label{eqprop23a}
  \tilde{a}_n(x) \coloneq \frac{\lambda_n}{2}-\norm{V}_\infty-h_n(x) \geq \frac{\lambda_n}{2}-\norm{V}_\infty-C_1 -\frac{\lambda_n}{8} 
  \geq  \frac{\lambda_n}{4}.
\end{equation}
For $\phi^i_n(x) \coloneq e^{-\gamma\lambda_n^{1/2}\abs{x-P^i_n}}$ with $0<\gamma \leq 1/4$,  
in $\set{d_{n,k}\geq R\lambda_n^{-1/2}}$, we obtain, for $n\in \mathbb{N}$ large enough, 
\[
\begin{aligned} 
    -\Delta \phi^i_n + \frac{\lambda_n}{4} \phi^i_n
    =\lambda_n\phi^i_n\ab[-\gamma^2+(N-1)\frac{\gamma}{\lambda_n^{1/2}\abs{x-P^i_n}}+ \frac{1}{4} ]\geq 0 
    \quad  \text{in} \set{d_{n,k}\geq R\lambda_n^{-1/2}}.
\end{aligned}
\]
Since $U(x)\to 0$ as  $\abs{x}\to+\infty$,  for $R>0$ large enough, we infer
\[
\ab(e^{\gamma R}\phi^i_n -a_0^{\frac{1}{p-2}}\lambda_n^{-\frac{1}{p-2}}u_n )\Big|_{\partial B_{R\lambda_n^{-1/2}}(P^i_n)}\to 1- U(R)>0\quad \text{as }n\to+\infty.
\]
Then, letting $\phi_n \coloneq a_0^{-\frac{1}{p-2}}  e^{\gamma R}\lambda_n^{1/(p-2)}\sum^k_{i=1}\phi^i_n$ and $L_n \coloneq -\Delta+\lambda_n+V-h_n$, 
by \eqref{eqprop23a} and \( u_n,\phi_n > 0 \), for $n\in\mathbb{N}$ large enough, we have
\[
L_n(\phi_n-u_n) \geq 0\quad \text{in }\set{d_{n,k}\geq R\lambda_n^{-1/2}},
\]
and
\[
\phi_n-u_n \geq 0  \quad \text{on }\set{d_{n,k}= R\lambda_n^{-1/2}}.
\]
Then, by the maximum principle \cite[Theorem 8.1]{GT}, for $n\in\mathbb{N}$ large enough, we conclude that
\[
0<u_n\leq \phi_n
=
a_0^{-\frac{1}{p-2}}e^{\gamma R}\lambda_n^{1/(p-2)}\sum^k_{i=1}e^{-\gamma\lambda_n^{1/2}\abs{x-P^i_n}}\quad \text{in }\set{d_{n,k}\geq R\lambda_n^{-1/2}}.
\]
By \cref{lem62},  we may find $C>0$ so that for $n\in\mathbb{N}$ large enough and all $x \in \set{d_{n,k} \leq R \lambda_n^{-1/2}}$, 
\[
\begin{aligned}
0 \leq u_n(x) \leq \max_{\R^N}u_n &\leq \left(\frac{2\lambda_n}{\tilde{\lambda}}\right)^{1/(p-2)} 
\leq C\phi_n (x) =C a_0^{- \frac{1}{p-2}} e^{\gamma R}\lambda_n^{1/(p-2)}\sum^k_{i=1}e^{-\gamma\lambda_n^{1/2}\abs{x-P^i_n}},
\end{aligned}
\]
which proves \eqref{eqprop23e}. Now, by repeating the argument of \cite[Lemma 2.12]{PV}, we complete the proof of \cref{propbu}. 
\end{proof}

We end this section by providing the proof of \cref{T:Liouv-musmall}.
\begin{proof}[Proof of \cref{T:Liouv-musmall}]
The latter assertion follows from the former and \eqref{lwb-mp}. In fact, 
by \eqref{lwb-mp}, there exists $\mu_1 > 0$ such that 
\[
\inf \Set{c_\tau (\mu) :  0 < \mu \leq \mu_1, \  \frac{1}{2} \leq \tau \leq 1 } \geq 1 . 
\]
Thus, from the former assertion with $\varepsilon = 1$, 
there exists $\mu_0 \in(0,\mu_1]$ such that 
$E_{V,\tau}$ admits no critical point $u$ with $u>0$ in $\R^N$, $\Vab{u}_2^2 \leq \mu_0$ and $E_{V,\tau} (u) \geq \inf_{0<\mu \leq \mu_0} c_\tau(\mu)$.

For the former assertion, 
though it suffices to consider the case $N \geq 5$ due to \cref{lemliou}, 
for later use in \cref{sec:generalV}, we present a proof which works for $N \geq 1$. 
We argue by contradiction and suppose that for some $N \geq 1$ and $\varepsilon_0>0$, 
there exist $\{u_n\}$ and $\{\tau_n\}$ such that 
\[
\begin{aligned}
	&-\Delta u_n + V u_n = \chi_\Omega \ab( \tau_n f_1(u_n) + f_2(u_n) ) \quad \text{in} \ \R^N, 
	\quad u_n > 0 \quad \text{in} \ \R^N, 
	\\
	&0 < \Vab{u_n}_2^2 \to 0, 
	\quad E_{V,\tau_n} (u_n) \geq \varepsilon_0 > 0, 
	\quad \tau_n \to \tau_\infty \in \ab[ \frac{1}{2} , 1 ]. 
\end{aligned}
\]
Write 
\[
M_n \coloneq \max_{\R^N} u_n, \quad M_{\Omega,n} \coloneq \max_{ \overline{\Omega} } u_n.
\]
If $\{M_n\}$ is bounded, then $\Vab{u_n}_2 \to 0$ and the interpolation inequality 
yield $\Vab{u_n}_r \to 0$ for every $r \in [2,\infty)$. 
Therefore, from \eqref{A1}, \eqref{A3} and 
\[
\int_{\R^N} \vab{\nabla u_n}^2 + Vu_n^2 \, dx = \int_{\Omega} \tau_n f_1(u_n)u_n + f_2(u_n)u_n \, dx \to 0,
\]
we have a contradiction: $E_{V,\tau_n} (u_n) \to 0$. Thus, $M_n \to +\infty$ holds.

Next, we claim that 
\begin{equation}\label{eqmoneganmn}
    \frac{M_{\Omega,n}}{M_n} \to 0.
\end{equation}
In fact, suppose $M_{\Omega,n} / M_n \to \alpha_0>0$. 
Pick up $\{\tilde{P}_n\} \subset \bar{\Omega}$ so that $M_{\Omega,n} = u_n(\tilde{P}_n)$, and set 
\[
\tilde{\varepsilon}_{\Omega,n} \coloneq M_{\Omega,n}^{ - \frac{p-2}{2} }, \quad 
\tilde{U}_n(y) \coloneq \tilde{\varepsilon}_{\Omega,n}^{ \frac{2}{p-2} } u_n \ab( \tilde{\varepsilon}_{\Omega,n} y + \tilde{P}_n ), 
\quad 
\tilde{\Omega}_n \coloneq \tilde{\varepsilon}_{\Omega,n}^{-1} \ab(\Omega - \tilde{P}_n). 
\]
Notice that $\tilde{U}_n$ satisfies 
\[
\begin{dcases}
	-\Delta \tilde{U}_n + \tilde{\varepsilon}_{\Omega,n}^2 V \ab( \tilde{\varepsilon}_{\Omega,n} y + \tilde{P}_n ) \tilde{U}_n 
	= \tilde{\varepsilon}_{\Omega,n}^{ \frac{2(p-1)}{p-2} } \chi_{\tilde{\Omega}_n} (y) 
	\ab[ f_2 \ab(  \tilde{\varepsilon}_{\Omega,n}^{-\frac{2}{p-2}} \tilde{U}_n ) 
	+ \tau_n f_1 \ab( \tilde{\varepsilon}_{\Omega,n}^{-\frac{2}{p-2}} \tilde{U}_n ) 
	 ] \quad \text{in} \ \R^N,
	 \\
	 \tilde{U}_n(0) = 1, \quad \Vab{\tilde{U}_n}_\infty = \frac{M_n}{M_{\Omega,n}} \to \alpha_0^{-1} \in (0,\infty), \quad 
	 0 \in \overline{\tilde{\Omega}_n}. 
\end{dcases}
\]
Hence, $M_n \to \infty$, \eqref{A3} and elliptic regularity imply that up to a subsequence, 
$M_{\Omega,n} \to \infty$, $\tilde{U}_n \to \tilde{U}_\infty$ in $C^1_{\rm loc} (\R^N)$, where $\tilde{U}_\infty$ is a positive solution to 
\[
-\Delta \tilde{U}_\infty = a_0 \tau_\infty \chi_{\tilde{D}} \tilde{U}_\infty^{p-1} \quad \text{in} \ \R^N, 
\quad \tilde{U}_\infty(0) = 1, 
\quad \Vab{\tilde{U}_\infty}_\infty \leq \alpha_0^{-1}, 
\]
where $\tilde{D} = \R^N$ or $\tilde{D} = \tilde{H}$ (a half-space) due to $0 \in \overline{\tilde{\Omega}_n}$. 
However, this contradicts \cite{GS} if $\tilde{D} = \R^N$ and \cref{lemnelambda+0} if $\tilde{D} = \tilde{H}$. 
Thus, \eqref{eqmoneganmn} holds.

Set 
\[
v_{n,0} (x) \coloneq \ab( u_n(x) - M_{\Omega,n} )_+ \leq u_n(x).
\]
By $v_{n,0} \equiv 0$ on $\Omega$ and $\nabla u_n \cdot \nabla v_{n,0} = \vab{\nabla v_{n,0}}^2$, we have 
\[
\int_{\R^N} \vab{\nabla v_{n,0}}^2 \,dx = - \int_{\R^N} V u_n v_{n,0} \, dx 
\leq \Vab{V}_\infty \Vab{u_n}_2^2  \to 0.
\]
Put 
\[
\frac{1}{q_N} \coloneq \frac{1}{4} \ab( \frac{1}{2} - \frac{1}{N} ) + \frac{3}{4} \cdot \frac{1}{2} = \frac{2N-1}{4N}.
\]
The standard Gagliardo--Nirenberg inequality \eqref{stdGN} leads to 
\begin{equation}\label{vn0un}
\Vab{v_{n,0}}_{q_N} \leq C_N \Vab{\nabla v_{n,0}}_2^{1/4} \Vab{v_{n,0}}_2^{3/4} \leq C_N \norm{V}_\infty^{1/8} \Vab{u_n}_2 \to 0,
\end{equation}
where $C_N:=C_{q_N,N}^{1/q_N}$.

To proceed, we notice the following inequality: for each $1 < a_1 < a_2$ and $r > 0$,
\begin{equation}\label{basic-ineq}
	s \ab(s-a_2)_+^r \leq \frac{a_2}{a_2-a_1} (s-a_1)_+^{r+1} \quad \text{for every $s \geq 0$}. 
\end{equation}
In fact, the inequality trivially holds for any $s \in [0,a_2]$. 
On the other hand, if $s > a_2$, then 
\eqref{basic-ineq} follows from $s \leq a_2(s-a_1)/(a_2-a_1)$. 
For $\ell \geq 1$, write 
\[
\xi \coloneq \frac{q_N}{2} > 1, \quad 
p_\ell \coloneq \xi^{\ell-1} \cdot q_N \geq q_N > 2, \quad v_{n,\ell} (x) \coloneq \ab(  u_n(x)  - M_{\Omega,n} \sum_{m=0}^\ell \frac{1}{2^m} )_+.
\]
Multiplying the equation of $u_n$ by $v_{n,\ell}^{p_\ell -1}$ gives 
\[
\int_{\R^N} \nabla u_n \cdot (p_\ell-1) v_{n,\ell}^{p_\ell-2} \nabla v_{n,\ell} \, dx = - \int_{\R^N} V u_n v_{n,\ell}^{p_\ell-1} \, dx. 
\]
The left-hand side is computed as 
\[
\begin{aligned}
	(p_\ell-1) \int_{\R^N} \nabla u_n \cdot v_{n,\ell}^{p_\ell-2} \nabla v_{n,\ell} \, dx
	&= 
	(p_\ell -1) \int_{\R^N} v_{n,\ell}^{p_\ell-2} \vab{\nabla v_{n,\ell}}^2 \, dx 
	= (p_\ell-1) \ab( \frac{2}{p_\ell} )^2 \Vab{ \nabla \ab( v_{n,\ell}^{ p_\ell /2 } ) }_2^2.
\end{aligned}
\]
On the other hand, by \eqref{basic-ineq} with 
$a_1 = M_{\Omega,n} \sum_{m=0}^{\ell-1} 2^{-m}$ and $a_2 = M_{\Omega,n} \sum_{m=0}^\ell 2^{-m} \leq 2M_{\Omega,n}$, 
we have 
\[
u_n(x)  v_{n,\ell} (x)^{p_\ell -1} \leq 2^{\ell + 1} v_{n,\ell-1}(x)^{p_{\ell}}.
\]
Hence, the right-hand side is estimated as 
\[
- \int_{\R^N} V u_n v_{n,\ell}^{p_\ell -1} \, dx 
\leq \Vab{V}_\infty \int_{\R^N} u_n v_{n,\ell}^{p_\ell-1} \, dx 
\leq \Vab{V}_\infty 2^{\ell + 1} \Vab{v_{n,\ell-1}}_{p_\ell}^{p_\ell}. 
\]
The Gagliardo--Nirenberg inequality \eqref{stdGN} as in \eqref{vn0un} yields 
\[
C_N^8 \Vab{\nabla \ab( v_{n,\ell}^{p_\ell/2} )}_2^2 \Vab{v_{n,\ell}^{p_\ell/2}  }_{2}^{6}  \geq 
\Vab{ v_{n,\ell}^{p_\ell/2} }_{q_N}^8
= \Vab{ v_{n,\ell} }_{ p_\ell \cdot \xi }^{4p_\ell} = \Vab{v_{n,\ell}}_{p_{\ell + 1}}^{4p_\ell}.
\]
Hence, 
\[
\begin{aligned}
	\Vab{v_{n,\ell}}_{p_{\ell+1}}^{4p_\ell} 
	&\leq 
	C_N^8 \Vab{\nabla \ab( v_{n,\ell}^{p_\ell/2} )}_2^2 \Vab{v_{n,\ell} }_{p_\ell}^{3p_\ell}  
	\leq 
	C_N^8 \frac{p_\ell^2}{4(p_\ell-1)} \Vab{V}_\infty 2^{\ell + 1} \Vab{v_{n,\ell-1}}_{p_\ell}^{p_\ell}  \Vab{v_{n,\ell} }_{p_\ell}^{3p_\ell}.
\end{aligned}
\]
By $p_\ell^2 \leq 2p_\ell(p_\ell-1)$ and $v_{n,\ell} (x) \leq v_{n,\ell-1} (x)$, we infer that 
\[
\Vab{v_{n,\ell}}_{p_{\ell+1}} \leq \ab[ D_{N,V} \ab( 2 \xi )^{\ell} ]^{ 1/(4p_\ell) } \Vab{v_{n,\ell-1}}_{p_\ell},
\]
where $D_{N,V}>0$ is a constant depending only on $N$ and $V$. 
Using this inequality repeatedly, we observe that for each $\ell \in \N$, 
\[
\Vab{ v_{n,\ell} }_{p_{\ell + 1}} \leq D_{N,V}^{\sum_{ m=1 }^{\ell} 1/(4p_m) } \cdot (2\xi)^{ \sum_{m=1}^\ell m /(4p_m) } 
\Vab{v_{n,0}}_{p_1}. 
\]
Since $p_1 = q_N$, $\xi > 1$ and $(u_n-2M_{\Omega,n})_+ \leq v_{n,\ell}$ for every $\ell \geq 1$, 
it follows from \eqref{vn0un} that 
\[
\Vab{ (u_n - 2 M_{\Omega,n})_+ }_{p_{\ell + 1}} \leq D_{N,V}' \Vab{v_{n,0}}_{q_N} \leq D_{N,V}'' \Vab{u_n}_2. 
\]
Letting $\ell \to \infty$ implies 
\[
M_n - 2 M_{\Omega,n} \leq D_{N,V}''  \Vab{u_n}_2 \to 0,
\]
which is absurd due to $M_{\Omega,n} / M_n \to 0$ as $n \to +\infty$. Thus \cref{T:Liouv-musmall} holds. 
\end{proof}

\begin{remark}\label{rem:generalV}
	When we replace $V$ by a general $W \in L^\infty(\R^N)$, 
	the above proof is still valid and the same conclusion to \cref{T:Liouv-musmall} holds. 
\end{remark}

 \section{Proofs of \cref{th1,th2}: The case where $\Omega$ is bounded}
 \label{sec:pf-main}
 
 We now provide the proofs of \cref{th1,th2}.
\begin{proof}[Proofs of \cref{th1,th2}]
We prove \cref{th1,th2} at the same time and as in the proof of \cref{th51,th52} 
we use the convention $\mu_0 = +\infty$ when $1 \leq N \leq 4$ and \eqref{A1}--\eqref{A5} hold. 
Let $\mu \in (0,\mu_0)$ and 
$\{(u_n,\lambda_n)\} \subset S_\mu \times \R^+$ be a sequence of solutions to (\hyperref[eqVt]{$P_{\mu,\tau_n}$}) given by \cref{th51} (resp. \cref{th52}), 
corresponding to some $\tau_n \to 1^-$. 
Our aim is to prove that $\{u_n\}$ is a bounded Palais-Smale sequence for the functional $E_{V,1}$ constrained on $S_\mu$ at some level $c>0$.
    
We prove that $\{u_n\}$ is bounded in $H^1(\R^N)$ by contradiction, and suppose that 
$\int_{\R^N}\abs{\nabla u_n}^2\,dx \to +\infty$ as $n \to +\infty$. 
From $0 = E_{V,\tau_n}'(u_n)u_n + \lambda_n \Vab{u_n}_2^2$ and $E_{V,\tau_n}(u_n)=c_{\tau_n}$, it follows that 
\begin{equation}\label{eqsec71}
	\int_{\R^N}\abs{\nabla u_n}^2+\ab(V(x)+\lambda_n)\abs{u_n}^2\,dx = \int_{\Omega}\tau_n f_1(u_n)u_n+f_2(u_n)u_n\,dx
\end{equation}
and
\begin{equation}\label{eqsec72}
	\frac{1}{2}\int_{\R^N}\abs{\nabla u_n}^2+V(x)\abs{u_n}^2\,dx -\int_{\Omega}\tau_n F_1(u_n)+F_2(u_n)\,dx=c_{\tau_n}>0.
\end{equation}
By \cref{lemmps} (resp. \cref{lemmpe}), 
\begin{equation}\label{eqsec7ctau}
	0<c_1\leq c_\tau \leq c_{1/2} \quad \text{for each $\tau \in \ab[\frac{1}{2}, 1]$}, 
\end{equation}
which implies that $\{c_{\tau_n}\}$ is bounded. 
Thus, by \eqref{eqsec72}, we see that 
\begin{equation}\label{F-blowup}
\int_{\Omega}\tau_n F_1(u_n)+F_2(u_n)\,dx \to +\infty \quad \text{as $n \to + \infty$}.
\end{equation}
\eqref{A1} and the definition of  $f_2$ given in \eqref{eqdeff1f2} yield 
\[
\abs{f_2(t)}\leq C\abs{t} \text{ and } \abs{F_2(t)}\leq C\abs{t}^2 \quad \text{for all $t \in \R$}. 
\]
Hence, 
\begin{equation}
\label{F2-bounded}
\int_{\Omega}F_2(u_n)\, dx\leq C\mu\quad  \text{and}\quad \int_{\Omega}f_2(u_n)u_n\,dx\leq C\mu.
\end{equation}
Combining this fact with \eqref{F-blowup} implies
\begin{equation}\label{F1-blowup}
\int_{\Omega}\tau_n F_1(u_n)\,dx \to +\infty \quad \text{as $n \to + \infty$}.
\end{equation}
Furthermore, by \eqref{A3}, the definition of $f_1$ given in \eqref{eqdeff1f2} and L'H\^{o}pital's rule, 
\[
\lim_{\abs{t} \to \infty} \frac{F_1(t)}{\abs{t}^p} = \lim_{\abs{t} \to \infty} \frac{f_1(t)}{p |t|^{p-2} t} = \frac{a_0}{p}. 
\]
Hence, for any $\delta \in  (0,a_0)$ small enough, there exists $M_\delta >0$ such that if $\abs{t} > M_\delta$, then 
\begin{equation}\label{eqF1f1}
F_1(t) \leq \frac{a_0+\delta}{p} \abs{t}^p = \frac{a_0+\delta}{a_0-\delta} \cdot \frac{1}{p} \cdot (a_0-\delta) \abs{t}^p \leq \frac{a_0+\delta}{a_0-\delta} \cdot \frac{1}{p} f_1(t)t.
\end{equation}

Let $\delta > 0 $ be chosen such that $\frac{a_0+\delta}{a_0-\delta} \cdot \frac{1}{p}<\frac{1}{2}$. 
Then, it follows from \eqref{eqsec71}, \eqref{eqsec72}, \eqref{F2-bounded} and \eqref{F1-blowup} that 
\[
\begin{aligned}
	\lambda_n\mu &=\int_{\Omega}\tau_n f_1(u_n)u_n-2\tau_n F_1(u_n)+f_2(u_n)u_n-2F_2(u_n)\,dx-2c_{\tau_n}\\
	&\geq -C\mu-C+\int_{\Omega\cap \set{u_n >M_\delta }}\tau_n f_1(u_n)u_n-2\tau_n F_1(u_n)\,dx \\
	& \geq  -C\mu-C+\ab(p\frac{a_0-\delta}{a_0+\delta}-2)\int_{\Omega\cap \set{u_n >M_\delta }}\tau_n F_1(u_n)\,dx\\
  & \geq -C\mu- \tilde{C}  + \ab(p\frac{a_0-\delta}{a_0+\delta}-2)\int_{ \Omega }\tau_n F_1(u_n)\,dx \to +\infty \quad \text{as }n \to +\infty.
\end{aligned}
\]
Therefore, $\lambda_n \to +\infty$ holds. Since $m_{\lambda_n,\tau_n} (u_n) \leq 2$ due to \cref{th51} (resp. \cref{th52}), 
\eqref{equn} is satisfied by $\{(u_n,\lambda_n)\}$, and hence \cref{lem62} and \cref{propbu} are applicable. 
Then \cref{lem62} and \cref{propbu} yield for some $k \in \set{1,2}$, 
\[
\lambda_n^{\frac{N}{2}-\frac{2}{p-2}}\int_{\R^N}\abs{u_n}^2\,dx \to a_0^{-\frac{2}{p-2}}k\norm{U}_2^2 < +\infty  \quad\text{as }n \to +\infty.
\]
However, due to $\frac{N}{2}-\frac{2}{p-2}>0$ and $\norm{u_n}_2^2 = \mu>0$, this is a contradiction. 
Thus, $\{u_n\}$ is bounded in $H^1(\R^N)$. 

By \eqref{eqsec7ctau} and a standard argument, 
we know that $\{u_n\}$ is a bounded Palais-Smale sequence for the functional $E_{V,1}$ constrained on $S_\mu$ at level $c = \lim_{n \to \infty} c_{\tau_n} \geq c_1 > 0$. 
Then, by repeating the argument in the last part of the proofs of \cref{th51,th52}, 
we conclude that $u_n \to u$ in $H^1(\R^N)$ and $\lambda_n \to \lambda>0$ as $n\to +\infty$ for some $(u,\lambda) \in H^1(\R^N)\times \R^+$ and $(u,\lambda)$ is a solution to problem \eqref{eqV}.
\end{proof}

\section{Proof of \cref{T:whole}: The case $\Omega = \R^N$}
\label{sec:the-whole-sp}

In this section, we treat the case $\Omega = \R^N$ in \eqref{eqV} and prove \cref{T:whole}.

\subsection{Proof of \cref{T:whole} \ref{T:whole-i}}

In this subsection, we aim to prove the existence of minimizers to \eqref{e:loc-min}. 
Throughout this subsection, \eqref{A1} and \eqref{nA6} are always assumed. In particular, 
$E_V \in C^1(H^1(\R^N) , \R)$ and critical points of $E_V|_{S_\mu}$ give solutions of \eqref{eqV} with some $\lambda \in \R$ where 
\[
E_V(u) \coloneq \int_{\R^N} \frac{1}{2} \vab{\nabla u}^2 + \frac{1}{2} V u^2 dx - \int_{\R^N} F(u) dx. 
\]

\begin{lemma}\label{l:up-lo-EV}
	Under \eqref{A1} and \eqref{nA6}, there exists $\overline{\mu}_1=\overline{\mu}_1(N,V,f) >0$ such that 
		\[
		E_V(u) \geq \frac{3}{2} 
		\quad \text{for every $\mu \in (0, \overline{\mu}_1]$ and $u \in S_\mu$ with $4 \leq \vab{u}_V^2 \leq 8$}.
		\]
	Moreover, for any $\mu \in (0,\overline{\mu}_1]$ and $\tau \in (0,\mu)$, 
	\[
	m_{V,\mu} \leq 0 \quad \text{and} \quad m_{V,\mu} \leq m_{V,\tau} + m_{V,\mu - \tau}.
	\]
	In particular, $(0,\overline{\mu}_1] \ni \mu  \mapsto m_{V,\mu} \in (-\infty,0]$ is nonincreasing.
\end{lemma}

\begin{proof}
By \eqref{A1} and \eqref{nA6}, there exists $C=C(f)>0$ such that 
$\vab{F(t)} \leq t^2 + C \vab{t}^{q_4}$. 
As in the proof of \cref{lemmps}, it follows from \cref{lemgn} that 
for any $u \in B_\alpha \coloneq \set{ u \in S_\mu \colon \vab{u}_V^2 = \alpha }$ with $4 \leq \alpha \leq 8$, 
\[
\begin{aligned}
	E_V(u) 
	&\geq 
	\frac{1}{2} \alpha - \Vab{u}_2^2 - C \Vab{u}_{q_4}^{q_4}
	\geq 
	2 - \mu -  C C_{q_4,N,V} \alpha^{\beta_{q_4}q_4/2} \mu^{ (1-\beta_{q_4})q_4/2}.
\end{aligned}
\]
Hence, we may find $\overline{\mu}_1>0$ such that 
$E_V(u) \geq 3/2$ for every $\mu \in (0,\overline{\mu}_1]$ and $u \in S_\mu$ with $4 \leq \vab{u}_V^2 \leq 8$.

Next, notice that for any $\varepsilon>0$ there exists $C_\varepsilon>0$ such that $\vab{F(t)} \leq \varepsilon t^2 + C_\varepsilon \vab{t}^{q_4}$ for all $t \in \R$. 
Therefore, it follows from \cref{lemgn} that for each $\mu \in (0,\overline{\mu}_1]$ and $u \in S_\mu$, 
\[
\begin{aligned}
	E_V(u) 
	&\leq 
	\frac{1}{2} \vab{u}_V^2 + \varepsilon \mu + C_\varepsilon \Vab{u}_{q_4}^{q_4}  
	\leq 
	\frac{1}{2} \vab{u}_V^2 + \varepsilon \mu + C_\varepsilon C_{q_4,N,V} \vab{u}_V^{\beta_{q_4}q_4}\mu^{ (1-\beta_{q_4})q_4/2}.
\end{aligned}
\]
Recalling $\sigma_0 = 0$ in \eqref{bot-spec}, we may find $\{u_n\} \subset S_\mu$ such that $\vab{u_n}_V \to 0$, hence $u_n \in \mathcal{A}_\mu$. 
Therefore, 
\[
m_{V,\mu} \leq \limsup_{n \to +\infty} E_V(u_n) \leq \varepsilon \mu. 
\]
Since $\varepsilon > 0$ is arbitrary, $m_{V,\mu} \leq 0$ holds.

Finally, let $\tau \in (0,\mu)$ and $\{u_n\}_n \subset \mathcal{A}_\tau$ satisfy $E_V(u_n) \to m_{V,\tau}$. 
By the fact that $m_{V,\tau} \leq 0$ and the first assertion of this lemma, we may assume that $\vab{u_n}_V^2 < 4$.
Hence, there exists $\{\varphi_n\}_{n \in \N} \subset \mathcal{A}_{\tau} \cap C_0^\infty(\R^N) $ such that 
\[
E_{V} (\varphi_n) \to m_{V,\tau}, \quad \vab{\varphi_n}_V^2 < 4. 
\]
In a similar way, there exists $\{\psi_n\}_{n \in \N} \subset \mathcal{A}_{\mu - \tau }  \cap C_0^\infty(\R^N)$ such that 
\[
E_V(\psi_n) \to m_{V,\mu - \tau}, \quad \vab{\psi_n}_V^2 < 4. 
\]
For a sufficiently large $L_n \in \N$, we have 
\[
\supp \varphi_n \cap \supp \ab( \psi_n (\cdot - L_n e_1) ) = \emptyset, \quad 
\varphi_n + \psi_n ( \cdot - L_n e_1 ) \in \mathcal{A}_{\mu}.
\]
Since $V$ is $1$-periodic, 
\[
m_{V,\mu} \leq E_V( \varphi_n + \psi_n(\cdot - L_n e_1) ) = E_V(\varphi_n) + E_V(\psi_n) \to m_{V,\tau} + m_{V,\mu - \tau}. 
\]
Thus, we complete the proof. 
\end{proof}

Next, the strict subadditivity of $m_{V,\mu}$ is proved:

\begin{lemma}\label{l:str-ineq}
	Suppose \eqref{A1}, \eqref{nA6} and \eqref{nA7}. 
	For any $0<\tau < \mu \leq \overline{\mu}_1$ and $\xi > 1$ with $\xi^2 \tau \leq \mu$, 
	\[
	m_{V,\xi^2 \tau} \leq \xi^2 m_{V,\tau}
	\]
	and the strict inequality holds provided either $m_{V,\tau} < 0$ or there exists $u \in \mathcal{A}_\tau$ such that $E_{V} (u) = m_{V,\tau}$. 
	In addition, if $m_{V,\tau} < 0$ or $m_{V,\tau}$ is attained, then the strict subadditivity holds: 
	\[
	m_{V,\mu} < m_{V,\tau} + m_{V,\mu - \tau}. 
	\]
\end{lemma}

\begin{proof}
Let $ 0< \tau < \mu \leq\overline{\mu}_1$ and $\xi>1$ satisfy $\xi^2 \tau \leq \mu$. 
Fix any $\{u_n\}_n \subset \mathcal{A}_\tau$ so that 
$E_V(u_n) \to m_{V,\tau} \leq 0$ and $E_V(u_n) \leq \xi^{-2}$ for all $n \in \N$. 
From \cref{l:up-lo-EV}, we may suppose $\vab{u_n}_V^2 < 4$ for every $n \in \N$. 
For each $n \in \N$, we claim $\vab{\theta u_n}_V^2 < 4$ for any $\theta \in [1,\xi]$. 
Otherwise, for some $n \in \N$, there exists $\theta \in (1,\xi]$ such that $\vab{\theta u_n}_{V}^2 = 4$. 
From $\Vab{\theta u_n}_2^2 = \theta^2 \tau \leq \mu \leq \overline{\mu}_1$ and $u_n \not \equiv 0$, 
\cref{l:up-lo-EV} and \eqref{nA7} give 
\[
\frac{3}{2} \leq E_V(\theta u_n) = \theta^2 E_V(u_n) -  \int_{\R^N} F(\theta u_n) - \theta^2 F(u_n) dx < \theta^2 E_V(u_n) \leq \theta^2 \xi^{-2}\leq 1,
\]
which is a contradiction. Thus, $\vab{\theta u_n}_V^2 < 4$ for every $n \in \N$ and $\theta \in (1,\xi]$.

We remark that $\theta u_n \in \mathcal{A}_{\theta^2 \tau}$ for every $\theta \in[1,\xi]$. 
Hence, 
\begin{equation}\label{e:m}
m_{V,\xi^2\tau} \leq E_V(\xi u_n) = \xi^2 E_V(u_n) - \int_{\R^N} F(\xi u_n) - \xi^2 F(u_n) dx \leq \xi^2 E_V(u_n). 
\end{equation}
Letting $n \to +\infty$ gives 
\[
m_{V,\xi^2\tau} \leq \xi^2 m_{V,\tau}. 
\]

In what follows, we prove that $m_{V,\xi^2 \tau} < \xi^2 m_{V,\tau}$ holds when $m_{V,\tau} < 0$ or $m_{V,\tau}$ is attained. 
First, assume that there exists $u \in \mathcal{A}_\tau$ such that $E_V(u) = m_{V,\tau}$. 
In \eqref{e:m}, we substitute $u=u_n$ to obtain 
\[
m_{V,\xi^2\tau} \leq \xi^2 m_{V,\tau} - \int_{\R^N} F(\xi u) - \xi^2 F(u) dx. 
\]
Since $u \not \equiv 0$, \eqref{nA7} implies 
\[
\int_{\R^N} F(\xi u) - \xi^2 F(u) dx > 0, \quad m_{V,\xi^2 \tau} < \xi^2 m_{V,\tau}. 
\]

On the other hand, suppose that $E_V(u_n) \to m_{V,\tau} < 0$ holds where $\{u_n\} \subset \mathcal{A}_\tau$. 
By \eqref{e:m}, if we can verify 
\begin{equation}\label{F-pos}
	\limsup_{n \to \infty} \int_{\R^N} F(\xi u_n) - \xi^2 F(u_n) dx > 0, 
\end{equation}
then $m_{V,\xi^2 \tau} < \xi^2 m_{V,\tau}$ holds.

To prove \eqref{F-pos}, remark that $\{u_n\}$ is bounded in $H^1(\R^N)$ 
thanks to $u_n \in S_\tau$, $\vab{u_n}_V^2 \leq 8$ and $V \in L^\infty(\R^N)$. 
Next, we claim that 
\[
\liminf_{n \to +\infty} \sup_{z \in \Z^N} \Vab{ u_n }_{L^2(z+Q)} > 0, \quad \text{where $Q \coloneq [0,1]^N$}. 
\]
Indeed, if this fails to hold, then Lions' lemma yields $\Vab{u_n}_q \to 0$ for any $q \in (2,2^*)$. 
Hence, \eqref{A1}, \eqref{nA6} and the fact $\sigma_0 = 0$ lead to 
\[
0 > m_{V,\tau} = \liminf_{n \to +\infty} E_V(u_n) = \liminf_{n \to +\infty} \frac{1}{2} \vab{u_n}_V^2 \geq 0,
\]
which is a contradiction. In particular, there exists $\{z_n\} \subset \Z^N$ such that 
$\Vab{u_n(\cdot + z_n)}_{L^2(Q)} \to c_0 > 0$. 
Since $V$ is $1$-periodic, without loss of generality, we may suppose $z_n=0$. 
Recalling that $\{u_n\}_n$ is bounded in $H^1(\R^N)$, 
we may suppose $u_n \to u_\infty$ weakly in $H^1(\R^N)$ and $u_n \to u_\infty$ a.e. $\R^N$. 
Then \eqref{nA7}, Fatou's lemma and $u_\infty \not \equiv 0$ yield 
\[
0 < \int_{\R^N} F(\xi u_\infty) - \xi^2 F(u_\infty) dx 
\leq \liminf_{n \to +\infty} \int_{\R^N} F(\xi u_n) - \xi^2 F(u_n) dx. 
\]
Hence, \eqref{F-pos} and $m_{V,\xi^2 \tau} < \xi^2m_{V,\tau}$ hold.

Finally, let $\tau \in (0,\mu)$ and assume that $m_{V,\tau} < 0$ or $m_{V,\tau}$ is attained. 
When $\tau \in [\mu/2, \mu )$, select $\xi > 1$ so that $\xi^2 \tau = \mu$. From
\[
\xi^2 - 1 = \frac{\mu-\tau}{\tau} \leq 1, \quad\tau = \frac{\tau}{\mu - \tau}  \cdot (\mu-\tau)
\]
and the strict inequality on $m_{V,\xi^2 \tau} < \xi^2 m_{V,\tau}$, it follows that 
\[
m_{V,\tau} \leq \frac{\tau}{\mu-\tau} \cdot m_{V,\mu-\tau}, \quad 
m_{V,\mu} < \xi^2 m_{V,\tau} = m_{V,\tau} + \ab( \xi^2 -1 ) m_{V,\tau} 
\leq m_{V,\tau} + m_{V,\mu-\tau}.
\]
On the other hand, if $\tau \in (0,\mu/2)$, then $\tau < \mu /2 <  \mu - \tau$ and 
\[
m_{V,\mu-\tau} < \frac{\mu-\tau}{\tau} m_{V,\tau} \leq 0. 
\]
Hence, by changing the role of $\tau$ and $\mu-\tau$ in the above argument, 
we obtain $m_{V,\mu} < m_{V,\tau} + m_{V,\mu -\tau}$. 
\end{proof}

Next, we prove the existence of minimizers when $m_{V,\mu} < 0$:

\begin{lemma}\label{l:attain}
	Assume \eqref{A1}, \eqref{nA6}, \eqref{nA7}, $\mu \in (0,\overline{\mu}_1]$ and $m_{V,\mu} < 0$. 
	Then each minimizing sequence is compact in $H^1(\R^N)$ up to translations, 
	and $m_{V,\mu}$ is attained by a positive function.  
\end{lemma}

\begin{proof}
Let $\mu \in (0,\overline{\mu}_1]$ and $\{u_n\}_{n \in \N}$ be any minimizing sequence for $m_{V,\mu} < 0$. 
Due to $E_V(u_n) \to m_{V,\mu} < 0$, as in the proof of \cref{l:str-ineq}, we may verify
\[
\liminf_{n \to +\infty} \sup_{z \in \Z^N} \Vab{u_n}_{L^2(z+Q)} > 0.
\]
Moreover, thanks to the periodicity of $V$, we may suppose 
$\Vab{u_n}_{L^2(Q)} \to c_0>0$ and $u_n \rightharpoonup u_\infty \not \equiv 0$. 
Notice that $\vab{u_n}_V^2 < 4$ holds due to \cref{l:up-lo-EV} and $E_V(u_n) \to m_{V,\mu} < 0$. 
Since the functional $H^1(\R^N) \ni u \mapsto \vab{u}_V^2 \in [0,+\infty)$ is convex due to $\sigma_0=0$, 
it is weakly lower semicontinuous on $H^1(\R^N)$ and
\[
\vab{u_\infty}_V^2 \leq \liminf_{n \to +\infty} \vab{u_n}_V^2 \leq 4. 
\]
In particular, $\Vab{u_\infty}_2^2 \leq \mu$, $u_\infty \in \mathcal{A}_{\Vab{u_\infty}_2^2}$ and 
$m_{V,\Vab{u_\infty}_2^2} \leq E_V(u_\infty)$.

If $0<\Vab{u_\infty}_2^2 < \mu$, then by writing $v_n \coloneq u_n - u_\infty$, the weak convergence $u_n \rightharpoonup u_\infty$ and the Brezis--Lieb lemma give
\[
4  > \vab{u_n}_V^2 = \vab{u_\infty}_V^2 + \vab{v_n}_V^2 + o(1), \quad 
E_V(u_n) = E_V(u_\infty) + E_V(v_n) + o(1), \quad \Vab{v_n}_2^2 \to \mu - \Vab{u_\infty}_2^2. 
\]
Thus, we may suppose $\vab{u_\infty}_V^2 \leq 4$, $\vab{v_n}_V^2 \leq 5$ and $\Vab{v_n}_2^2 \leq \mu - \Vab{u_\infty}_2^2 +o(1)$, hence 
$v_n \in \mathcal{A}_{ \Vab{v_n}_2^2 }$. 
Since $\{v_n\}$ is bounded in $H^1(\R^N)$ and $E_V$ is uniformly continuous on each bounded set of $H^1(\R^N)$, 
letting $n \to + \infty$ yields 
\[
m_{V,\mu} \geq E_V(u_\infty) + \liminf_{n \to +\infty} E_V \ab( \sqrt{ \mu - \Vab{u_\infty}_2^2 } \frac{v_n}{\Vab{v_n}_2} ) \geq  m_{V,\Vab{u_\infty}_2^2} + m_{V, \mu - \Vab{u_\infty}_2^2 }.
\]
\cref{l:up-lo-EV} leads to
\[
m_{V,\mu} = m_{V,\Vab{u_\infty}_2^2} + m_{V,\mu - \Vab{u_\infty}_2^2}, \quad E_V(u_\infty) = m_{V,\Vab{u_\infty}_2^2},
\]
which contradicts \cref{l:str-ineq}. Hence, $\mu = \Vab{u_\infty}_2^2$ and $u_n \to u_\infty$ strongly in $L^2(\R^N)$. 
Furthermore, by 
\[
m_{V,\mu} = \lim_{n \to \infty} E_V(u_n) \geq  E_V(u_\infty) \geq m_{V,\mu},
\]
we have $E_V(u_\infty) = m_{V,\mu}$. 
Since $u_n\to u_\infty$ strongly in $L^2(\R^N)$ and $\{u_n\}$ is bounded in $H^1(\R^N)$, 
the interpolation inequality yields $u_n\to u_\infty$ strongly in $L^q(\R^N)$ for every $q\in(2,2^*)$. By \eqref{A1}, \eqref{nA6} and $V\in L^\infty(\R^N)$, it follows that
\[
\int_{\R^N}F(u_n)\,dx\to\int_{\R^N}F(u_\infty)\,dx,
\qquad
\int_{\R^N}V u_n^2\,dx\to\int_{\R^N}V u_\infty^2\,dx.
\]
These facts together with $E_V(u_n)\to E_V(u_\infty)$ give $\Vab{\nabla u_n}_2^2 \to \Vab{\nabla u_\infty}_2^2$.
Thus, $u_n \to u_\infty$ strongly in $H^1(\R^N)$. 
Finally, since $f$ is odd and $F$ even, $E_V(\vab{u_\infty}) = E_V(u_\infty)$ and $\vab{u_\infty} \in \mathcal{A}_\mu$. 
Hence, we may assume $u_\infty \geq 0$ in $\R^N$. 
Since $u_\infty$ is a minimizer, it is a solution to \eqref{eqV} with some $\lambda \in \R$, 
and the strong maximum principle yields $u_\infty > 0$ in $\R^N$. 
Therefore, $u_\infty$ is the desired minimizer. 
\end{proof}

We are ready to prove \cref{T:whole} \ref{T:whole-i}.

\begin{proof}[Proof of \cref{T:whole} \ref{T:whole-i}]
Let $\mu \in (0,\overline{\mu}_1]$ where $\overline{\mu}_1>0$ is the constant in \cref{l:up-lo-EV}. 
According to \cref{l:attain}, what remains to prove is $m_{V,\mu} < 0$. 
Fix any $\eta \in C^\infty_0(\R^N)$ with
\[
0 \leq \eta \leq 1, \quad 
\eta(x) = \begin{dcases}
	1 & \text{if $\vab{x} \leq 1$}, \\
	0 & \text{if $\vab{x} \geq 2$},
\end{dcases}
\]
and recall a positive $1$-periodic solution $\psi_0$ to $-\Delta u + Vu = 0$ in $\R^N$ from \cref{subsec:-rewr}. 
Consider
\[
z_n (x) \coloneq n^{-N/2} \eta \ab( \frac{x}{n} ) \psi_0(x), \quad 
v_n (x) \coloneq \sqrt{\mu} \frac{z_n}{\Vab{z_n}_2} \in S_\mu. 
\]
In \cite[Lemma 3.4 and Proposition 3.6]{HKS92}, it was shown that 
%
\[
\abs{v_n}_V^2\leq Cn^{-2}, \quad \norm{v_n}_{q}^{q}\geq C n^{-\frac{N}{2}(q-2)} 
\quad \text{for $q \in (2,2^*)$}. 
\]
Though \cite{HKS92} dealt with the case $N \geq 2$, the argument works also for $N=1$. 
Moreover, by \cref{lemgn}, for $q_6$ in \eqref{nA8}, we have 
\[
\norm{v_n}_{q_6}^{q_6} \leq C\mu^{\frac{q_6}{2}(1 - \beta_{q_6})} \abs{v_n}_{V}^{\beta_{q_6} {q_6}}
\leq \tilde C\mu^{\frac {q_6}{2}(1 - \beta_{q_6})} n^{-\beta_{q_6} {q_6}}.
\]
Since $\beta_{q_6} {q_6}=\frac{N}{2}(q_6-2)>2>\frac{N}{2}(q_5-2)$, 
by \eqref{nA8}, we conclude that, for $n \in \mathbb{N}^+$ large enough, 
\[
\abs{v_n}_V^2\leq 6, \quad m_{V,\mu} \leq  E_V(v_n)<0.
\]
Since $V \in L^\infty(\R^N)$ and $\bigcup_{0<\mu \leq \bar{\mu}_1} \mathcal{A}_\mu$ is bounded in $H^1(\R^N)$, 
it is easily seen that $m_{V,\mu} \to 0$ as $\mu \to 0^+$. 
This completes the proof of \cref{T:whole} \ref{T:whole-i}. 
\end{proof}

\subsection{Proof of \cref{T:whole} \ref{T:whole-ii}}

In this subsection, we deal with the existence of mountain pass type critical points of $E_V|_{S_\mu}$ 
when $\mu>0$ is small. 
Throughout this subsection, we always assume \eqref{A1}--\eqref{A4} and \eqref{nA9}. 
To begin with, we remark that from \eqref{A3} and \eqref{nA9}, there exists $C>0$ such that 
\begin{equation}\label{eq:est-f-above}
	f(t)t \leq C \ab( \vab{t}^{2+4/N} + \vab{t}^p ) 
	\quad \text{for all $t \in \R$}.
\end{equation}
Then we need the following lemma on the Lagrange multiplier and the Morse index:

\begin{lemma}\label{lemb1}
	Suppose $N \geq 1$, \eqref{A1}--\eqref{A4} and \eqref{nA9}. 
	For any $M > 0$, there exists $\mu_M > 0$ such that for each $\tau \in [1/2,1]$ and each solution $(u,\lambda)\in H^1(\R^N)\times [0,+\infty)$ to
	\[
	-\Delta u + (V+\lambda)u =  \ab( \tau f_1(u) + f_2(u) ),\quad  u>0 \quad \text{in} \ \R^N, \quad \Vab{u}_2^2 \leq \mu_M, 
	\]
	$\lambda\geq M$ and $\tilde{m}_{\tau,0}(u)\geq 1$ hold.
\end{lemma}
\begin{proof}
	We argue by contradiction and suppose that for some $N \geq 1$, 
	there exist $\{(u_n,\lambda_n)\}\subset H^1(\R^N)\times [0,+\infty)$ and $\{\tau_n\}$ such that 
	\[
	\begin{aligned}
		&-\Delta u_n + (V+\lambda_n) u_n =  \ab( \tau_n f_1(u_n) + f_2(u_n) ) \quad \text{in} \ \R^N, 
		\quad u_n > 0 \quad \text{in} \ \R^N, 
		\\
		&0 < \Vab{u_n}_2^2 \eqcolon \mu_n \to 0, 
		\quad \tau_n \to \tau_\infty \in \ab[ \frac{1}{2} , 1 ]. 
	\end{aligned}
	\]
	Moreover, $\{(u_n,\lambda_n)\}$ satisfies $\lambda_n \leq M$ for some $M>0$ or $\tilde{m}_{\tau_n,0}(u_n) = 0$. 
	Elliptic regularity yields $u_n\in C^{1,\alpha}(\R^N)$ for all $\alpha \in (0,1)$ and $u_n(x) \to 0$ as $\abs{x}\to +\infty$. 
	Thus, we can pick $P_n \in \R^N$ such that $u_n(P_n)=\norm{u_n}_\infty$.

	We first show that $\norm{u_n}_\infty \to +\infty$. By \eqref{eq:est-f-above} and \cref{lemgn},
	\[
	\begin{aligned}
		\abs{u_n}_V^2\leq \abs{u_n}_V^2+\lambda_n\norm{u_n}^2_2&=\int_{\R^N}\tau_n f_1(u_n)u_n + f_2(u_n)u_n \,dx \\
		&\leq C\norm{u_n}_{2+4/N}^{2+4/N}+C\norm{u_n}_{p}^{p}\\
		&\leq C\norm{u_n}_{2+4/N}^{2+4/N} \ab(1+\norm{u_n}_\infty^{p-2-4/N})\\
		& \leq C\Vab{u_n}_2^{4/N}\abs{u_n}_V^{2} \ab(1+\norm{u_n}_\infty^{p-2-4/N}).
	\end{aligned}
	\]
	By $\Vab{u_n}_2^2 \to 0$ and $u_n \not \equiv 0$, 
	we conclude that $\norm{u_n}_\infty \to +\infty$ as $n\to +\infty$. 
	From this fact and the blow-up argument (see the proof of \cref{lem62}) with \cite{GS}, 
	it follows that $\lambda_n \to +\infty$ as $n \to +\infty$. 
	Hence, up to a subsequence, $\{(u_n,\lambda_n)\}$ satisfies $\tilde{m}_{\tau_n,0}(u_n) = 0$. 
	
	To derive a contradiction, we will find $\psi_n\in T_{u_n}S_{\mu_n}$ such that 
	$Q_{\lambda_n,\tau_n,u_n}(\psi_n)<0$ for $n \in \mathbb{N}^+$ large enough. 
	To this end, let $\varepsilon_n \coloneq \lambda_n^{-\frac{1}{2}}$ and $U_n(y) \coloneq (\tau_\infty a_0)^{\frac{1}{p-2}}\varepsilon_n^{\frac{2}{p-2}}u_n(\varepsilon_ny+P_n)$. 
	By repeating the argument in the proof of \cref{lem62}, we know that, up to a subsequence, $U_n \to U$ in $C^1_{\text{loc}}(\R^N)$, where $U \in H^1(\R^N)$ 
	is a positive radial solution to 
	\begin{equation*}
		\begin{dcases}
			-\Delta U + U=  U^{p-1} ,\quad U>0 \quad  \text{in }\R^N, 
			\\ U(0)=\norm{U}_\infty.
		\end{dcases}
	\end{equation*}
	Moreover, by elliptic regularity, $U \in C^\infty(\R^N)$ and for some $\delta>0$,
	\[
	\vab{U(x)}+\vab{\nabla U(x)}+\vab{D^2U(x)} + \vab{D^3U(x)}
	\le C e^{-\delta|x|}.
	\]
	Hence, $\gamma(\theta) \coloneq \theta^{N/2} U(\theta \cdot) \in C^2( [1/2,2] , H^1(\R^N) )$ and $\Vab{\gamma(\theta)}_2^2 = \Vab{U}_2^2$.

	Put 
	\[
	\begin{aligned}
		&Z \coloneq \gamma'(1) = \frac{N}{2}U + x \cdot \nabla U \in T_U S_{\Vab{U}_2^2},  \\
		&	I(u) \coloneq \frac{1}{2} \Vab{\nabla u}_2^2 + \frac{1}{2} \Vab{u}_2^2 - \frac{1}{p} \Vab{u}_p^p, \quad 
		Q_u(\varphi) \coloneq \int_{\R^N} \vab{\nabla \varphi}^2 + \varphi^2 - (p-1) |u|^{p-2} \varphi^2 \, dx. 
	\end{aligned}
	\]
	From $I'(U) = 0$, it follows that 
	\[
	Q_U(Z) = I''(U)[Z,Z] = \frac{d^2}{d\theta^2} I(\gamma(\theta))|_{\theta = 1}.
	\]
	On the other hand, the Pohozaev identity gives 
	\[
	\Vab{\nabla U}_2^2 = N \ab( \frac{1}{2} - \frac{1}{p} ) \Vab{U}_p^p.
	\]
	Since $N+2 - Np/2 < 0$ holds thanks to $p> 2+4/N $, we have 
	\[
	\begin{aligned}
		\frac{d^2}{d\theta^2} I(\gamma(\theta))|_{\theta = 1} 
		&= 
		\frac{d^2}{d \theta^2} 
		\ab[ \frac{\theta^2}{2} \Vab{\nabla U}_2^2 + \frac{1}{2} \Vab{U}_2^2 - \frac{\theta^{Np/2-N}}{p} \Vab{U}_p^p  ]\bigg|_{\theta=1}
		\\
		&= \Vab{\nabla U}_2^2 - N \ab( \frac{1}{2} - \frac{1}{p} ) \ab( \frac{Np}{2} - N-1 ) \Vab{U}_p^p
		\\
		&= N \ab( \frac{1}{2} - \frac{1}{p} ) \ab[ N+2 - \frac{Np}{2} ] \Vab{U}_p^p < 0.
	\end{aligned}
	\]
	Hence, $Q_U(Z) < 0$ and $(U,Z)_2 = 0$.

	By the density of $C_0^\infty(\mathbb R^N)$ in
	$H^1(\mathbb R^N)$, we can choose
	\[
	Z_n\in C_0^\infty(\mathbb R^N)\qquad \text{and}\qquad 
	Z_n\to Z
	\quad\text{in }H^1(\mathbb R^N).
	\]
	Fix a nonnegative function $\eta\in C_0^\infty(\mathbb R^N)$ with $\Vab{\eta}_2=1$. 
	Since $U>0$, $(U,\eta)_2 > 0$ holds. 
	Set
	\[
	\widehat Z_n
	:=Z_n-
	\frac{ (Z_n,U)_2 }{(U,\eta)_2}\eta.
	\]
	It is easily seen from $\Vab{Z_n-Z}_{H^1} \to 0$ and $(U,Z)_2 = 0$ that 
	\[
	\widehat Z_n\in C_c^\infty(\mathbb R^N), \quad (\widehat{Z}_n , U)_2 = 0, \quad 
	\Vab{ \widehat{Z}_n - Z }_{H^1} \to 0.
	\]
	Since $Q_U$ is continuous on
	$H^1(\mathbb R^N)$, $Q_U(\widehat Z_n)\to Q_U(Z)<0$ as $n \to +\infty$. 
	We fix a sufficiently large $n_0 \in \mathbb{N}^+$ so that $Q_U(\widehat Z_{n_0})<0$ and set
	\[
	\varphi_n \coloneq \widehat{Z}_{n_0} - \frac{ ( U_n , \widehat{Z}_{n_0} )_2 }{(U_n,\eta)_2} \eta \in C^\infty_0(\R^N). 
	\]
	It is not difficult to check that $\varphi_n$ is well-defined due to $U_n>0$, and that 
	\[
	\ab( U_n ,\varphi_n )_2 = 0, \quad 
	\Vab{\varphi_n - \widehat{Z}_{n_0}}_{H^1} \to 0.
	\]
	
	Define
	\[
	\psi_n(x) \coloneq
	\varepsilon_n^{-(N-2)/2}
	\varphi_n\left(\frac{x-P_n}{\varepsilon_n}\right) \in C^\infty_0(\R^N).
	\]
	Direct computations yield 
	\[
	\int_{\mathbb R^N}u_n\psi_n\,dx
	=
	(\tau_\infty a_0)^{-1/(p-2)}
	\varepsilon_n^{(N+2)/2-2/(p-2)}
	\int_{\mathbb R^N}U_n\varphi_n\,dy
	=0.
	\]
	Thus $\psi_n\in T_{u_n}S_{\mu_n}$. Notice that 
	\[
	\supp \varphi_n \subset \supp \widehat{Z}_{n_0} \cup \supp \eta \quad \text{for every $n \in \N$}. 
	\]
	Let $\delta \in (0,1)$ be a small constant fixed later. 
	Since $\sup_n \Vab{\varphi_n}_\infty < \infty$, 
	applying the arguments in \eqref{eqlem21m1}--\eqref{eqlem21hx2}, we obtain 
	\[
	\begin{aligned}
		Q_{\lambda_n,\tau_n,u_n} (\psi_n) 
		\leq Q_{U_n} (\varphi_n) + \delta \int_{\R^N} U_n^{p-2} \varphi_n^2 \, dx + o(1).
	\end{aligned}
	\]
Since $\supp\varphi_n$ is contained in a fixed compact set, $U_n\to U$ in $C^1_{\rm loc}(\R^N)$ and $\varphi_n\to\widehat Z_{n_0}$ in $H^1(\R^N)$, we have
\[
C_0\coloneq\sup_n\int_{\R^N}U_n^{p-2}\varphi_n^2\,dx<+\infty.
\]
By $Q_{U_n}(\varphi_n)\to Q_U(\widehat Z_{n_0})<0$, we may choose
\[
0<\delta<\frac{|Q_U(\widehat Z_{n_0})|}{2\max\{C_0,1\}},
\]
and then $Q_{\lambda_n,\tau_n,u_n}(\psi_n)<0$ holds for all sufficiently large $n$.
Since $\psi_n \in T_{u_n}S_{\mu_n}$ and 
	\[
	\lambda_n = - \frac{E_{V,\tau_n}'(u_n) u_n }{\Vab{u_n}_2^2},
	\]
	we obtain a contradiction to $\widetilde m_{\tau_n,0}(u_n)=0$ for all $n \in \mathbb{N}^+$.
\end{proof}
Now we give a proof of \cref{T:whole} \ref{T:whole-ii}. 

\begin{proof}[Proof of \cref{T:whole} \ref{T:whole-ii}]
	Let $\mu_1>0$ be given by \cref{lemb1}
	with $M=1$. 
	Choose $\overline{\mu}_2 \in(0,\mu_1]$ sufficiently
	small that the mountain pass geometry of \cref{lemmps}
	holds for every $\mu\in(0,\overline{\mu}_2]$ and
	$\tau\in[1/2,1]$. Fix $\mu\in(0,\overline{\mu}_2]$. 
	
	By repeating the argument in the proofs of \cref{lemmps,th51,th52}, for almost every $\tau \in [1/2,1]$, there exist sequences $\{u_n\}\subset S_\mu$ and $\zeta_n\to0^+$ such that, as $n \to +\infty$,
	\begin{enumerate}[label=\rm(\roman*)]
		\item $E_{V,\tau}(u_n) \to c_\tau>0$;
		\item \label{bps2r}$\norm{E_{V,\tau}'(u_n) }_{ T_{u_n}^* S_{\mu} }\to 0$;
		\item \label{bps3r} $u_n \geq 0$ and $\{u_n\}$ is bounded in $H^1(\R^N)$;
        \item \label{bps4r} $\widetilde m_{\tau,\zeta_n}(u_n)\leq1$;
	\end{enumerate}
	a sequence $\lambda_n \coloneq -\frac{1}{\mu}E'_{V,\tau}(u_n)u_n$ is bounded
	and, if  there exists a subspace $W_n \subset H^1(\R^N)$ such that 
	\begin{equation}\label{eqineqmorse-r}
		E''_{V,\tau}(u_n)[w,w]+\lambda_n(w,w)_2<-\zeta_n \norm{w}^2_{H^1} \quad \text{for all }w \in W_n\setminus\{0\},
	\end{equation}
	then $\dim W_n \leq 2$ holds. Up to a subsequence, we may assume 
	\begin{equation*}
		\begin{aligned}
			u_n &\rightharpoonup  u_\tau  \quad \text{in } H^1(\R^N),\\
			u_n &\to u_\tau \quad \text{in } L^\nu_{\text{loc}}(\R^N) \text{ for all } \nu\in [1,2^*),\\
			u_n &\to u_\tau \quad \text{a.e. in } \R^N,
		\end{aligned}
	\end{equation*}
	and $\lambda_n \to \lambda_\tau \in \R$ without loss of generality. 
	It is clear that $u_{\tau}\in H^{1}(\R^N)$ solves 
	\begin{equation*}
		-\Delta u_\tau+V(x)u_\tau+\lambda_\tau u_\tau=\tau f_1(u_\tau)+ f_2(u_\tau).
	\end{equation*}
	
	Adapting the argument in the proof of \cref{lem42},
	we show that ${\lambda_\tau}\geq0$. We argue by contradiction and suppose that ${\lambda_\tau}<0$. 
	Since $\inf\sigma\ab(-\Delta +V)=0$, we may find $\varphi_{\lambda_\tau} \in C_0^\infty (\R^N)$ such that 
	\[
	\| \varphi_{\lambda_\tau} \|_2 = 1, \quad | \varphi_{\lambda_\tau} |_V^2 < -\frac{{\lambda_\tau}}{3}. 
	\]
	From the compactness of $\supp \varphi_{\lambda_\tau}$ and \eqref{A2}, for every $n \in \mathbb{N}^+$, 
	there exist $\alpha_i \in \Z^N$ ($i=1,2,3$) such that 
	\begin{equation}\label{eqf2''+f1''r}
		\begin{aligned}
			&\supp \varphi_{\lambda_\tau} ( \cdot - \alpha_{i}) \cap \supp \varphi_{\lambda_\tau} (\cdot - \alpha_{j}) = \emptyset\quad \text{for every $1\leq i <j \leq 3$},
			\\
			& \int_{\R^N}| f_2'(u_n(x)) + \tau f_1'(u_n(x)) |\abs{\varphi_{\lambda_\tau}(x-\alpha_i)}^2\,dx \leq -\frac{\lambda_\tau}3 \quad \text{for every $i=1, 2,3$}. 
		\end{aligned}
	\end{equation}
	For any $\varphi_{{\lambda_\tau},i}(x) \coloneq \varphi_{\lambda_\tau}(x - \alpha_{i})$, the periodicity of $V$ leads to 
	\[
	| \varphi_{{\lambda_\tau},i} |_{V}^2 = | \varphi_{\lambda_\tau} |_V^2 < -\frac{{\lambda_\tau}}{3}.
	\]
	Choose $\tilde{C}_{\lambda_\tau} > 0$ so that 
	\[
	\tilde{C}_{\lambda_\tau} \Vab{\varphi_{\lambda_\tau}}_{H^1}^2 \leq \Vab{\varphi_{\lambda_\tau}}_2^2. 
	\]
	By \eqref{eqf2''+f1''r} and \( \norm{\varphi_{\lambda_\tau}}_2 = 1 \),
\[
\begin{aligned}
			&\int_{\R^N} | \nabla \varphi_{{\lambda_\tau}, i} |^2 + V(x) \varphi_{{\lambda_\tau},i}^2 + {\lambda_\tau} \varphi_{{\lambda_\tau},i}^2 \,dx-
			\int_{\R^N} \ab[ f_2'(u_n) + \tau f_1'(u_n) ] \varphi_{{\lambda_\tau},i}^2 \, dx\\
			&\qquad\leq |\varphi_{{\lambda_\tau},i}|_V^2 + {\lambda_\tau} \|\varphi_{{\lambda_\tau},i}\|_2^2 - \frac{\lambda_\tau}{3}\|\varphi_{{\lambda_\tau},i}\|_2^2\\
			&\qquad< \frac{\lambda_\tau}{3}\norm{\varphi_{\lambda_\tau}}_2^2 \leq \frac{\tilde C_{\lambda_\tau}\lambda_\tau}{3}\Vab{\varphi_{\lambda_\tau}}_{H^1}^2<0.
		\end{aligned}\]
	Set \(Y_{d,{\lambda_\tau}}\coloneq \operatorname{span}  \set{ \varphi_{{\lambda_\tau},1},\varphi_{\lambda_\tau,2}, \varphi_{{\lambda_\tau},3} } \). 
	It is easily seen that $\dim Y_{d,{\lambda_\tau}} = 3$, which contradicts \eqref{eqineqmorse-r} for $n \in \mathbb{N}^+$ large enough.

	Next, we exclude the vanishing case. We argue by contradiction and suppose that, for each $R>0$, 
	\[\lim_{n \to \infty}\sup_{y \in \Z^N} \Vab{u_n}_{L^2(y+Q)}  =0, \quad \text{where $Q=[0,1]^N$}.\]
	By Lions' lemma, we have, as $n \to +\infty$,
	\[
	u_n \to0\quad\text{in }L^s(\mathbb R^N)
	\qquad\text{for every }2<s<2^*.
	\]
	This fact and \eqref{A1}--\eqref{A3} give 
	\begin{equation}\label{eqrvanish}
		\int_{\mathbb R^N}f_2(u_n)u_n+{\tau}f_1(u_n)u_n\,dx\to0 \quad \text{and}\quad     \int_{\mathbb R^N}F_2(u_n)+{\tau}F_1(u_n)\,dx\to0.
	\end{equation}
	Testing
	$E'_{V,\tau}(u_n)+{\lambda_\tau}(u_n,\cdot)_2=o(1)$
	in $H^{-1}(\mathbb R^N)$ with $u_n$, and using the
	boundedness of $\{u_n\}$ in $H^1(\mathbb R^N)$, we obtain
	\begin{equation*}\label{eqrvanishun}
		|u_n|_V^2+{\lambda_\tau}\|u_n\|_2^2=o(1).
	\end{equation*}
Since ${\lambda_\tau}\geq0$, we have $|u_n|_V^2\to0$. This together with \eqref{eqrvanish} implies $E_{V,\tau}(u_n)\to0$, which contradicts $E_{V,\tau}(u_n) \to c_\tau>0$.
	
	Since $\{u_n\}$ is not vanishing, there exists $\{y_n\}\subset \Z^N$ such that
	\begin{equation*}\label{eqnonva}
		\limsup_{n \to \infty} \Vab{u_n}_{L^2(y_n+Q)} \geq \delta.
	\end{equation*}
	By the periodicity of $V$, we may assume $y_n = 0$. Then, up to a subsequence, we may assume that, as $n \to +\infty$,
	\begin{equation*}\label{eqrvn}
		\begin{aligned}
			&u_n \rightharpoonup  u_\tau \not\equiv 0 \quad &&\text{in } H^1(\R^N),\\
			&u_n \to u_\tau \quad &&\text{in } L^\nu_{\text{loc}}(\R^N) \text{ for all } \nu\in [1,2^*),\\
			&u_n \to u_\tau \quad &&\text{a.e. in } \R^N. 
		\end{aligned}
	\end{equation*}
	It is clear that $u_{\tau}\in H^{1}(\R^N)$ solves 
	\begin{equation*}
		-\Delta u_\tau+V(x)u_\tau+\lambda_\tau u_\tau=\tau f_1(u_\tau)+ f_2(u_\tau).
	\end{equation*}
	By using strong maximum principle as in the proofs of \cref{th51,th52}, we see that $u_\tau>0$. 
	Therefore, \cref{lemb1} implies $\lambda_\tau\geq 1$. Moreover, Fatou's lemma yields 
	\[
	0<\|u_\tau\|_2^2\leq\mu \leq \overline{\mu}_2.
	\]

	To prove $u_n \to u_\tau$ strongly in $H^1(\R^N)$, 
	consider $r_n \coloneq u_n-u_\tau$. 
	Our aim is to prove that $\{r_n\}$ vanishes, that is, 
	\[
	\lim_{n \to +\infty} \sup_{z \in \Z^N} \Vab{r_n}_{L^2(z+Q)} = 0. 
	\]
	If this does not hold, then up to a subsequence, there exists $\{k_n\}\subset \Z^N$ such that
	\begin{equation*}
		\limsup_{n \to +\infty} \Vab{r_n}_{L^2(k_n+Q)} \geq \delta'.
	\end{equation*}
	By the exact same argument, we see that, as $n \to +\infty$,
	\begin{equation}\label{eqrvnkn}
		\begin{aligned}
			&u_n(\cdot+k_n)-u_\tau(\cdot+k_n) \rightharpoonup  u^2_\tau \not\equiv 0 \quad &&\text{in } H^1(\R^N),\\
			&u_n(\cdot+k_n)-u_\tau(\cdot+k_n)   \to u^2_\tau \quad &&\text{in } L^\nu_{\text{loc}}(\R^N) \text{ for all } \nu\in [1,2^*),\\
			&u_n(\cdot+k_n)-u_\tau(\cdot+k_n)   \to u^2_\tau \quad &&\text{a.e. in } \R^N. 
		\end{aligned}
	\end{equation}
	Since $u_n(\cdot+z_n) - u_\tau(\cdot+z_n) \rightharpoonup 0$ weakly in $H^1(\R^N)$ for any bounded sequence $\{z_n\} \subset \Z^N$, 
	we see that $|k_n|\to+\infty$. Moreover, it is clear that
	$u^2_\tau>0$ satisfies
	\begin{equation*}
		-\Delta u_\tau^2+V(x) u_\tau^2+\lambda_\tau u_\tau^2=\tau f_1(u_\tau^2)+ f_2(u_\tau^2), 
		\quad \Vab{u^2_\tau}_2^2 \leq \mu \leq \overline{\mu}_2. 
	\end{equation*}
	Put $u^1_\tau \coloneq u_\tau$. By \cref{lemb1},
	\[
	\tilde m_{\tau,0}(u^i_\tau)\geq1,
	\qquad i=1,2.
	\]
	From the density of $C^\infty_0(\R^N)$ in $H^1(\R^N)$,
	there exist $\phi_i,\eta_i\in C_0^\infty(\mathbb R^N)$
	such that
	\[
	\int_{\mathbb R^N}u^i_\tau\phi_i\,dx=0,
	\qquad
	Q_{\lambda_\tau,\tau,u^i_\tau}(\phi_i)<0,
	\qquad
	\int_{\mathbb R^N}u^i_\tau\eta_i\,dx\neq0.
	\]
	Set $k_n^1 \coloneq 0$, $k_n^2 \coloneq k_n$, and define
	\[
	\xi_n^i(x) \coloneq 
	\phi_i(x-k_n^i)
	- \frac{\ab( u_n, \phi_i(\cdot - k^i_n) )_2}{ (u_n , \eta_i(\cdot - k^i_n))_2 }
	\eta_i(x-k_n^i).
	\]
	Then, for $n \in \mathbb{N}^+$ large enough, $\xi_n^i$ is well-defined and $\xi_n^i\in T_{u_n}S_\mu$.
	Moreover, by \eqref{eqrvn} and \eqref{eqrvnkn}, $\xi_n^i(\cdot + k^i_n) \to \phi_i$ in $H^1(\R^N)$ and 
	$\{\xi^i_n\}$ is bounded in  $L^\infty (\R^N)$. 
	Let $K_i \coloneq \supp \phi_i\cup \supp \eta_i$ and $W_{n,i} \coloneq\operatorname{span}\{\xi_n^i\}$ 
	for $i=1,2$ and $n \in \mathbb{N}$ large enough.
	For $j=1,2$, \eqref{A2} and H\"older's inequality with $u_n \to u_\tau$ in $L^\nu_{\text{loc}}(\R^N)$ with $\nu \in [1,2^*)$ lead to 
	\begin{equation*}
		\begin{aligned}
			\max_{ w \in W_{n,1}, \norm{ w}_{H^1}^2 = 1 }
			\int_{\R^N} \abs{f_j'(u_n)-f_j'(u_\tau)}w^2\, dx
			\leq \ab( \max_{w \in W_{n,1}, \norm{ w}_{H^1}^2 = 1  }\norm{w}_\infty^2 )\int_{K_1}\abs{f_j'(u_n)-f_j'(u_\tau)}\, dx \to 0,
		\end{aligned}
	\end{equation*}
	and, similarly,
	\begin{equation*}
		\begin{aligned}
			&\max_{ w \in W_{n,2}, \norm{w}_{H^1}^2 = 1 }
			\int_{\R^N} \abs{f_j'(u_n(x))-f_j'(u_\tau^2(x-k_n))}\abs{w(x)}^2\, dx \to 0. 
		\end{aligned}
	\end{equation*}
	Since $|k_n|\to+\infty$,  $\supp \xi^1_n \cap \supp \xi^2_n = \emptyset$ and $\dim W_n = 2$ for $n \in \mathbb{N}^+$ large enough, 
	where $W_n \coloneq W_{n,1} \oplus W_{n,2} \subset T_{u_n} S_\mu$. 
	Hence, there exists
	$\kappa>0$ such that for $n \in \mathbb{N}^+$ large enough,
	\[
	Q_{\lambda_n,\tau,u_n}(\xi)< - \kappa
	\qquad
	\text{for every }
	\xi\in W_n \text{ and }\norm{\xi}_{H^1}^2=1,
	\]
	that is,
	\[
	Q_{\lambda_n,\tau,u_n}(\xi)
	\leq-\kappa\|\xi\|_{H^1}^2
	\qquad
	\text{for every } \xi \in W_n \subset T_{u_n} S_\mu.
	\]
	This contradicts
	$\tilde m_{\tau,\zeta_n}(u_n)\leq1$ and $\{r_n\}$ vanishes.

	Since $\Vab{r_n}_q \to 0$ holds for every $q \in (2,2^*)$ via Lions' lemma, 
	by  
	\[
	o(1) = E_{V,\tau}'(u_n) + \lambda_\tau (u_n , \cdot)_2, \quad 0 = E_{V,\tau}'(u_\tau) + \lambda_\tau (u_\tau,\cdot)_2,
	\]
	as in the proof of \cref{lemurhoconverge}, we obtain 
	\[
	|r_n|_V^2+\lambda_\tau\|r_n\|_2^2\to0  \quad \text{as }n \to +\infty.
	\]
	From $\lambda_\tau \geq 1$, \eqref{eqineqmorse-r} and  \cite[Remark 2.6]{CGJT}, it follows that 
\[
u_n\to u_\tau\quad\text{in }H^1(\mathbb R^N), \quad \tilde{m}_{\tau,0}(u_\tau)\leq 1, \quad \text{and}
\quad \|u_\tau\|_2^2=\mu.
\]
On the other hand, \cref{lemb1} gives $\tilde{m}_{\tau,0} (u_\tau) = 1$.

	Let $\tau_n\to1^-$, and let
	\(
	\{(u_{\tau_n},\lambda_{\tau_n})\}
	\subset S_\mu\times(0,+\infty)
	\)
	be the sequence of positive solutions constructed in the preceding argument. Thus, for each $n\in\mathbb N^+$,
	\[
	\begin{dcases}
		-\Delta u_{\tau_n}+V(x)u_{\tau_n}
		+\lambda_{\tau_n}u_{\tau_n}
		=\tau_n f_1(u_{\tau_n})+f_2(u_{\tau_n})
		\quad\text{in }\mathbb R^N,\\
		u_{\tau_n}>0\quad\text{in }\mathbb R^N,\\
		\displaystyle\int_{\mathbb R^N}|u_{\tau_n}|^2\,dx=\mu.
	\end{dcases}
	\]
	We show that $\{u_{\tau_n}\}$ is bounded in $H^1(\mathbb R^N)$
	by adapting the argument used in the proof of \cref{th1}.
	We argue by contradiction, and suppose that
	$|u_{\tau_n}|_V^2\to+\infty$. By \eqref{A1} and the definition of $f_1$ given in \eqref{eqdeff1f2}, for every $M>0$, there exists $C_M>0$ such that
	\[
	\abs{f_1(t)t} \leq C_M\abs{t}^2\quad \text{and}\quad \abs{F_1(t)} \leq C_M\abs{t}^2\quad \text{for every $t \in \R$ with  $\abs{t}\leq M$}.
	\]
	By repeating the argument in the proofs of \cref{th1,th2}, and choosing $\delta>0$ and
	$M_\delta>0$ as in \eqref{eqF1f1}, we obtain
	\[
	\begin{aligned}
		\lambda_{\tau_n}\mu =&\int_{\R^N}\tau_n f_1(u_{\tau_n})u_{\tau_n}-2\tau_n F_1(u_{\tau_n})+f_2(u_{\tau_n})u_{\tau_n}-2F_2(u_{\tau_n})\,dx-2c_{\tau_n}\\
		\geq& -C\mu-2c_{\tau_n}+\int_{ \set{u_{\tau_n} \leq M_\delta }}\tau_n f_1(u_{\tau_n})u_{\tau_n}-2\tau_n F_1(u_{\tau_n})\,dx\\
		&+\int_{ \set{u_{\tau_n} >M_\delta }}\tau_n f_1(u_{\tau_n})u_{\tau_n}-2\tau_n F_1(u_{\tau_n})\,dx \\
		\geq&  -C\mu-2c_{\tau_n}+\ab(p\frac{a_0-\delta}{a_0+\delta}-2)\int_{ \set{u_{\tau_n} >M_\delta }}\tau_n F_1(u_{\tau_n})\,dx\\
		\geq& -C\mu -2c_{\tau_n}+ \ab(p\frac{a_0-\delta}{a_0+\delta}-2)\int_{\R^N }\tau_n F_1(u_{\tau_n})\,dx \to +\infty \quad \text{as }n \to +\infty.
	\end{aligned}
	\] 
Therefore, if $\{u_{\tau_n}\}$ were unbounded, 
then the above estimate would imply $\lambda_{\tau_n}\to+\infty$, and the blow-up argument in the proof of \cref{propbu} leads to, after passing to a subsequence,
\[
\lambda_{\tau_n}^{\frac N2-\frac{2}{p-2}}\|u_{\tau_n}\|_2^2
\longrightarrow a_0^{-\frac{2}{p-2}}k\|U\|_2^2<+\infty
\]
for some $k\in\{1,2\}$. This is impossible because $p>2+4/N$ and $\frac N2-\frac{2}{p-2}>0$, 
whereas $\|u_{\tau_n}\|_2^2=\mu>0$ is fixed. Hence $\{u_{\tau_n}\}$ is bounded in $H^1(\R^N)$.
Now, we can exploit the argument in the beginning of the proof, 
	and find $\{z_n\} \subset \Z^N$ and $(u_\infty,\lambda_\infty) \in H^1(\R^N) \times \R$ such that, 
	up to a subsequence, $u_{\tau_n} (\cdot + z_n) \to u_\infty$ strongly in $H^1(\R^N)$ and $\lambda_{\tau_n} \to \lambda_\infty$.
Moreover, the preceding construction gives $\lambda_{\tau_n}\geq1$ for every $n$, and hence $\lambda_\infty\geq1>0$. Since $u_{\tau_n}(\cdot+z_n)\to u_\infty$ strongly in $H^1(\R^N)$ and $\tau_n\to1$,  \eqref{A2} implies
\[
E_V(u_\infty)=\lim_{n\to\infty}E_{V,\tau_n}(u_{\tau_n})=\lim_{n\to\infty}c_{\tau_n}\geq c_1>0.
\]
In addition, \cref{lemb1} and $\tilde{m}_{\tau_n,0} (u_{\tau_n}) = 1$ lead to $\tilde{m}_{1,0} (u_\infty) = 1$. 
Finally, from the argument in \cref{lemmps}, we also observe $c_1=c_1(\mu) \to +\infty$ as $\mu \to 0^+$, 
which completes the proof. 
\end{proof}

\subsection{Proof of \cref{T:whole} \ref{T:whole-iii}}

This subsection is devoted to the proof of \cref{T:whole} \ref{T:whole-iii}. 
Under \eqref{A6}, since $tf(t) > 0$ holds for every $t \in \R \setminus \{0\}$, 
it is not necessary to introduce $f_1$ and $f_2$ in \eqref{eqdeff1f2}, and we define 
\[
E_{V,\tau} (u) \coloneq \frac{1}{2} \vab{u}_V^2 - \tau \int_{\R^N} F(u) \, dx.
\]
It is easily seen that \cref{lemmps} holds and the monotonicity trick in \cite{BCJS} can be applied for $\{E_{V,\tau}\}_{\tau \in [1/2,1]}$. 
In what follows, the same notations for $m_{\lambda,\tau} (u)$, $Q_{\lambda,\tau,u} (\varphi)$ and $\tilde{m}_{\tau,\theta}(u)$ with $\Omega = \R^N$ 
are used with this $E_{V,\tau}$. 
To proceed, the following two lemmas are necessary. 

We also observe that all the estimates in the blow-up argument that rely on
$ |f_1(t)|\leq C|t|^{p-1} $
remain valid upon replacing the former inequality by the following one: for every \(\varepsilon>0\), there exists \(C_\varepsilon>0\) such that 
\[ |f(t)|\leq \varepsilon |t|+C_\varepsilon |t|^{p-1} \quad \text{for all $t \in \R$}. \]

\begin{lemma}\label{l:est-locmin}
	Assume $N \geq 1$, \eqref{A1}--\eqref{A4} and \eqref{A6}. 
	There exist $\tilde{M}>0$ and $\mu_3'>0$ such that 
	\[
	\lambda + \Vab{u}_{H^1} \leq \tilde{M}
	\]
	holds for every $\tau \in [1/2,1]$ and every $(u,\lambda)\in H^1(\R^N)\times [0,+\infty)$ satisfies $\Vab{u}_2^2 \leq \mu_3'$ and 
	\begin{equation}\label{eq:loc-min}
		-\Delta u + (V+\lambda)u =   \tau f (u) \quad \text{in} \ \R^N ,\quad  u>0 \quad \text{in} \ \R^N, \quad 
		\tilde{m}_{\tau,0} (u) = 0.
	\end{equation}
	Moreover, 
for every $\varepsilon>0$, there exists $\mu_\varepsilon>0$ such that, for every $\mu\in(0,\mu_\varepsilon]$, $\tau\in[1/2,1]$, and every $(u,\lambda)\in S_\mu\times[0,+\infty)$ satisfying \eqref{eq:loc-min}, one has
\[
|E_{V,\tau}(u)|\le\varepsilon.
\]
Furthermore, for every $M>0$, there exists $\mu_M>0$ such that, for every $\mu\in(0,\mu_M]$, $\tau\in[1/2,1]$, and every $(u,\lambda)\in S_\mu\times[0,+\infty)$ satisfying
\[
-\Delta u+(V+\lambda)u=\tau f(u) \quad \text{in} \ \R^N,\qquad u>0\quad\text{in }\R^N, \quad \tilde m_{\tau,0}(u)\ge1, \quad E_{V,\tau}(u)\ge1,
\]
 one has $\lambda> M$.
\end{lemma}

\begin{proof}
We prove the first assertion by contradiction and assume that there exist $\{\tau_n\} \subset [1/2,1]$ 
and $\{(u_n,\lambda_n)\} \subset H^1(\R^N) \times [0,+\infty)$ such that $\tau_n \to \tau_\infty \in [1/2,1]$, 
$\lambda_n + \Vab{\nabla u_n}_2 \to \infty$ and 
\[
-\Delta u_n + (V+\lambda_n) u_n = \tau_n f(u_n) \quad \text{in} \ \R^N, \quad u_n > 0 \quad \text{in} \ \R^N, 
\quad \Vab{u_n}_2 \to 0, \quad \tilde{m}_{\tau_n,0} (u_n) = 0. 
\]
We claim $\Vab{u_n}_\infty \to +\infty$. When $\lambda_n \to +\infty$, then the argument in the proof of \cref{lem61} is valid and 
$\Vab{u_n}_\infty \to +\infty$ holds. 
On the other hand, when $\Vab{\nabla u_n}_2 \to +\infty$ and $\{\Vab{u_n}_\infty\}_n$ is bounded, 
we have $\Vab{u_n}_q \to 0$ for every $q \in [2,+\infty)$ due to $\Vab{u_n}_2 \to 0$ 
and 
\[
\Vab{\nabla u_n}_2^2 + o(1) \leq \Vab{\nabla u_n}_2^2 + \int_{\R^N} V u_n^2 \, dx + \lambda_n \Vab{u_n}_2^2 
= \int_{\R^N} \tau_n f(u_n) u_n  \, dx \to 0,
\]
which contradicts $\Vab{\nabla u_n}_2 \to +\infty$. Hence, $\Vab{u_n}_\infty \to +\infty$ holds in either case.

Let $P_n$ be a maximum point of $u_n$. 
The blow-up argument in the proof of \cref{lem62} and \cite{GS} lead to $\lambda_n \to + \infty$ 
and hence up to a subsequence, 
\[
U_n(x) \coloneq (\tau_\infty a_0)^{\frac{1}{p-2}} \lambda_n^{-\frac{1}{p-2}} u_n \ab( \lambda_n^{-\frac{1}{2}} x + P_n )
\to U \quad \text{in $C^1_{\rm loc} (\R^N)$},
\]
where $U$ is a positive solution to 
\[
-\Delta U + U = \vab{U}^{p-2} U \quad \text{in} \ \R^N, \quad U>0 \quad \text{in} \ \R^N, \quad  \quad U(0) = \max_{\R^N} U,  \quad U(x) \to 0 \quad \text{as $\vab{x} \to +\infty$}. 
\]
Then as in the proof of \cref{lemb1}, we may find $\psi_n \in T_{u_n} S_{\mu_n}$ where $\mu_n \coloneq \Vab{u_n}_2^2$ such that 
$Q_{\lambda_n,\tau_n,u_n} (\psi_n) < 0$. 
Since $\lambda_n = - E_{V,\tau}'(u_n)u_n / \Vab{u_n}_2^2$, this contradicts $\tilde{m}_{\tau_n,0} (u_n) = 0$.

Next, we prove the second assertion. 
If the assertion fails to hold, then there exist 
$\{\mu_n\}$ and $\{(u_n,\lambda_n,\tau_n)\} \subset S_{\mu_n} \times [0,+\infty) \times [1/2,1]$ such that 
$\mu_n \to 0^+$, $(u_n,\lambda_n,\tau_n)$ satisfies \eqref{eq:loc-min}, $\tau_n \to \tau_\infty$ and 
$\inf_{n \in \N}\vab{E_{V,\tau_n}(u_n)} \geq \delta_0 >0 $ for some $\delta_0$. 
Then $\Vab{u_n}_\infty \to +\infty$ holds due to $\vab{E_{V,\tau_n}(u_n)} \geq \delta_0 > 0$. 
Hence, the rest of the argument is the same as in the previous case and we obtain a contradiction to $\tilde{m}_{\tau_n,0} (u_n) = 0$.

The third assertion may be verified similarly. 
If there exists $\{\mu_n\}$ and $\{(u_n,\lambda_n,\tau_n)\} \subset S_{\mu_n} \times [0,+\infty) \times [1/2,1]$ such that 
$\mu_n \to 0^+$, $\{\lambda_n\}$ is bounded, $\tau_n \to \tau_\infty$ and $(u_n,\lambda_n, \tau_n)$ satisfies 
\[
-\Delta u_n + (V+\lambda_n) u_n = \tau_n f(u_n)  \ \text{in} \ \R^N, \quad u_n > 0 \ \text{in} \ \R^N, \quad 
\tilde{m}_{\tau_n,0} (u_n) \geq 1, \quad E_{V,\tau_n}(u_n) \geq 1. 
\]
Then $\Vab{u_n}_\infty \to +\infty$ holds as $n \to +\infty$ due to $E_{V,\tau_n}(u_n) \geq 1$ and $\Vab{u_n}_2 \to 0$. 
Hence, the blow-up argument and \cite{GS} yield $\lambda_n \to + \infty$, which is a contradiction. 
\end{proof}

\begin{lemma}\label{l:Morse-index}
	Suppose $N \geq 1$, \eqref{A1}--\eqref{A4} and \eqref{A6}, and let $\mu > 0$ and $\tau \in [1/2,1]$. 
	Then every solution to 
	\[
	-\Delta u + (V+\lambda) u = \tau f(u) \quad \text{in} \ \R^N, \quad 
	u > 0 \quad \text{in} \ \R^N, \quad \Vab{u}_2^2 = \mu
	\]
	satisfies $m_{\lambda,\tau} (u) \geq 1$. 
\end{lemma}

\begin{proof}
By elliptic regularity, $u \in C(\R^N)$, and $\Vab{u}_2^2 = \mu>0$ with the strong maximum principle yields $u>0$ in $\R^N$. 
Therefore, the range of $u$ contains a nonempty open interval $I\subset(0,+\infty)$. Since \(f(t)/t\) is strictly increasing on \((0,+\infty)\) by \eqref{A6}, the mean value theorem yields some \(t_0\in I\) such that
\(f'(t_0)t_0>f(t_0)\).
By the continuity of \(t\mapsto f'(t)t-f(t)\), there exists a nonempty open interval \(J\subset I\) such that $f'(t)t>f(t)$ for every $t \in J$.
Moreover, \eqref{A6} implies \(f'(t)t\geq f(t)\) for every \(t>0\), and $\{x\in\R^N:u(x)\in J\}$ is a nonempty open set and hence has positive measure.
Consequently,
\[
\begin{aligned}
	Q_{\lambda,\tau,u} (u)
	= 
	\int_{\R^N} \vab{\nabla u}^2 + (V+\lambda) u^2 - \tau f'(u) u^2 \, dx 
	<
	\int_{\R^N} \vab{\nabla u}^2 + (V+\lambda) u^2 - \tau f(u) u \, dx = 0,
\end{aligned}
\]
which implies $m_{\lambda,\tau} (u) \geq 1$. 
\end{proof}

\begin{proof}[Proof of \cref{T:whole} \ref{T:whole-iii}]
Suppose $N \geq 1$, \eqref{A1}--\eqref{A4} and \eqref{A6}.
By \cref{lemmps,l:est-locmin}, there exists $\bar{\mu}_3>0$ such that if $\mu \in (0,\bar{\mu}_3]$, then 
\begin{enumerate}[label=(\Roman*), ref=\Roman*]
	\item \label{I}
	$c_\tau \geq c_1>2$;
	\item \label{II}
	$\vab{E_{V,\tau}(u)} \leq 1$ and $\lambda \leq \tilde{M}$ hold for any $(u,\lambda,\tau) \in S_\mu \times [0,+\infty) \times [1/2,1]$ fulfilling \eqref{eq:loc-min}; 
	\item \label{III}
	$\lambda > \tilde{M}$ for every $(u,\lambda,\tau) \in S_\mu \times [0,+\infty) \times [1/2,1]$ with 
	$E_{V,\tau} (u) \geq 1$, $-\Delta u + (V+\lambda) u = \tau f(u)$ in $\R^N$, 
	$u>0$ in $\R^N$ and $\tilde{m}_{\tau,0} (u) \geq 1$. 
\end{enumerate}
Let $\mu \in (0,\bar{\mu}_3]$. 
Since \cite{BCJS} may be applied for $E_{V,\tau}$ and $c_\tau = c_\tau(\mu)$, 
as in the proof of \cref{th51,th52}, for almost every $\tau \in [1/2,1]$, 
there exist sequences $\{u_n\}\subset S_\mu$ and $\zeta_n\to0^+$ such that, as $n \to +\infty$,
\begin{enumerate}[label=\rm(\roman*)]
	\item $E_{V,\tau}(u_n) \to c_\tau>0$;
	\item $\norm{E_{V,\tau}'(u_n) }_{ T_{u_n}^* S_{\mu} }\to 0$;
	\item $u_n \geq 0$ and $\{u_n\}$ is bounded in $H^1(\R^N)$;
	\item $\widetilde m_{\tau,\zeta_n}(u_n)\leq1$;
\end{enumerate}
a sequence $\lambda_n \coloneq -\frac{1}{\mu}E'_{V,\tau}(u_n)u_n$ is bounded
and, if  there exists a subspace $W_n \subset H^1(\R^N)$ such that 
\begin{equation}\label{eqineqmorse-iii}
	E''_{V,\tau}(u_n)[w,w]+\lambda_n(w,w)_2<-\zeta_n \norm{w}^2_{H^1} \quad \text{for all }w \in W_n\setminus\{0\},
\end{equation}
then $\dim W_n \leq 2$ holds. Up to a subsequence, we may assume $\lambda_n \to \lambda_\tau \geq 0$ 
as in the proof of \cref{T:whole} \ref{T:whole-ii}. 
Since $E_{V,\tau} (u_n) \to c_\tau > 0$, the sequence $\{u_n\}$ does not vanish. 
From the periodicity of $V$, without loss of generality, we may suppose 
\begin{equation*}
	\begin{aligned}
		u_n &\rightharpoonup  u_\tau^1 \not \equiv  0  \quad \text{in } H^1(\R^N),\\
		u_n &\to u_\tau^1 \quad \text{in } L^\nu_{\text{loc}}(\R^N) \text{ for all } \nu\in [1,2^*),\\
		u_n &\to u_\tau^1 \quad \text{a.e. in } \R^N.
	\end{aligned}
\end{equation*}
It is clear that $u_{\tau}^1 \in H^{1}(\R^N)$ satisfies 
\begin{equation}\label{eq:limeq}
	-\Delta u+V(x)u+\lambda_\tau u =\tau f(u) \quad \text{in} \ \R^N, \quad u > 0 \quad \text{in} \ \R^N, 
	\quad \Vab{u}_2^2 \leq \mu
\end{equation}
and by \cref{l:Morse-index}, $m_{\lambda_\tau,\tau} (u^1_\tau) \geq 1$ holds.

Let $r_n^1 \coloneq u_n-u_\tau^1$. Brezis--Lieb lemma implies
\[
\mu = \Vab{u_n}_2^2 = \Vab{u_\tau^1}_2^2 + \Vab{r_n^1}_2^2 + o(1), \quad  c_\tau + o(1) = E_{V,\tau}(u_\tau^1) + E_{V,\tau}(r_n^1).
\]
Our aim is to prove that 
$\{r_n^1\}_n$ vanishes. 
If $\{r_n^1\}_n$ does not vanish, 
then there exists $\{y^1_n\}_n \subset \Z^N$ such that 
$\{r_n^1 (\cdot+y^1_n)\}_n$ converges to $u_\tau^2 \not \equiv 0$ weakly in $H^1(\R^N)$, 
in $L^\nu_{\rm loc}(\R^N)$ for all $\nu \in [1,2^*)$ and a.e. $\R^N$, where 
$u^2_\tau$ is a solution to \eqref{eq:limeq} with $m_{\lambda_\tau,\tau} (u^2_\tau) \geq 1$ due to \cref{l:Morse-index}. 
It is easily seen that $\vab{y_n^1} \to \infty$ thanks to $r^1_n \rightharpoonup 0$ weakly in $H^1(\R^N)$. 
Writing $r^2_n(x) \coloneq r_n^1(x) - u^2_\tau( x - y^1_n ) $,  we have 
\[
\mu = \Vab{u_\tau^1}_2^2 + \Vab{u_\tau^2}_2^2 + \Vab{r^2_n}_2^2 + o(1), \quad c_\tau + o(1) = E_{V,\tau} (u_\tau^1) + E_{V,\tau} (u^2_\tau) 
+ E_{V,\tau}(r^2_n) + o(1).
\]

Next, we claim that $\{r_n^2\}$ vanishes. Otherwise, there exists $\{y^2_n\}_n \subset \Z^N$ such that 
$\{r^2_n(\cdot + y^2_n)\}_n$ converges to $u_\tau^3 \not \equiv 0$ weakly in $H^1(\R^N)$, 
in $L^\nu_{\rm loc} (\R^N)$ for all $\nu \in [1,2^*)$ and a.e. $\R^N$, where 
$u^3_\tau$ is also a solution to \eqref{eq:limeq} with $m_{\lambda_\tau,\tau} (u^3_\tau) \geq 1$. 
Set $z_n^1\coloneq0$, $z_n^2\coloneq y_n^1$ and $z_n^3\coloneq y_n^2$. From $m_{\lambda_\tau,\tau}(u^i_\tau)\geq1$ for $1\leq i\leq3$ and 
$\vab{z_n^i-z_n^j}\to\infty$ for $1\leq i<j\leq3$, 
as in the proof of \cref{T:whole} \ref{T:whole-ii}, 
for sufficiently large $n$, we may find $\varepsilon_0>0$ and $W_n \subset H^1(\R^N)$ such that 
$\dim W_n = 3$ and 
\[
Q_{\lambda_n,\tau,u_n} (w) = E_{V,\tau}''(u_n) [w,w] + \lambda_n (w,w)_2 \leq - \varepsilon_0 \Vab{w}_{H^1}^2 
\quad \text{for all $w \in W_n \setminus \{0\}$}.
\]
However, this contradicts \eqref{eqineqmorse-iii}. Hence $\{r_n^2\}_n$ vanishes.

Since $\{r_n^2\}_n$ vanishes, $\Vab{r_n^2}_\nu \to 0$ holds for all $\nu \in (2,2^*)$. From 
\[
o(1) = \Vab{ -\Delta u_n + V u_n + \lambda_\tau u_n - \tau f(u_n) }_{H^{-1}}
\]
and 
\[
 -\Delta u^i_\tau (\cdot - z^i_n) + V u^i_\tau (\cdot - z^i_n) + \lambda_\tau u^i_\tau(\cdot-z_n^i) = \tau f(u^i_\tau (\cdot - z^i_n)) \quad \text{in} \ \R^N \quad (i=1,2),
\]
it follows that 
\[
\vab{r_n^2}_{V}^2 \leq \vab{r_n^2}_V^2 + \lambda_\tau \Vab{r_n^2}_2^2 
= \tau \int_{\R^N} g_n(x) r_n^2 \, dx + o(1) = o(1),
\]
where 
\[
g_n(x) \coloneq f(u_n(x)) - f(u^1_\tau(x-z^1_n)) - f(u^2_\tau(x-z^2_n)).
\]
Hence, 
\begin{equation}\label{e:ctau-u1u2}
	c_\tau =  E_{V,\tau} (u_\tau^1) + E_{V,\tau} (u_\tau^2). 
\end{equation}
Up to relabeling $u_\tau^1$ and $u_\tau^2$, there are three cases to consider:
\begin{enumerate}[label=\textbf{Case \arabic*}:]
	\item 
	$\tilde{m}_{\tau,0} (u^i_\tau) = 0$ for $i=1,2$;
	\item 
	$\tilde{m}_{\tau,0} (u^1_\tau) = 0 < 1 \leq \tilde{m}_{\tau,0} (u^2_\tau)$; 
	\item 
	$\tilde{m}_{\tau,0} (u^i_\tau) \geq 1$ for $i=1,2$. 
\end{enumerate}
In view of \eqref{I} and \eqref{II}, Case 1 does not occur by \eqref{e:ctau-u1u2}, $c_\tau > 2$ and $\vab{E_{V,\tau} (u^i_\tau)} \leq 1$ for $i=1,2$. 
In Case 2, \eqref{II} gives $\lambda_\tau\le\tilde M$ because $\tilde m_{\tau,0}(u_\tau^1)=0$. Moreover, by \eqref{e:ctau-u1u2}, $c_\tau>2$ and $|E_{V,\tau}(u_\tau^1)|\le1$, we have
\[
E_{V,\tau}(u_\tau^2)=c_\tau-E_{V,\tau}(u_\tau^1)>1.
\]
Therefore, property \eqref{III} together with $\tilde m_{\tau,0}(u_\tau^2)\ge1$ yields $\lambda_\tau>\tilde M$, a contradiction.
If Case 3 occurred, then the proof of \cref{T:whole} \ref{T:whole-ii} works to find $\kappa > 0$ and a subspace $\tilde{W}_n \subset T_{u_n} S_\mu$ such that $\dim \tilde{W}_n = 2$ and $Q_{\lambda_n,\tau,u_n} (\xi) \leq - \kappa \Vab{\xi}_{H^1}^2$ for all $\xi \in \tilde{W}_n$. 
However, this contradicts $\tilde{m}_{\tau,\zeta_n} (u_n) \leq 1$. 
Thus, Case 3 does not happen. 
Since none of the three cases can occur, $\{r_n^1\}$ vanishes.

By $\Vab{r^1_n}_\nu \to 0$ for all $\nu \in (2,2^*)$, we have 
\[
\vab{r^1_n}_V^2 + \lambda_\tau \Vab{r^1_n}_2^2 
= 
\tau \int_{\R^N} \bab{f(u_n) - f(u_\tau^1) } r_n^1 \, dx = o(1), \quad 
E_{V,\tau} (u^1_\tau) = \lim_{n \to +\infty} E_{V,\tau} (u_n) = c_\tau > 2. 
\]
Thus, \cref{l:est-locmin} leads to $\tilde{m}_{\tau,0} (u^1_\tau) \geq 1$ and $\lambda_\tau > \tilde{M} > 0$. 
Hence, $\Vab{r^1_n}_2 \to 0$, $\Vab{u_n - u^1_\tau}_{H^1} \to 0$ and $(u_\tau,\lambda_\tau) \in S_\mu \times (0,+\infty)$ 
is a solution to 
\[
-\Delta u + Vu + \lambda u = \tau f(u) \quad \text{in} \ \R^N, \quad \Vab{u}_2^2 = \mu
\]
and satisfies $E_{V,\tau} (u_\tau^1) = c_\tau$ and $\tilde{m}_{\tau,0} (u_\tau^1) = 1$ due to $\tilde{m}_{\tau,\zeta_n}(u_n) \leq 1$. 

Let $\tau_n\to1^-$, and let
\(
\{(u_{\tau_n},\lambda_{\tau_n})\}_n
\subset S_\mu\times(0,+\infty)
\)
be the sequence of positive solutions constructed in the preceding argument. Thus, for each $n\in\mathbb N^+$,
\[
\left\{\begin{aligned}
	&-\Delta u_{\tau_n}+V(x)u_{\tau_n}
	+\lambda_{\tau_n}u_{\tau_n}
	=\tau_n f(u_{\tau_n})\quad\text{in }\mathbb R^N, \quad u_{\tau_n} > 0 \quad \text{in} \ \R^N, 
	\\
	& \Vab{u_{\tau_n}}_2^2 = \mu, \quad E_{V,\tau_n} (u_{\tau_n}) = c_{\tau_n}, \quad \tilde{m}_{\tau_n,0} (u_{\tau_n}) = 1.
\end{aligned}\right.
\]
As in the proof of \cref{T:whole} \ref{T:whole-ii}, we verify that $\{u_{\tau_n}\}_n$ is bounded in $H^1(\mathbb R^N)$.
Since $2<c_1\leq c_{\tau_{n+1}}\leq c_{\tau_n}\leq c_{1/2}$ and $\tilde{m}_{\tau_n,0}(u_{\tau_n})=1$,
we may exploit the argument in the beginning of the proof,
and find $\{z_n\} \subset \Z^N$ and $(u_\infty,\lambda_\infty) \in H^1(\R^N) \times \R$ such that, 
up to a subsequence, $u_{\tau_n} (\cdot + z_n) \to u_\infty$ strongly in $H^1(\R^N)$ and $\lambda_{\tau_n} \to \lambda_\infty$.
Furthermore, $\lambda_{\tau_n}>\tilde M$ and $E_{V,\tau_n}(u_{\tau_n})=c_{\tau_n}\geq c_1>2$ for every $n$. Passing to the limit and using $u_{\tau_n}\to u$ strongly in $H^1(\R^N)$ together with $\tau_n\to1$, we obtain
\[
\lambda_\infty\geq \tilde M>0,
\qquad
E_V(u_\infty)=\lim_{n\to\infty}E_{V,\tau_n}(u_{\tau_n})=\lim_{n\to\infty}c_{\tau_n}\geq c_1>2.
\]
Hence, $(u_\infty,\lambda_\infty) \in S_\mu \times (0,+\infty)$ is the desired solution. 
As in the proof of \cref{T:whole} \ref{T:whole-ii} with \cref{l:est-locmin}, $E_V(u_\infty) > 2$ and $\tilde{m}_{\tau_n,0} (u_{\tau_n}) = 1$, 
we see that $c_1=c_1(\mu) \to + \infty$ as $\mu \to 0^+$ and $\tilde{m}_{1,0} (u_\infty) = 1$. 
\end{proof}

\appendix
\section{Proof of \cref{lemnelambda+0}}
\label{app:lemnelambda+0}
 In this appendix, we give the proof of \cref{lemnelambda+0}. 
\begin{proof}[Proof of \cref{lemnelambda+0}]
We argue by contradiction and suppose that $u$ is a solution to \eqref{eqneH0} with $u \not \equiv 0$. 
By a translation, we may suppose
\[
H= \Set{ x = (x_1,\dots, x_N) \in \R^N : x_N > 0 }. 
\]
Elliptic regularity yields $u \in W^{2,r}_{\rm loc} (\R^N)$ for every $r < \infty$, $u \in C^2(H)$ 
and the strong maximum principle (\cite[Theorem 8.19]{GT}) leads to $u>0$ in $\R^N$.

When $N=1$, by $-u'' = \chi_{(0,\infty)} u^{p-1} \geq 0$, we have $\lim_{x \to \pm \infty} u'(x) = u_{\pm \infty} \in \R \cup \set{\pm \infty}$. 
The fact that $u$ is bounded implies $0 = u_{-\infty} = u_{+\infty}$. From $-u'' \geq 0$, it follows that 
$u'' \equiv 0$, $u' \equiv 0$ and hence $u \equiv 0$ from the differential equation. 
However, this is a contradiction.

On the other hand, when $N=2$, 
a function $u_\varepsilon \coloneq \rho_\varepsilon * u \in C^2(\R^2)$ where $\{\rho_\varepsilon\}$ is a mollifier satisfies 
$-\Delta u_\varepsilon \geq 0$ in $\R^2$. 
Hence Liouville's theorem (\cite[Theorem 29 in Chapter 2]{PW}) yields $u_\varepsilon \equiv \text{const.}$ and 
letting $\varepsilon \to 0$ gives $u \equiv \text{const.}$. 
From \eqref{eqneH0}, we deduce $u \equiv 0$ in $\R^2$.

We now consider the case where $N \geq 3$. The proof is a minor modification of \cite[\S8.4]{QuSo19}. 
If we can prove that $u(x) = u(x_N)$ for each $x \in H$, 
then the situation is reduced to the case $N=1$ and a contradiction occurs.

In what follows, we aim to prove that $u(x) = u(x_N)$ for each $x \in H$. 
To this end, we first prove 

\smallskip 

\noindent
\textbf{Step 1:} \textsl{For any positive bounded solution $u$ to \eqref{eqneH0}, its Kelvin transform $v$ satisfies 
\[
v(z) = v(-z_1,z_2,\dots,z_N) = v(z_1,-z_2,z_3,\dots,z_N) = \dots = v(z_1,\dots,z_{N-2}, - z_{N-1} , z_N)
\]
for each $z \in \R^N \setminus \set{ 0} $ where 
\[
v(z) \coloneq |z|^{2-N} u \ab( \frac{z}{|z|^2} ).
\]
}

\smallskip 

First, notice that for each $z \in \R^N \setminus \set{0}$, $\chi_H( z / |z|^2 ) = \chi_H(z)$ and 
$v$ solves 
\[
-\Delta v(z) = \frac{1}{|z|^{N+2}} \ab(-\Delta u ) \ab( \frac{z}{|z|^2} ) 
= |z|^{-N-2} \chi_H(z) |z|^{(p-1)(N-2)} v^{p-1} (z) 
= |z|^{-\gamma} \chi_H(z) v^{p-1} (z),
\]
where $\gamma \coloneq 2N - p(N-2) > 0$. 
Moreover, $v \in C(\R^N \setminus \set{0})$, $v>0$ in $\R^N \setminus \set{0}$ and there exists $C_0>0$ such that 
\begin{equation}\label{decay-v}
	v(z) \leq C_0 |z|^{2-N} \quad \text{for all $|z| \geq 1$}.
\end{equation}
In addition, by $-\Delta v \geq 0$ in $\R^N \setminus \set{0}$ and $N \geq 3$, 
as in \cite[Lemma 4.4]{QuSo19}, we may verify
\[
v \in W^{2,r}_{\rm loc} (\R^N \setminus \{0\}) \quad \text{for any $r<\infty$}, 
\quad v \in L^1_{\rm loc} (\R^N), \quad 
-\Delta v \geq 0 \quad \text{in} \ \mathcal{D}'(\R^N).
\]
Hence, the maximum principle in \cite[Proposition 52.3(ii)]{QuSo19} gives 
\begin{equation}\label{min-v-R}
	v(z) \geq \min_{\partial B_R} v \eqcolon \eta(R) > 0 \quad \text{for each $z \in B_R \setminus \set{0} $ and $R>0$}. 
\end{equation}

To prove Step 1 for $z_1$, let $\lambda \leq 0$ and write
\[
\begin{aligned}
	z^\lambda &\coloneq \left( 2\lambda - z_1, z_2,\dots,z_N \right), \quad 0^\lambda \coloneq (2\lambda,0,\dots,0), 
	\\
	\Sigma(\lambda) & \coloneq \Set{ z \in \R^N : z_1 < \lambda }, \quad \Sigma'(\lambda) \coloneq \Sigma(\lambda) \setminus \set{0^\lambda},
	\\
	w(z;\lambda) &\coloneq v(z^\lambda) - v(z) \quad \text{for $z$ $\in \Sigma'(\lambda)$. }
\end{aligned}
\]
Since $\chi_H(z^\lambda) = \chi_H(z)$ and $|z^\lambda| \leq |z|$ for any $z \in \Sigma(\lambda)$, 
by $\gamma > 0$, on $\Sigma'(\lambda)$, $w$ satisfies 
\begin{equation}\label{diff-ineq-w}
	\begin{aligned}
		-\Delta w &= |z^\lambda|^{-\gamma} \chi_H (z^\lambda) v^{p-1} (z^\lambda) - |z|^{-\gamma} \chi_H(z) v^{p-1} (z)
		\\
		&= 
		\ab[ |z^\lambda|^{-\gamma} - |z|^{-\gamma} ] \chi_H (z^\lambda)  v^{p-1}(z^\lambda) 
		+ |z|^{-\gamma} \chi_H(z) \ab[ v^{p-1} (z^\lambda) - v^{p-1}(z) ]
		\\
		&\geq 
		|z|^{-\gamma} \chi_H (z) (p-1) \int_0^1 \left[ v(z) + \theta w(z;\lambda) \right]^{p-2} d \theta w(z;\lambda) 
		\eqcolon |z|^{-\gamma} \chi_H(z) \tilde{c}(z;\lambda) w(z;\lambda).
	\end{aligned}
\end{equation}
Let $\alpha \coloneq (N-2)/2$ and introduce 
\[
\tilde{w} (z;\lambda) \coloneq |z|^\alpha w (z;\lambda).
\]
From 
\[
\begin{aligned}
	&-\Delta \tilde{w} = -\Delta (|z|^\alpha) w - 2 \nabla (|z|^\alpha) \cdot \nabla w + |z|^\alpha (-\Delta w),
	\\
	&-\Delta (|z|^\alpha) = - \left[ \alpha (\alpha -1) + (N-1) \alpha \right] |z|^{\alpha-2} = - \alpha (N+\alpha-2) |z|^{\alpha-2} = - \frac{3}{4} (N-2)^2 |z|^{\alpha-2},
	\\
	&
	\nabla (|z|^\alpha) \cdot \nabla w
	= \alpha |z|^{\alpha-2} z \cdot \nabla w 
	= \alpha |z|^{\alpha-2} z \cdot[ -\alpha |z|^{-\alpha-2} z \tilde{w} + |z|^{-\alpha} \nabla \tilde{w} ] 
	= \alpha |z|^{-2} \left[  -\alpha \tilde{w} + z\cdot\nabla \tilde{w} \right],
\end{aligned}
\]
it follows that 
\[
-\Delta \tilde{w} \geq - \frac{1}{4} (N-2)^2 \frac{\tilde{w}}{|z|^2} - \frac{N-2}{|z|^2} z \cdot \nabla \tilde{w} 
+ |z|^{-\gamma} \chi_H(z) \tilde{c} (z;\lambda) \tilde{w} \quad \text{in} \ \Sigma'(\lambda). 
\]
By writing 
\[
c(z;\lambda) \coloneq \frac{1}{4} (N-2)^2 |z|^{-2}  - |z|^{-\gamma} \chi_H(z) \tilde{c} (z;\lambda),
\]
we have 
\begin{equation}\label{diff-ineq-tildew}
	-\Delta \tilde{w} + \frac{N-2}{|z|^2} z \cdot \nabla \tilde{w} + c(z;\lambda) \tilde{w} \geq 0 \quad \text{in} \ \Sigma'(\lambda). 
\end{equation}

Next, we prove $\tilde{w} \geq 0$ in $\Sigma'(\lambda)$ provided $\lambda \ll -1$. 
Indeed, if this were false, then there exists $\{\lambda_n\}$ such that 
\[
\lambda_n \to - \infty, \quad \inf_{ \Sigma'(\lambda_n) } \tilde{w} (\cdot ; \lambda_n) < 0.
\]
If $|z-0^{\lambda_n}| < 1$, then $|z^{\lambda_n}| < 1 < 2 |\lambda_n| - 1 \leq |z_1| \leq |z|$. 
Therefore, \eqref{decay-v} and \eqref{min-v-R} yield 
\[
w(z;\lambda_n) = v( z^{\lambda_n} ) - v(z) \geq \eta(1) - C_0 |z|^{2-N} > 0 
\quad \text{if $|z - 0^{\lambda_n} | < 1$}. 
\]
Moreover, \eqref{decay-v} implies 
\[
\tilde{w} (z;\lambda_n) = |z|^{(N-2)/2} \ab[ v(z^{\lambda_n}) - v(z) ] \to 0 \quad \text{as $|z| \to \infty$ with $z \in \Sigma(\lambda_n)$}. 
\]
By $\tilde{w}(\cdot;\lambda_n) = 0$ on $\partial \Sigma(\lambda_n)$, 
there exists $q_n \in \Sigma'(\lambda_n)$ such that 
\[
\tilde{w} (q_n ; \lambda_n) = \inf_{ \Sigma'(\lambda_n) } \tilde{w} < 0, \quad |q_n - 0^{\lambda_n} | \geq 1.
\]
Notice that $|q_n| \to \infty$ holds due to $q_n \in \Sigma'(\lambda_n)$, which gives $v(q_n) \to 0$.

Next, we claim that $v(q_n^{\lambda_n}) \to 0$ holds. In fact, if $v(q_{n_k}^{\lambda_{n_k}}) \geq c_0 > 0$ holds for some $c_0$ and $\{n_k\}$, 
then 
\[
w (q_{n_k}) \geq c_0 - v(q_{n_k}) \to c_0 > 0 \quad \text{as $k \to \infty$},
\]
which contradicts $\tilde{w} ( q_{n_k} ; \lambda_{n_k} ) < 0$. 
Thus, $v(q_n^{\lambda_n}) \to 0$ and $|q_n^{\lambda_n}| \to \infty$ as $n \to \infty$.

If $|q_n|/|q_n^{\lambda_n}| \to \infty$, then since the definition of $v$ leads to 
$|z|^{N-2} v(z) \to u(0) > 0$ as $|z| \to \infty$, we have 
\[
0 > |q_n|^{N-2} w(q_n) = \ab( \frac{|q_n|}{|q_n^{\lambda_n}|} )^{N-2} |q_n^{\lambda_n} |^{N-2} v(q_n^{\lambda_n}) - |q_n|^{N-2} v(q_n) \to +
\infty,
\]
which is a contradiction. Thus, $|q_n| \leq C_1 |q_n^{\lambda_n}|$ holds for some $C_1>0$ and all $n$. 
In particular, \eqref{decay-v} implies 
\[
v(q_n^{\lambda_n}) + v(q_n) \leq C_2 |q_n|^{2-N}   \quad \text{for any $n$}. 
\]
This together with the definition of $\tilde{c}$ yields 
\[
\chi_{H} (q_n)\tilde{c}(q_n ; \lambda_n) \leq\chi_H(q_n) (p-1) C_3 |q_n|^{ - (N-2) (p-2) }.
\]
Thus, if $n$ is sufficiently large, then 
\[
\begin{aligned}
	c(q_n;\lambda_n) 
	&=  \frac{(N-2)^2}{4} |q_n|^{-2} - |q_n|^{-2N+p(N-2)} \chi_H(q_n) \tilde{c} (q_n;\lambda_n)
	\\
	&\geq \frac{(N-2)^2}{4} |q_n|^{-2} - (p-1)C_3 |q_n|^{-4}> 0.
\end{aligned}
\]
Recalling the strong maximum principle (\cite[Theorem 8.19]{GT}) and \eqref{diff-ineq-tildew}, we may find $r_0>0$ such that 
$\tilde{w} (\cdot ; \lambda_n) \equiv \tilde{w} (q_n ; \lambda_n) < 0$ in $B_{r_0}(q_n)$, 
however, this contradicts \eqref{diff-ineq-tildew} since we may assume $c(\cdot ; \lambda_n) > 0$ in $B_{r_0}(q_n)$. 
Thus, $\tilde{w} (\cdot ; \lambda) \geq 0$ in $\Sigma'(\lambda)$ provided $\lambda \ll -1$.

Let us consider 
\[
\bar{\mu} \coloneq \sup \Set{ \mu \leq 0 :
	\text{$\tilde{w}(\cdot ; \lambda) \geq 0$ in $\Sigma'(\lambda)$ for all $\lambda \in (-\infty,\mu]$} 
	} \in (-\infty,0].
\]
We shall prove $\bar{\mu} = 0$. If $\bar{\mu} < 0$, then 
it is easily seen from the definition of $\bar{\mu}$ and 
the continuity of $v$ that $\tilde{w} (\cdot; \bar{\mu} ) \geq 0$ in $\Sigma'(\bar{\mu})$. 
Moreover, there exists $\{\lambda_n\}$ such that 
\[
\lambda_n > \bar{\mu}, \quad \lambda_n \to \bar\mu, \quad \inf_{ \Sigma'(\lambda_n) } \tilde{w} (\cdot ; \lambda_n) < 0.
\]
If $\tilde{w} (\cdot ;\bar \mu) \not \equiv 0$, then \eqref{diff-ineq-tildew} yields 
\begin{equation}\label{diff-ineq-tildew-2}
	-\Delta \tilde{w} + \frac{N-2}{|z|^2} z \cdot \nabla \tilde{w} + \left(  c(z ; \bar \mu) \right)_+ \tilde{w} \geq  0 \quad \text{in} \ \Sigma'(\bar \mu).
\end{equation}
Hence, the strong maximum principle gives 
$\tilde{w} (\cdot ; \bar \mu) > 0$ in $\Sigma'(\bar \mu)$. 
From \eqref{diff-ineq-w}, 
it follows that $\tilde{c}(z;\bar \mu) > 0$ in $U \coloneq B_{ \abs{\bar \mu}/2 }( 0^{\bar \mu} ) \setminus\set{0^{\bar \mu}} \subset \Sigma' (\bar{\mu})  $ and 
$-\Delta w \geq 0$ in $U$. 
Again as in \cite[Lemma 4.4]{QuSo19}, we may prove $-\Delta w \geq 0$ in $\mathcal{D}'(U)$ and the maximum principle yields 
\[
w(\cdot ; \bar \mu) \geq \min_{ |z- 0^{\bar \mu} | = |\bar \mu|/2  } w(\cdot ; \bar \mu) \coloneq \eta_0 > 0 \quad \text{in} \ U. 
\]
If $z \in B_{ \abs{\bar \mu}/4 }(0^{\lambda_n} ) \setminus \set{ 0^{\lambda_n} } $, then by 
$\{ z-2(\lambda_n-\bar \mu) e_1 \}^{\bar \mu} = 2 \lambda_n e_1 - z = z^{\lambda_n} \neq 0$, 
$z - 2(\lambda_n - \bar \mu) e_1 \in U$ and 
\[
\begin{aligned}
	w(z;\lambda_n) 
	&= w \left( z - 2 (\lambda_n - \bar \mu) e_1 ; \bar \mu \right) + v \left( z - 2(\lambda_n - \bar \mu) e_1  \right) - v(z)
	\\
	&\geq \eta_0 + v \left( z - 2(\lambda_n - \bar \mu) e_1  \right) - v(z),
\end{aligned}
\]
for all sufficiently large $n$, the continuity of $v$ gives $w(z;\lambda_n) \geq 0$. 
Thus, there exists $q_n \in \Sigma'(\lambda_n)$ such that 
\begin{equation}\label{prop-qn}
	\abs{ q_n - 0^{\lambda_n} } \geq \frac{|\bar \mu|}{4}, \quad 
	\tilde{w} (q_n;\lambda_n) = \inf_{ \Sigma'(\lambda_n) } \tilde{w} (\cdot ; \lambda_n) < 0, \quad 
	\nabla \tilde{w} (q_n;\lambda_n) = 0.
\end{equation}

We claim that $\{q_n\}$ is bounded. 
If $|q_n| \to \infty$, then by $\lambda_n \to \bar \mu < 0$, we see that 
$|q_n^{\lambda_n}| \to \infty$ and $|q_n|/|q_n^{\lambda_n} | \to 1$ as $n \to \infty$. 
Since $v(q_n) + v(q_n^{\lambda_n}) \leq 2C_0 |q_n|^{2-N}$ holds due to \eqref{decay-v}, 
it follows that 
\[
c(q_n;\lambda_n) \geq \frac{1}{4} (N-2)^2 |q_n|^{-2} - \chi_H(q_n) (p-1) C |q_n|^{-4} >0. 
\]
Recalling \eqref{diff-ineq-tildew} and that $q_n$ is a global minimum point of $\tilde{w}(\cdot;\lambda_n)$, 
we observe from the strong maximum principle that for some $r_n>0$, 
\[
c(z;\lambda_n) > 0, \quad 
0 > \tilde{w} (z; \lambda_n) \equiv \tilde{w} (q_n;\lambda_n) \quad \text{for any $z \in B(q_n,r_n)$}.
\]
This contradicts \eqref{diff-ineq-tildew} and 
$(q_n)$ is bounded.

Let us assume $q_n \to \bar q \in \overline{\Sigma (\bar \mu)}$. 
By $|q_n - 0^{\lambda_n} | \geq |\bar \mu| /4$, we have $\bar q \neq 0^{\bar \mu}$. 
Since $\tilde w (\cdot ; \bar \mu) \geq 0$ in $\Sigma' (\bar \mu)$, \eqref{prop-qn} yields 
\[
\tilde{w} (\bar q ; \bar \mu) = 0, \quad \nabla \tilde{w} (\bar q ; \bar \mu) = 0.
\]
When $\bar q \in \Sigma'(\bar \mu)$, 
the strong maximum principle with \eqref{diff-ineq-tildew-2} leads to $\tilde{w} (\cdot ; \bar \mu) \equiv 0$ in $\Sigma'(\bar \mu)$, 
which contradicts $\tilde{w} (\cdot ; \bar \mu) > 0$ in $\Sigma'(\bar \mu)$. 
When $\bar q \in \partial \Sigma '(\bar \mu)$, 
since $\tilde{w} \in W^{2,r}_{\rm loc} ( \partial \Sigma'(\bar \mu) + B_\varepsilon(0) )$ for any $r < \infty$, 
the proof for Hopf's lemma still works (\cite[Proposition 52.1]{QuSo19} or \cite[Lemma 3.4]{GT}), 
and $\partial \tilde{w} (\bar q ; \bar \mu) / \partial x_1 < 0$, which contradicts $\nabla \tilde{w} (\bar q ; \bar \mu) = 0$. 
Thus, if $\bar \mu < 0$, then $\tilde{w} (\cdot ;\bar \mu) \equiv 0$ in $\Sigma'(\bar \mu)$, and 
$v$ is symmetric with respect to $\set{z_1 = \bar \mu}$.

In this case, since $v \in C^2(H)$, we have 
\[
-\frac{\Delta v(z)}{v^{{p-1}}(z)} = |z|^{-2N+p(N-2)}  \quad \text{for all $z \in H$}.
\]
The left-hand side is symmetric with respect to $\set{z_1 = \bar \mu}$, 
however, the right-hand side is not. 
Therefore, this is a contradiction and the case $\bar \mu < 0$ does not occur. 
Thus, $\bar \mu = 0$ holds.

From $\bar \mu = 0$, it follows that
\[
v(z) \leq v(z^0) = v(-z_1,z') \quad \text{for all $z \in \R^N$ with $z_1 < 0$}. 
\]
On the other hand, the Kelvin transform of $u(-x_1,x')$ is $v(-z_1,z')$ and 
the above argument can be applied to $v(-z_1,z')$ to obtain 
$v(-z_1,z') \leq v(z)$ for any $z \in \R^N$ with $z_1<0$. 
Hence, $v(-z_1,z') = v(z_1,z')$ holds for any $z \in \R^N$ with $z_1<0$.

For $i=2,3,\dots,N-1$, the above argument can be applied to 
\[
u(x_i, x_1,\dots, x_{i-1} , x_{i+1} , \dots, x_N )
\]
and the assertions in Step 1 hold.

\smallskip

\noindent
\textbf{Step 2:} \textsl{For each $x \in H$, $u(x) = u(x_N)$ holds.}

\smallskip

For $a \in \R$, set $u_a(x) \coloneq u(x-ae_1)$. The corresponding Kelvin transform is given by 
\[
v_a (z) = |z|^{2-N} u_a \ab( \frac{z}{|z|^2} ) 
= |z|^{2-N}  u \ab( \frac{z}{|z|^2} - a e_1 ) \quad \text{for each $z \in H$.}
\]
From Step 1 applied to $u_a$, it follows that 
\[
u \ab( \frac{z_1}{|z|^2} - a , \frac{z'}{|z|^2} ) = u \ab( - \frac{z_1}{|z|^2} - a, \frac{z'}{|z|^2}) 
\quad \text{for every $z \in H$ and $a \in \R$}. 
\]
Let $w' = (w_2,\dots, w_N) \in \R^{N-1}$ be arbitrary with $w_N >0$. We shall prove 
\[
u (z_1, w') = u(0,w') \quad \text{for all $z_1 \in \R$}. 
\]
To this end, put 
\[
I \coloneq \ab[ - \frac{1}{2|w'|} , \frac{1}{2|w'|} ], \quad 
t(s) \coloneq \frac{1}{2} \ab[ \frac{1}{|w'|} + \sqrt{ \frac{1}{|w'|^2} - 4s^2 } ] 
\quad \text{for $s \in I$}. 
\]
It is immediate to verify that 
\[
\frac{t(s)}{s^2 + (t(s))^2} = |w'| \quad \text{for any $s \in I$}.
\]
For $s \in I$, consider  
\[
z = \ab(  s, t(s) \frac{w'}{|w'|} ) \in H.
\]
Then 
\[
\frac{z}{|z|^2} = \ab( \frac{s}{s^2+(t(s))^2} , w' )
\]
and 
\begin{equation}\label{u-basic}
	u\left( F(s) - a , w' \right) = u \left( - F(s) - a , w' \right) = u \left( F(-s) - a , w' \right) 
	\quad \text{for each $s \in I$ and $a \in \R$},
\end{equation}
where $F(s) \coloneq s / (s^2 + (t(s))^2 )$. 
By $F(0) = 0$, $F(-s) = -F(s)$ and $\lim_{s \to 1/(2|w'|)} F(s) = |w'|$, 
it is easily seen that $[-|w'|,|w'|] \subset F(I)$ and 
for each $a \in [-|w'|,|w'|]$ there exists $s_a \in I$ such that $F(s_a) = a$. 
Hence, \eqref{u-basic} gives 
\[
u(0,w') = u\left( -2a, w' \right) \quad \text{for all $a \in [-|w'|,|w'|]$},
\]
which is equivalent to 
\begin{equation}\label{equiv-1st}
	u(0,w') = u(z_1,w') \quad \text{for every $z_1 \in \R$ with $|z_1| \leq 2|w'|$}. 
\end{equation}
Next, for each $|w'| \leq |a| \leq 2 |w'|$, there exists $s_{a,\pm} \in I$ such that 
$F(s_{a,+})  =a  - |w'|$ if $a >0$ and 
$F(s_{a,-}) = a + |w'|$ if $a < 0$. Thus, \eqref{u-basic} and \eqref{equiv-1st} lead to 
\[
\begin{aligned}
	&u(0,w') = u(-|w'|,w') = u( |w'| - 2a, w' ) & &\text{for all $a \in [|w'|,2|w'|]$},
	\\
	&u(0,w') = u(|w'|,w') = u( -2a - |w'| , w' ) & &\text{for all $a \in [-2|w'|,-|w'|]$}.
\end{aligned}
\]
Hence, 
\begin{equation*}
	u(0,w') = u(z_1,w') \quad \text{for all $z_1 \in [-3|w'|,3|w'|]$}. 
\end{equation*}
We repeat this procedure and assume that for some $k \in \N$, 
\[
u(0,w') = u(z_1,w') \quad \text{for any $z_1 \in [-k|w'|,k|w'|]$}. 
\]
For every $(k-1)|w'| \leq |a| \leq k |w'|$, choose $s_{a,\pm} \in I$ so that 
$F(s_{a,+}) = a - (k-1) |w'|$ if $a>0$ and $F(s_{a,-}) = a + (k-1) |w'|$ if $a<0$. 
Then \eqref{u-basic} and the assumption imply 
\[
\begin{aligned}
	&u(0,w') = u(- (k-1) |w'|,w') = u( (k-1)|w'| -2a , w' ) & &\text{for all $a \in [(k-1)|w'|,k|w'|]$},
	\\
	&u(0,w') = u( (k-1) |w'|,w') = u( -2a - (k-1) |w'|  , w' ) & &\text{for all $a \in [-k|w'|,-(k-1)|w'|]$},
\end{aligned}
\]
which gives
\[
u(0,w') = u(z_1,w') \quad \text{for each $z_1$ $\in [-(k+1) |w'|, (k+1) |w'|]$}.
\]
Thus, 
\[
u(0,w') = u(z_1,w') \quad \text{for all $z_1 \in \R$}. 
\]
Since $w' \in \R^{N-1}$ with $w_N>0$ is arbitrary, we have 
\[
u(z) = u(0,z') \quad \text{for each $z \in H$}. 
\]
Repeating this argument for other variables $z_2,\dots, z_{N-1}$, 
we obtain 
\[
u(z) = u(0,\dots, 0, z_N) \quad \text{for each $z \in H$}
\]
and we complete the proof. 
\end{proof}

\section{Extension to general bounded potentials}
\label{sec:generalV}
In this appendix, we treat \eqref{eqV} for general bounded potentials when $\Omega$ is bounded. 
Here we emphasize that $V$ is not necessarily periodic. 
Let
\[
E_V(u) \coloneq \frac{1}{2} \int_{\R^N} \vab{\nabla u}^2 + V(x) \vab{u}^2 \, dx - \int_{\R^N} \chi_\Omega(x) F(\vab{u}) \, dx.
\]
and consider the bottom of the spectrum of the operator $-\Delta + V(x)$: 
     \begin{equation*}
     \sigma_0 \coloneq 
     \inf \sigma (-\Delta+V)
     =
     \inf_{u \in H^1(\R^N)\backslash\{0\}}\frac{\int_{\R^N}\left(\abs{\nabla u}^2+ V(x)\abs{u}^2\right)\, dx}{\int_{\R^N}\abs{u}^2\, dx}.
     \end{equation*}
When $\sigma_0$ is an eigenvalue of $-\Delta + V$, 
then, by combining the argument in \cite{CaJe} with the blow-up analysis developed in \cref{sec:blow-up} and 
assuming $f(t) t > 0$ for each $t \in (0,+\infty)$ (which would play a crucial role in repeating the argument of \cite[Lemma 3.4]{CaJe}), 
it is natural to expect that there exists $\mu_0>0$ such that, for every $\mu\in(0,\mu_0)$, problem \eqref{eqV} admits two solutions $(u_1,\lambda_1)$ and $(u_2,\lambda_2)$ with $\lambda_1,\lambda_2>-\sigma_0$: $u_1$ is a local minimizer of $E_V$ constrained to $S_\mu$, while the other one $u_2$ is a positive solution at a mountain pass level of  $E_V|_{S_\mu}$. Moreover, $E_V(u_2)>E_V(u_1)$.

In what follows,
we focus on the complementary case where $\sigma_0$ may not be an eigenvalue of $-\Delta + V$. 
More precisely, we assume the following spectrum assumption: 
\begin{enumerate}[label=(A\arabic*),ref=A\arabic*,start=11]
  	\item \label{A9}  \( \sigma_0=\inf \sigma_{\text{ess}} (-\Delta+V)\).
\end{enumerate}
Notice that if \(\sigma_0\) is not an eigenvalue, then \eqref{A9} holds. 
Hence, if $V$ is periodic or $V \geq 0$ and $V(x) \to 0$ as $\vab{x} \to +\infty$, then we may verify \eqref{A9}. 
However, in general, the converse does not hold; \(\sigma_0\) may be a threshold eigenvalue even when \eqref{A9} holds. 
For related spectral theory, we refer to \cite{Bo20}.

Our main result in this appendix is the existence of solutions to \eqref{eqV} when $\mu>0$ is small: 

\begin{theorem}\label{th1'}
     Let $N \geq 1$, $V\in C(\R^N) \cap L^\infty(\R^N)$, 
     $\Omega\subset  \R^N$ be a (nonempty) bounded open set with smooth boundary $\partial \Omega$, \eqref{A1}--\eqref{A4} and \eqref{A9} hold. 
     Then, there exists $\mu_0>0$ such that, for every $\mu\in (0,\mu_0)$, problem \eqref{eqV} has a solution $(u,\lambda)\in H^1(\R^N) \times (-\sigma_0,+\infty)$.
\end{theorem}
Without loss of generality, we may suppose 
\begin{equation}\label{inf-spect-g}
	\sigma_0 =0.
\end{equation}
Since the proof of \cref{th1'} is very similar to that of \cref{th1}, we only highlight the main differences.
Let
\[
E_{V,\tau} (u) \coloneq \frac{1}{2} \int_{\R^N} \vab{\nabla u}^2 + V(x) u^2 \, dx 
- \int_{\R^N} \chi_\Omega(x) \ab[ \tau F_1(u) + F_2(u) ] dx, 
\quad \tau \in \ab[\frac{1}{2}, 1],
\]
where $f_1,f_2$ are defined as in \eqref{eqdeff1f2}.

We begin from the uniform mountain pass structure for $E_{V,\tau}|_{S_\mu}$.
\begin{lemma}\label{lemmps-g}
Let $N \geq 1$ and suppose \eqref{A1} and \eqref{A3}. 
Then, for every $\alpha_0>0$, there exists $\mu^*>0$ such that, for any $\mu \in (0, \mu^*)$, 
there exist $w_1, w_2 \in S_\mu$ such that $w_1,w_2\geq 0$ and 
\[
    c_\tau = c_\tau(\mu) \coloneq 
    \inf_{\gamma \in \Gamma} \max_{t \in [0,1]} E_{V,\tau}(\gamma(t)) \geq \frac{3}{8} \alpha_0  > \max\{E_{V,\tau}(w_1), E_{V,\tau}(w_2), 0\} 
    \quad \text{for all $\tau \in \ab[ \frac{1}{2}, 1 ]$}, 
\]
where
\begin{equation*}
  \Gamma \coloneq \Set{\gamma \in C([0,1], S_\mu) : \gamma(0) = w_1, \, \gamma(1) = w_2 }.
\end{equation*}
Moreover, $[1/2, 1] \ni \tau \mapsto c_\tau \in (0,\infty)$ is nonincreasing. 
\end{lemma}

\begin{proof}
We only find $w_1,w_2$ and prove the inequality for $c_\tau$. Set $B_\alpha \coloneq \set{u \in S_\mu: \norm{\nabla u}_2^2=\alpha}$, and
fix $\alpha_0>0$ and $\varepsilon_0>0$. 
By \eqref{eqsubsec121} and the standard Gagliardo--Nirenberg inequality \eqref{stdGN}, 
for any $u \in B_{\alpha_0}$ and $\tau \in [\frac{1}{2},1]$, we obtain 
\[
E_{V,\tau}(u) 
\geq 
\frac{1}{2}\alpha_0-\frac{1}{2}\norm{V}_\infty \norm{u}_2^2-\frac{1}{2}\varepsilon_0\norm{u}_2^2-C\norm{u}_p^p 
\geq 
\frac{1}{2}\alpha_0-\frac{1}{2}(\varepsilon_0+\norm{V}_\infty)\mu-C C_{p,N}\mu^{(1 - \beta_p)p/2}\alpha_0^{\beta_p p/2}. 
\]
Thus, there exists  $\mu^*\in(0,  \min \{ \alpha_0/(4\varepsilon_0+4\norm{V}_\infty) , \alpha_0 / 8 \} )$ small enough such that
\begin{equation}\label{lweonBa-g}
	E_{V,\tau}(u) \geq \frac{3\alpha_0}{8} \quad \text{for every $\mu\in (0,\mu^*)$, $u \in B_{\alpha_0}$ and $\tau \in \ab[\frac{1}{2},1]$}. 
\end{equation}

Fix any $\mu\in (0,\mu^*)$. 
By 
\eqref{eqsubsec121} and $0<\mu < \mu^* < \min \{ \alpha_0 / (4\varepsilon_0+4\norm{V}_\infty), \alpha_0/8\} $,
there exists $w_1 \in S_\mu$ satisfying $w_1\geq 0$ and $\norm{\nabla w_1}_2^2\in (0,\alpha_0/8)$ small enough such that 
for each $\tau \in [\frac{1}{2},1]$, 
\[
\begin{aligned}
    E_{V,\tau}(w_1) &\leq \frac{1}{2}\norm{\nabla w_1}_2^2+ \frac{\varepsilon_0+\norm{V}_\infty}{2}\|w_1\|_{2}^2 +CC_{p,N} \|w_1\|_{2}^{(1 - \beta_p)p} \norm{\nabla w_1}_{2}^{\beta_p p}\\
    &\leq\frac{3}{16}\alpha_0 +CC_{p,N} \mu^{(1 - \beta_p)p/2} \norm{\nabla w_1}_{2}^{\beta_p p} \leq \frac{\alpha_0}{4}.
\end{aligned}
\]
By \eqref{A1} and \eqref{A3}, we find $R_1>R_0>0$ such that 
\[
\abs{F_2(t)} \leq C\abs{t}^2 \quad \text{for all $t \in \R$}
\]
and
\[
F_1(t) \geq \frac{a_0}{4p}\abs{t}^p-C\abs{t}^2 \quad \text{for every $t \in \R$ with $\abs{t} \geq R_1$}. 
\]
Let $v \in C_0^\infty(\R^N)\cap S_{\mu}$ satisfy $v \geq 0$. 
For any $s>0$, set $v_s(x) \coloneq s^{N/2}v(s(x-x_0))$ for some  $x_0 \in \Omega$. 
Then, by $p>2+\frac{4}{N}$ and $\tau \in [\frac{1}{2},1]$, 
for $w_2 \coloneq v_{s_0}$ with $s_0>0$ sufficiently large, 
we have $\supp w_2 \subset \Omega$, $\norm{\nabla w_2}^2_2> \alpha_0$ and
\[
\begin{aligned}
    E_{V,\tau}(w_2)&\leq \frac{1}{2}\norm{\nabla w_2}_2^2+\frac{1}{2}\norm{V}_\infty\norm{w_2}_2^2+C\norm{w_2}_2^2-\frac{a_0 }{8p}\norm{w_2}_p^p\\
    &=\frac{s_0^2}{2}\norm{\nabla v}_2^2+\frac{1}{2}\norm{V}_\infty\mu+C\mu-\frac{a_0s_0^{Np/2-N}}{8p}\norm{v}_p^p<0.
\end{aligned}
\]
Thus, by $\norm{\nabla w_1}^2_2\in (0,\alpha_0/8)$ and $\norm{\nabla w_2}^2_2> \alpha_0$, for all $\tau \in [\frac{1}{2},1]$, 
\eqref{lweonBa-g} gives
\[
c_\tau=\inf_{\gamma\in\Gamma}\max_{t\in[0,1]}E_{V,\tau}(\gamma(t)) \geq \frac{3\alpha_0}{8} > \frac{\alpha_0}{4} \geq\max\{E_{V,\tau}(w_{1}),E_{V,\tau}(w_{2})\}.
\]
This completes the proof. 
\end{proof}
\begin{proof}[Proof of \cref{th1'}]
By repeating the argument in the proofs of \cref{th51,th52}, for almost every $\tau \in [1/2,1]$, there exist sequences $\{u_n\}\subset S_\mu$ and $\zeta_n\to0^+$ such that, as $n \to +\infty$,
\begin{enumerate}[label=\rm(\roman*)]
    \item $E_{V,\tau}(u_n) \to c_\tau>0$;
    \item \label{bps2'}$\norm{E_{V,\tau}'(u_n) }_{ T_{u_n}^* S_{\mu} }\to 0$;
    \item \label{bps3'} $u_n \geq 0$ and $\{u_n\}$ is bounded in $H^1(\R^N)$;
\end{enumerate}
a sequence $\lambda_n \coloneq -\frac{1}{\mu}E'_{V,\tau}(u_n)u_n$ is bounded
 and, if  there exists a subspace $W_n \subset H^1(\R^N)$ such that 
\begin{equation}\label{eqineqmorse-g}
    E''_{V,\tau}(u_n)[w,w]+\lambda_n(w,w)_2<-\zeta_n \norm{w}^2_{H^1} \quad \text{for all }w \in W_n\setminus\{0\},
\end{equation}
then $\dim W_n \leq 2$ holds. Up to a subsequence, we may assume $u_n \rightharpoonup u_\tau$ weakly in $H^1(\R^N)$,  $u_n \to u_\tau$ in $L^r(\Omega)$ for all $r \in [1,2^*)$, and $\lambda_n \to \lambda_\tau \in \R$ without loss of generality. 

It is clear that $u_{\tau}\in H^{1}(\R^N)$ solves 
\begin{equation*}
    -\Delta u_\tau+V(x)u_\tau+\lambda_\tau u_\tau=\tau \chi_\Omega f_1(u_\tau)+\chi_\Omega f_2(u_\tau).
\end{equation*}
By elliptic regularity we have $ u_\tau \in L^\infty (\R^N)$, hence, $\tau \chi_\Omega f_1'(u_\tau)+\chi_\Omega f_2'(u_\tau) \in L^\infty (\R^N)$. 
Since $\Omega$ is bounded, by \cite[Corollary 7.13]{Bo20}, we have 
\begin{equation}\label{eqsptf1f2}
\inf \sigma_\text{ess}\ab\{-\Delta +V+\lambda_\tau - \chi_\Omega \ab( \tau f_1'(u_\tau) + f_2'(u_\tau) ) \}=\inf \sigma_\text{ess}(-\Delta +V+ \lambda_\tau)=\lambda_\tau.
\end{equation}

We now claim that $\lambda_\tau  \geq 0$. Indeed, if $\lambda_\tau < 0$ holds, then by \eqref{eqsptf1f2}, there exists a subspace $W\subset C_0^\infty(\R^N)$ with $\dim W=3$ such that 
\begin{equation*}
    E''_{V,\tau}(u_\tau)[w,w]+\lambda_\tau(w,w)_2< \frac{\lambda_\tau}{2}\norm{w}^2_{2} \quad \text{for all }w \in W\setminus\{0\}.
\end{equation*}
Since $W$ is of finite dimension, every norm on $W$ is equivalent and there exists $C>0$ such that 
\begin{equation*}
    E''_{V,\tau}(u_\tau)[w,w]+\lambda_\tau(w,w)_2<C \frac{\lambda_\tau}{2}\norm{w}^2_{H^1} \quad \text{for all }w \in W\setminus\{0\}.
\end{equation*}
For $i=1,2$, \eqref{A2} and H\"older's inequality with $u_n \to u_\tau$ in $L^r(\Omega)$ with $r \in [1,2^*)$ lead to 
\begin{equation*}
\begin{aligned}
	\max_{ w \in W, \norm{ w}_{H^1}^2 = 1 }
	\int_{\Omega} \abs{f_i'(u_n)-f_i'(u_\tau)}w^2\, dx
\leq \ab( \max_{w \in W, \norm{ w}_{H^1}^2 = 1  }\norm{w}_\infty^2 )\int_{\Omega}\abs{f_i'(u_n)-f_i'(u_\tau)}\, dx \to 0.
\end{aligned}
\end{equation*}
Therefore, for $n\in\mathbb{N}$ large enough,
\begin{equation*}
    E''_{V,\tau}(u_n)[w,w]+\lambda_n(w,w)_2<\frac{C\lambda_\tau}{4} \quad \text{for all }w \in W \text{ with } \norm{ w}_{H^1}^2 = 1, 
\end{equation*}
which contradicts \eqref{eqineqmorse-g}. Hence, $\lambda_\tau \geq 0$ holds. 

By repeating the argument in the proof of \cref{lemurhoconverge}, we obtain
\begin{equation}\label{eqconvergence-g}
\int_{\R^N}|\nabla (u_{n}-u_{\tau})|^{2}+ \left(V(x)+\lambda_{\tau}\right)|u_{n}-u_{\tau}|^{2}\,dx\to 0 \quad \text{as }n \to \infty.
\end{equation}
Hence, it follows from \eqref{inf-spect-g}, $u_n \rightharpoonup u_\tau$ weakly in $H^1(\R^N)$ and  $u_n \to u_\tau$ in $L^r(\Omega)$ for all $r \in [1,2^*)$ that 
\[
E_{V,\tau}(u_\tau)=\lim_{n\to+\infty}E_{V,\tau}(u_n)=c_\tau>0,
\]
which implies that $u_\tau\not \equiv 0$. 
According to \cref{rem:generalV}, there exists $\mu_0 \in (0,\mu^*)$ such that 
for each $\tau \in [1/2,1]$, there is no critical point $u$ of $E_{V,\tau}$ with $u>0$ in $\R^N$, 
$\Vab{u}_2^2 \leq \mu_0$ and $E_{V,\tau} (u) \geq \inf_{0<\mu \leq \mu_0}c_\tau( \mu )$. 
Let $\mu \in (0, \mu_0)$. Since Fatou's lemma leads to $\norm{u_\tau}_2^2\leq \mu < \mu_0$,  
we get $u_\tau>0$ in $\R^N$ and $\lambda_\tau>0$. Recalling \eqref{inf-spect-g}, 
for $0<\delta<\min\{1, \lambda_\tau/(2\norm{V}_\infty)\}$, we have
\[
\begin{aligned}
&\int_{\R^N}|\nabla (u_{n}-u_{\tau})|^{2}+ \left(V(x)+\lambda_{\tau}\right)|u_{n}-u_{\tau}|^{2}\,dx\\
&\qquad \geq \int_{\R^N}|\nabla (u_{n}-u_{\tau})|^{2}+ (1-\delta)V(x)|u_{n}-u_{\tau}|^{2}+\left(\delta V(x)+\lambda_{\tau}\right)|u_{n}-u_{\tau}|^{2}\,dx\\
&\qquad \geq \int_{\R^N}|\nabla (u_{n}-u_{\tau})|^{2}+ (1-\delta)V(x)|u_{n}-u_{\tau}|^{2}+\frac{\lambda_{\tau}}{2}|u_{n}-u_{\tau}|^{2}\,dx\\
&\qquad \geq \delta\int_{\R^N}|\nabla (u_{n}-u_{\tau})|^{2}+\frac{\lambda_{\tau}}{2}|u_{n}-u_{\tau}|^{2}\,dx.
\end{aligned}
\]
Combining this fact with \eqref{eqconvergence-g} implies $u_n \to u_\tau$ in $H^1(\R^N)$.

Finally, the blow-up analysis in \cref{sec:blow-up} remains valid for general potentials $V\in L^\infty(\R^N)$. Indeed, for the blow-up sequences arising there, with $\lambda_n\to+\infty$ and centers $P_n$, the rescaled potential term is
\[
\frac{1}{\lambda_n}V\bigl(P_n+\lambda_n^{-1/2}y\bigr),
\]
and, for every bounded set $K\subset\R^N$,
\[
\sup_{y\in K}\left|\frac{1}{\lambda_n}V\bigl(P_n+\lambda_n^{-1/2}y\bigr)\right|
\le \frac{\|V\|_\infty}{\lambda_n}\longrightarrow0.
\]
Thus the potential term disappears in the limiting equation, and the local blow-up and Morse-index arguments of \cref{sec:blow-up} are unchanged. Repeating the remaining arguments in the proofs of \cref{th1,th2}, we complete the proof of \cref{th1'}.
\end{proof}
\begin{remark} We note that \cite[Lemma 3.4]{CaJe} relies on \cite[(V2)]{CaJe}, 
	namely, $\sigma_0$ is an eigenvalue. While \cref{T:Liouv-musmall} holds for general potentials $V \in L^\infty (\R^N)$ (cf. \cref{rem:generalV}) 
	and the proof can be modified to treat the case $\Omega = \R^N$. 
	Indeed, as in the proof of \cref{T:Liouv-musmall}, let $\Omega = \R^N$ and $\{(u_n,\lambda_n)\}$ satisfy 
	\[
	E_{V,\tau_n}(u_n) \geq \varepsilon_0 > 0, \quad E_{V,\tau_n}'(u_n) = 0, \quad u_n > 0, \quad \Vab{u_n}_2 \to 0, 
	\quad \tau_n \to \tau_\infty. 
	\]
	Then we may prove $\max_{\R^N} u_n \to \infty$ and the proof of \cref{lem62} yields a contradiction thanks to \cite{GS}. 
	Thus, the proof of \cref{th1'} can also be adapted to the setting considered in \cite{CaJe} 
	by changing \cite[(A2)]{CaJe} to \eqref{A9}. 
	Therefore, our argument complements the discussion in \cite{CaJe} concerning radial potentials 
	and mountain pass type critical points.
\end{remark}
\medskip

\noindent
\textbf{Conflict of interest.}
On behalf of all authors, the corresponding author states that there is no conflict of interest.

\medskip 

\noindent
\textbf{Ethics approval.}
 Not applicable.

\medskip 

\noindent
\textbf{Data Availability Statements.}
Data sharing not applicable to this article as no datasets were generated or analysed during the current study.

\medskip 

\noindent
\textbf{AI Usage Statement.}

The authors acknowledge the use of Gemini 3.1 Pro for the availability of the argument in \cite{Eastham} at \cref{subsec:-rewr}, Gemini 3.6 Flash for the availability of the argument in \cite{EsLi82} at the proof of \cref{lem25}, ChatGPT 5.6 Sol for the availability of the argument in \cite{Bo20} at the proof of \cref{th1'}, and ChatGPT 6 Astra for assistance in developing and organizing the proofs of \cref{lemb1} and \cref{T:whole} \ref{T:whole-ii}. 
All the arguments were revised and checked independently by the authors. 
In the preparation of the manuscript, the AI is also exploited  to improve expositions. 
The authors take full responsibility for the content of the paper.

\medskip

\noindent
\textbf{Acknowledgements.}
N. Ikoma was supported by JSPS KAKENHI Grant Number JP24K06802. C. Ji was supported by the National Natural Science Foundation of China (No.12571117).

\end{document}